\documentclass[11pt,oneside,a4paper,reqno]{amsart}

\usepackage[utf8]{inputenc}
\usepackage[T1]{fontenc}
\usepackage{lmodern}
\usepackage{etoolbox}
\usepackage[protrusion=false]{microtype}
\usepackage[dvipsnames]{xcolor}

\usepackage{ifdraft}

\usepackage{geometry}
\usepackage[all]{xy}
\pdfpagewidth=\paperwidth
\pdfpageheight=\paperheight

\usepackage{amsmath,amsfonts,amssymb,mathtools,manfnt,amsthm,mathrsfs}

\usepackage{tikz}
\usepackage{tikz-cd}
\usetikzlibrary{arrows}
\usetikzlibrary{arrows.meta}
\usetikzlibrary{positioning}
\usetikzlibrary{calc}

\usepackage{enumitem}

\makeatletter
\def\namedlabel#1#2{\begingroup
	#2%
	\def\@currentlabel{#2}%
	\phantomsection\label{#1}\endgroup
}
\g@addto@macro\bfseries{\boldmath}
\makeatother

\usepackage[symbol]{footmisc}

\usepackage{datetime}
\usepackage{url}

\DeclareMathOperator{\id}{id}
\DeclareMathOperator{\Id}{Id}

\DeclareMathOperator{\spaan}{span}

\DeclareMathOperator{\supp}{supp}

\DeclareMathOperator{\Tr}{Tr}
\DeclareMathOperator{\interior}{int}

\DeclareMathOperator{\Aut}{Aut}
\DeclareMathOperator{\End}{End}

\DeclareMathOperator{\gr}{Gr}

\DeclareMathOperator{\KMS}{KMS}

\DeclareMathOperator{\diam}{diam}

\DeclareMathOperator{\Lip}{Lip}

\newcommand\restr[2]{{
		\left.\kern-\nulldelimiterspace 
		#1 
		\vphantom{|} 
		\right|_{#2} 
}}

\newcommand{\ol}[1]{\overline{#1}}

\newcommand{\shift}{T} 

\renewcommand{\AA}{\mathbb{A}}

\newcommand{\CC}{\mathbb{C}}

\newcommand{\NN}{\mathbb{N}}

\newcommand{\QQ}{\mathbb{Q}}
\newcommand{\RR}{\mathbb{R}}
\newcommand{\TT}{\mathbb{T}}

\newcommand{\ZZ}{\mathbb{Z}}

\newcommand{\Aa}{\mathcal{A}}

\newcommand{\Kk}{\mathcal{K}}
\newcommand{\Ll}{\mathcal{L}}
\newcommand{\Mm}{\mathcal{M}}
\newcommand{\Nn}{\mathcal{N}}
\newcommand{\Oo}{\mathcal{O}}
\newcommand{\Pp}{\mathcal{P}}

\newcommand{\Tt}{\mathcal{T}}

\newcommand{\Zz}{\mathcal{Z}}

\newcommand{\Rfr}{\mathfrak{R}}

\newcommand{\AAA}{\mathsf{A}}
\newcommand{\BBB}{\mathsf{B}}

\newtheorem{thm}{Theorem}[section]
\newtheorem{cor}[thm]{Corollary}
\newtheorem{lem}[thm]{Lemma}
\newtheorem{prop}[thm]{Proposition}
\newtheorem{mainthm}{Theorem}

\theoremstyle{definition}
\newtheorem{dfn}[thm]{Definition}

\newtheorem{ass}[thm]{Assumption}

\newtheorem{algo}[thm]{Algorithm}

\theoremstyle{remark}
\newtheorem{rmk}[thm]{Remark}

\newtheorem{example}[thm]{Example}

\AtBeginEnvironment{example}{%
	\pushQED{\qed}%
}
\AtEndEnvironment{example}{\popQED\endexample}

\usepackage[unicode=true,colorlinks=true,linkcolor=blue,citecolor=blue,urlcolor=blue,breaklinks=true]{hyperref}

\usepackage{orcidlink}

\begin{document}
  
\title[Intrinsic Hutchinson measures and KMS states]{Intrinsic Hutchinson measures and KMS states for iterated function systems with large overlaps}

\author[M. Mampusti]{Michael Mampusti}
\email{michael.mampusti@gmail.com}
\address[M. Mampusti]{Independent Researcher\\ Adelaide\\ SA\\  Australia}

\author[A. Mundey]{Alexander Mundey\orcidlink{0000-0002-7791-4383}}
\email{alex.mundey@adelaide.edu.au}
\address[A. Mundey]{Australian Institute of Machine Learning \\ Adelaide University\\ SA 5005 \\ Australia}

\author[N. Seaton]{Nicholas Seaton}
\email{nicholasseaton94@gmail.com}
\address[N. Seaton]{Independent Researcher\\ Wollongong \\ NSW\\  Australia}

\date{\today}

\begin{abstract}
	We study iterated function systems with large overlaps via the Kajiwara--Watatani $C^*$-algebra and its KMS states. Our main contributions are threefold. First, we introduce Markov partitions for iterated function systems and give an algorithm that constructs a finite Markov partition whenever one exists. Second, for contractive systems admitting a finite Markov partition, we construct an intrinsic analogue of the Hutchinson measure that is independent of any labelling or weighting of the maps. It is the unique Borel probability measure satisfying a self-similar integral equation dictated by the set-valued dynamics of the system, and it recovers the Hutchinson measure with uniform weights when the system satisfies the open set condition. Finally, we give a complete classification of KMS states for the gauge action on the Kajiwara--Watatani algebra of every such system. The critical inverse temperature is the exponential growth rate of the intrinsic orbits, which can be strictly smaller than the natural logarithm of the number of maps, and the unique KMS state at this temperature is given by our intrinsic Hutchinson measure. These systems include many with large overlaps not covered by the results of Izumi--Kajiwara--Watatani.
\end{abstract}

\maketitle


\section{Introduction}
From the Cantor set to the Sierpi\'nski gasket, many self-similar fractals are the attractors of iterated function systems. In this article we examine the dynamics of iterated function systems, in particular, those with \emph{large overlaps}, where the images of the maps of the system may intersect nontrivially. 
We analyse these dynamical systems through the lens of operator algebras and their $\KMS$ states. 

We present three main results.
\begin{enumerate}[label=(\arabic*)]
	\item We introduce Markov partitions for iterated function systems, show that there is always a unique minimal Markov partition,
	and give an iterative procedure to construct a finite Markov partition whenever one exists (Theorem~\ref{thm:intro_markov}).
	\item We construct a unique intrinsic analogue of the Hutchinson measure on the attractor of a contractive iterated function system, whenever it admits a finite Markov partition (Theorem~\ref{thm:intro_mu_infty}).
	\item We give a complete classification of the $\KMS$ states for the gauge action on the Kajiwara--Watatani algebra \cite{KW06} associated to contractive iterated function systems admitting finite Markov partitions, a class that contains many systems with large overlaps not covered by the results of Kajiwara--Watatani \cite{KW04pre} and Izumi--Kajiwara--Watatani \cite{IKW07} (Theorem~\ref{thm:intro_main}).
\end{enumerate}

Throughout we adopt the convention that an iterated function system is a pair $(\AA,\Gamma)$ consisting of a second-countable compact Hausdorff space $\AA$, called the \emph{attractor}, together with a finite collection $\Gamma$ of continuous injections on $\AA$ such that $\AA = \bigcup_{\gamma \in \Gamma} \gamma(\AA)$. We work exclusively with iterated function systems that are restricted to the attractor. For some of our main results we work in the setting where $\AA$ carries a compatible metric for which each $\gamma \in \Gamma$ is a contraction.

In the contractive setting there are two natural ways to regard the dynamics of an iterated function system, and what separates them is what they count. The first, which we call the \emph{extrinsic dynamics}, treats the maps in $\Gamma$ as labelled ingredients: writing $\Gamma = \{\gamma_1,\ldots,\gamma_N\}$, a word $ w_1\cdots w_n$ in the alphabet $\{1,\ldots,N\}$ labels the composition $\gamma_{w_1}\circ\cdots\circ\gamma_{w_n}$. Letting $\Sigma_N \coloneqq \{1,\ldots,N\}^{\NN}$ denote the \emph{code space}, the address map $\pi \colon \Sigma_N \to \AA$ identifies $\AA$ with a topological quotient of $\Sigma_N$ (cf.~\cite{Bar93a,Kig01}).  The dynamics then reside on $\Sigma_N$, carried by the full shift, and are transported to $\AA$ via $\pi$. The extrinsic perspective counts $|\Gamma|^k = N^k$ orbits of length $k$ issuing from each point of $\AA$.

The second perspective, which we call the \emph{intrinsic dynamics}, forgets the labels and regards $\Gamma$ as the multivalued map $x \mapsto \Gamma x \coloneqq \{\gamma(x) \mid \gamma \in \Gamma\}$. In this perspective, the orbit of a point $x_0 \in \AA$ is the collection of paths $(\ldots,x_2,x_1,x_0)$ in $\AA$ with $x_{i+1} = \gamma(x_{i})$ for some $\gamma \in \Gamma$. The number of paths of length $k$ issuing from $x_0$ falls below $|\Gamma|^k$ if there is some $i < k$ and distinct $\gamma,\gamma' \in \Gamma$ such that $\gamma(x_i) = \gamma'(x_i)$. The existence of such coincidences is a strictly stronger condition than the maps in $\Gamma$ having overlapping images. Indeed, the system $x \mapsto x/2$ and $x \mapsto (x+1)/2$ on $[0,1]$ has overlapping images, but no coincidences, while $x \mapsto x/2$ and $x \mapsto 1-x/2$ on $[0,1]$ has such a coincidence at $x = 1$. These coincidences also occur if a map in $\Gamma$ is duplicated; in this case the extrinsic perspective counts the system with two copies of the map as dynamically distinct from the original, while the intrinsic perspective does not.
In this article we are interested in the intrinsic perspective and what can be said about iterated function systems when one refuses to choose labels.

Hutchinson's celebrated theorem \cite[\S~3.1]{Hut81} belongs to the intrinsic perspective. It identifies the attractor $\AA$ as the unique compact fixed point of the \emph{Hutchinson operator} $K \mapsto \bigcup_{\gamma \in \Gamma} \gamma(K)$ defined on the collection of nonempty compact subsets of $\AA$. Hutchinson also shows that, once weights on $\Gamma$ are chosen, there is a distinguished self-similar measure on $\AA$. However, the choice of weights that define the measure is extrinsic and different labellings or weightings of $\Gamma$ produce different ``Hutchinson measures'' (cf.~\cite{Fal14,Kig01}). This raises the question: is there an intrinsic analogue of the Hutchinson measure that does not depend on labels and weights, and is unchanged by the addition of a redundant map to $\Gamma$?

Overlapping iterated function systems have been studied extensively in fractal geometry, where separation conditions play a central role. The most classical is Hutchinson's \emph{open set condition} \cite[\S~5.2]{Hut81}, which asks for a nonempty open set $U$ whose images $\gamma(U)$ under each $\gamma \in \Gamma$ are pairwise disjoint and contained in $U$ (see Definition~\ref{dfn:open_set_condition}). For systems of similitudes on $\RR^d$ this is weakened by the \emph{weak separation condition} of Lau--Ngai \cite{LN99} (see also Zerner \cite{Zer96}), which permits overlaps but constrains how close distinct compositions of the maps can come to one another. The overlaps that matter for the intrinsic dynamics are the exact coincidences $\gamma(x) = \gamma'(x)$. Such coincidences can occur on open sets for general continuous injections but not for distinct similitudes on $\RR^d$. In place of a separation constraint on overlaps, our main results assume a combinatorial finiteness condition on how these coincidences propagate under $\Gamma$, so large overlaps are not in themselves an obstruction.

A complementary, operator-algebraic approach to the dynamics of iterated function systems was initiated by Kajiwara and Watatani in \cite{KW06}. They associate a $C^*$-algebra to a contractive iterated function system $(\AA,\Gamma)$ that we call the  \emph{Kajiwara--Watatani algebra} and denote by $\Oo_{\Gamma}$. It is the Cuntz--Pimsner algebra of a $C^*$-correspondence over  $C(\AA)$, the $C^*$-algebra of $\CC$-valued continuous functions on $\AA$, and is constructed from the cograph $\gr(\Gamma) = \{(\gamma(x),x) : x \in \AA, \gamma \in \Gamma\}$ of the system. The cograph is independent of labels and remains unperturbed by the addition of redundant maps. 

The Kajiwara--Watatani algebra carries a natural circle action, called the \emph{gauge action}, that is, in a sense, dual to the intrinsic dynamics. The $\KMS$ states for this action encode thermodynamic information about $(\AA,\Gamma)$, and correspond to dynamically compatible measures on $\AA$ by the machinery of Laca--Neshveyev \cite{LN04}.
Izumi--Kajiwara--Watatani \cite{IKW07} classified the $\KMS$ states of the gauge action on $\Oo_{\Gamma}$ under a fairly strict hypothesis---the \emph{orbit escape condition} (cf.~\cite{deC23})---that ensures that the number of intrinsic orbits of length $k$ eventually grows like $|\Gamma|^k$.
In \cite[Theorem~6.6]{IKW07} they show that there is a unique \emph{infinite-type} $\KMS$ state at the critical inverse temperature $\beta = \ln|\Gamma|$, and that for $\beta > \ln|\Gamma|$ the $\KMS_\beta$ states are parametrised in terms of the common values $\gamma(y) = \gamma'(y)$ of distinct maps $\gamma,\gamma' \in \Gamma$. Our results are phrased in terms of the \emph{branch points} $B_\Gamma$ of the dynamics:
points in the cograph for which the \emph{branch index} $(x,y) \mapsto |\{\gamma \in \Gamma \colon \gamma(y) = x\}|$ is discontinuous. When the orbit escape condition is assumed, Izumi--Kajiwara--Watatani show that the measure associated to the infinite-type $\KMS$ state corresponds to the Hutchinson measure with uniform weights on $\Gamma$. This suggested to us that this measure is the right intrinsic analogue of Hutchinson's, and raised the question of whether it can be constructed without the orbit escape condition.

De Castro \cite{deC23} has since recast the analysis of Izumi--Kajiwara--Watatani using the Ruelle--Perron--Frobenius theorem of Fan and Lau \cite{FL99}, extending the analysis to generalised gauge actions on Kajiwara--Watatani algebras and identifying the critical $\KMS$ state as the eigenmeasure of that operator, but still under the same orbit escape condition. Ionescu and Watatani \cite{IW08} extended the Kajiwara--Watatani construction to Mauldin--Williams graphs \cite{MW88}, which generalise iterated function systems by attaching compact metric spaces to the vertices of a finite directed graph and contractions to its edges. The $\KMS$ states of Ionescu and Watatani's algebras were characterised in the first author's thesis \cite{Mam19}, under either an orbit escape condition or the open set condition. Although we do not consider Mauldin--Williams graphs in this article, we believe that many of our results can be extended to this setting. 

We develop the machinery needed to classify $\KMS$ states for the gauge action on Kajiwara--Watatani algebras in a more general setting, and along the way construct our intrinsic analogue of the Hutchinson measure. 
The central idea is that combinatorial control of the orbits can be obtained from a \emph{Markov partition} for $(\AA,\Gamma)$: a Borel partition of $\AA$ that respects the intrinsic dynamics (see Definition~\ref{dfn:markov_partition}). A finite Markov partition reduces the analysis of the \emph{Perron--Frobenius operator} used by \cite{IKW07} to that of a finite nonnegative matrix (see Lemma~\ref{lem:PF_operator_markov_partition}), bringing the problem into reach of the same Perron--Frobenius techniques that drive the $\KMS$-state classification for graph $C^*$-algebras \cite{aHLRS13,KW13,aHLRS15}. 
A finite matrix also appears in the \emph{finite type} condition of Ngai--Wang \cite{NW01} for self-similar sets with overlaps, but it plays a different role: its spectral radius computes Hausdorff dimension, whereas ours computes the exponential growth rate of the intrinsic orbits.

A principal contribution of this paper is a general algorithm for producing Markov partitions.
 We introduce the \emph{fundamental partition} $\Aa$ of $(\AA,\Gamma)$, the coarsest Borel partition of $\AA$ that respects the coincidence structure of $\Gamma$ (see Definition~\ref{dfn:fundamental_partition}), and define a refinement operator $\Theta$ on the collection of partitions refining $\Aa$ (see Definition~\ref{dfn:Theta}). We prove the following characterisation of Markov partitions.  

\begin{mainthm}[cf.~Lemma~\ref{lem:Theta} and Theorem~\ref{thm:finite_iff_Theta_const}]\label{thm:intro_markov}
Let $(\AA,\Gamma)$ be an iterated function system and let $\Aa$ be its fundamental partition. Then
\begin{enumerate}[label=\textup{(\arabic*)}]
\item a Borel partition $\Mm$ of $\AA$ refining $\Aa$ is Markov if and only if $\Theta(\Mm) = \Mm$;
\item the common refinement $\Mm_\infty$ of the partitions $\Theta^k(\Aa)$, $k \ge 0$, is the unique minimal Markov partition of $(\AA,\Gamma)$;
\item every Markov partition of $(\AA,\Gamma)$ refines $\Mm_\infty$; and
\item the system $(\AA,\Gamma)$ admits a finite Markov partition if and only if $\Mm_\infty$ is finite, in which case $\Mm_{\infty}$ is the limit of the eventually-constant sequence $(\Theta^k(\Aa))_{k \ge 0}$.
\end{enumerate}
\end{mainthm}

Theorem~\ref{thm:intro_markov} yields an algorithm (Algorithm~\ref{algo:Markov_partition}) to explicitly construct a finite Markov partition when one exists. In Subsection~\ref{subsec:examples} we give examples of iterated function systems that admit finite Markov partitions. We also exhibit a contractive iterated function system on $[0,1]$ that does not admit a finite Markov partition (Example~\ref{ex:ikw_no_markov}). 

Given a finite Markov partition $\Mm = \{M_i\}_{i \in I}$ for $(\AA,\Gamma)$, we may analyse the intrinsic dynamics of $(\AA,\Gamma)$ through a directed graph $E$ with vertex set $I$ and edges determined by the coincidence structure of $\Gamma$ (see Definition~\ref{dfn:graph_of_markov_partition}). We leverage classical Perron--Frobenius theory to construct our intrinsic Hutchinson measure $\mu_\infty$ as the unique Borel probability measure satisfying a self-similar integral equation dictated by the intrinsic dynamics of $\Gamma$.

\begin{mainthm}[Intrinsic Hutchinson measure; cf.~Theorem~\ref{thm:inf_type}]\label{thm:intro_mu_infty}
Let $(\AA,\Gamma)$ be a contractive iterated function system admitting a finite Markov partition.
There is a unique Borel probability measure $\mu_\infty$ on $\AA$ and $\rho > 0$ such that 
\begin{equation}
\int_\AA a(x) \, d\mu_\infty(x) = \frac{1}{\rho}\int_\AA \sum_{x \in \Gamma y} a(x) \, d\mu_\infty(y)
\end{equation}
for all $a \in C(\AA)$, where $\Gamma y = \{\gamma(y) \mid \gamma \in \Gamma\}$. Moreover, if $\Mm = \{M_i\}_{i \in I}$ is the unique minimal Markov partition and $E$ is its associated directed graph, then $E$ has a unique recurrent component $R$ and the adjacency matrices $\AAA_E$ of $E$ and $\AAA_R$ of $E|_R$ have common spectral radius $\rho$. 
The vector $(\mu_\infty(M_i))_{i \in R} \in \RR_+^R$ is the unique unit-$1$-norm right Perron--Frobenius eigenvector of $\AAA_R$ and $\mu_{\infty}(M_i) = 0$ for all $i \in I \setminus R$.
\end{mainthm}

The measure $\mu_\infty$ may be thought of as a conformal measure for the Perron--Frobenius operator $F \colon C(\AA)^* \to C(\AA)^*$ given by $F(\mu)(a) = \int_\AA \sum_{x \in \Gamma y} a(x) \, d\mu(y)$. 
In Corollary~\ref{cor:OSC_Hutchinson} we show---recovering a result of the first author's thesis \cite{Mam19}---that
when the open set condition is satisfied, the measure $\mu_{\infty}$ coincides with the classical Hutchinson measure with uniform weights. In general, $\mu_{\infty}$ is distinct from the Hutchinson measure (see Example~\ref{ex:cantor_not_hutchinson}). 

Our main $\KMS$-state classification results are summarised in the following statement.
\begin{mainthm}[Classification of $\KMS$ states; cf.~Theorems~\ref{thm:inf_type} and~\ref{thm:finite_type_classification}]\label{thm:intro_main}
Suppose that $(\AA,\Gamma)$ is a contractive iterated function system admitting a finite Markov partition\footnote[2]{We exclude the trivial case in which all maps in $\Gamma$ coincide on $\AA$ as then $\AA$ is a singleton and $\ln(\rho) = 0$.} and let $\rho$ be as in Theorem~\ref{thm:intro_mu_infty}.
The $\KMS_\beta$ states for the gauge action on the Kajiwara--Watatani algebra $\Oo_{\Gamma}$ are classified as follows:
\begin{enumerate}[label=\textup{(\arabic*)}]
	\item for $0 < \beta < \ln(\rho)$ there are no $\KMS_\beta$ states;
	\item for $\beta = \ln(\rho)$ there is a unique $\KMS_{\ln(\rho)}$ state $\varphi_\infty$, it is of infinite type, and its restriction to $C(\AA)$ is given by integration against $\mu_\infty$ of Theorem~\ref{thm:intro_mu_infty}; and
	\item for $\beta > \ln(\rho)$ the simplex of $\KMS_\beta$ states is parametrised by probability measures on $p_1(B_\Gamma)$, where $p_1 \colon \gr(\Gamma) \to \AA$ is the projection onto the first coordinate.
\end{enumerate}
\end{mainthm}
 In Subsection~\ref{subsec:examples} we provide a range of examples covered by Theorem~\ref{thm:intro_main}, but not by \cite[Theorem~6.6]{IKW07}. We also show that there exist systems that satisfy the hypotheses of \cite[Theorem~6.6]{IKW07} but do not admit finite Markov partitions (see Examples~\ref{ex:ikw_no_markov} and~\ref{ex:ikw_no_finite_markov}). An extension of our results to countable Markov partitions would cover these systems (see Remark~\ref{rmk:countable_markov_possible}).

\subsection*{Structure of the paper}
 Section~\ref{sec:IFS} fixes our conventions for iterated function systems, and introduces the cograph, branch points and the branch index. 
 Section~\ref{sec:KW} reviews Kajiwara--Watatani algebras and the Laca--Neshveyev machinery for $\KMS$ states, and introduces the Perron--Frobenius operator in our setting. Section~\ref{sec:markov_partitions} introduces Markov partitions, the refinement operator $\Theta$, and proves Theorem~\ref{thm:intro_markov}. Section~\ref{sec:classification} proves Theorems~\ref{thm:intro_mu_infty} and~\ref{thm:intro_main}, and includes the analytic construction of $\mu_\infty$ via an argument inspired by \cite[\S~4]{Hut81}.

A reader interested principally in the intrinsic Hutchinson measure $\mu_\infty$ may safely skip the $C^*$-algebraic material. The recommended pathway is: Section~\ref{sec:IFS} (definitions, branch index); Section~\ref{sec:KW}, only Subsection~\ref{subsec:PF} (the Perron--Frobenius operator $F$), together with the notation $F_\beta \coloneqq e^{-\beta}F$ and the definitions of $\beta$-subinvariant, infinite-type and finite-type measures in Subsection~\ref{subsec:KMS_LN}; Section~\ref{sec:markov_partitions} in full (Markov partitions and the algorithm); and finally Section~\ref{sec:classification} up to the end of Subsection~\ref{subsec:inf_type}, which contains Theorem~\ref{thm:inf_type}---the existence and uniqueness statement for $\mu_\infty$---proved by measure-theoretic and Perron--Frobenius arguments.

\subsection*{Acknowledgements}
The second and third authors were supported by University of Wollongong RevITAlising Research grant IV036, and the second author was supported by University of Wollongong AGEiS CONNECT grant 141765. Early discussions between the first two authors occurred at the MATRIX workshop ``Refining $C^*$-algebraic invariants for dynamics using KK-theory'' hosted by the University of Melbourne. 

\subsection*{LLM Disclosure}
Anthropic's Claude (Opus 5 series) was used during proofreading, literature search, and to assist in streamlining some arguments. The core mathematical concepts, proof strategies, and the writing are the authors' own. They assume full responsibility for mathematical correctness.

\section{Iterated function systems and their branch points}\label{sec:IFS}

\subsection{Iterated function systems} 
There are various related definitions of iterated function systems extant in the literature. In this article we use the following non-standard definition. 

\begin{dfn}\label{dfn:IFS}
	An \emph{iterated function system} is a pair $(\AA,\Gamma)$ consisting of a second-countable compact Hausdorff space $\AA$---called the \emph{attractor}---together with a finite collection $\Gamma$ of continuous injections on $\AA$. In addition, we require that the space $\AA$ is \emph{$\Gamma$-invariant} in the sense that
	\begin{equation}\label{eq:Gamma_invariance}
		\AA = \bigcup_{\gamma \in \Gamma} \gamma(\AA).
	\end{equation}
	An iterated function system $(\AA,\Gamma)$ is said to be \emph{contractive} if $\AA$ admits a metric $d$ compatible with its topology and for each $\gamma \in \Gamma$ there exists $c_{\gamma} \in [0,1)$ such that
	\[
	d(\gamma(x), \gamma(y)) \le c_{\gamma} d(x,y) \quad \text{for all } x,y \in \AA.
	\]
	That is, each $\gamma \in \Gamma$ is a \emph{contraction}.
\end{dfn}

Since each $\gamma \in \Gamma$ is a continuous injection on a compact space, $\gamma(\AA)\simeq\AA$, so \eqref{eq:Gamma_invariance} decomposes $\AA$ into finitely many (possibly overlapping) copies of itself.
Since $\AA$ is compact, every compatible metric on $\AA$ makes the space complete.
Many of the results in this article hold without the contractivity assumption on $(\AA,\Gamma)$, but some results (like Theorem~\ref{thm:inf_type}) require it. 
A point of difference in Definition~\ref{dfn:IFS} when compared to other definitions of iterated function systems (cf.~\cite[p.~133]{Fal14}) is that we do not embed $\AA$ in a larger ambient space: our systems are restricted to the attractor. 

Central to the analysis of iterated function systems is the following map.
\begin{dfn}\label{dfn:Hutchinson_operator}
Let $\Kk_0(\AA)$ denote the collection of nonempty compact subsets of $\AA$.
The \emph{Hutchinson operator} of $(\AA,\Gamma)$---also denoted by $\Gamma$---is the function $\Gamma \colon \Kk_0(\AA) \to \Kk_0(\AA)$ defined by
	\[
	\Gamma(K) = \bigcup_{\gamma \in \Gamma} \gamma(K).
	\] 
	We write $\Gamma x \coloneqq \Gamma(\{x\})$ for $x \in \AA$.  It will be clear from context whether $\Gamma$ refers to a collection of maps or the associated Hutchinson operator. 
\end{dfn}

By \eqref{eq:Gamma_invariance}, the set $\AA$ is a fixed point of the Hutchinson operator. If $(\AA,\Gamma)$ is contractive, then Hutchinson's Theorem \cite[\S~3.1]{Hut81} (cf.~\cite[Theorem~1.1.4]{Kig01}) states that $\AA$ is its \emph{unique} fixed point, and that for every $K \in \Kk_0(\AA)$ the sequence $\Gamma^n(K)$ converges to $\AA$ in the Hausdorff metric on $\Kk_0(\AA)$. It is in this sense that $\AA$ is an attractor of the system $(\AA,\Gamma)$.

In an iterated function system, the images of distinct maps often \emph{overlap}. Systems whose overlaps are ``small'' include those satisfying the following separation condition.

\begin{dfn}[{\cite[\S~5.2]{Hut81}}]\label{dfn:open_set_condition}
	An iterated function system $(\AA,\Gamma)$ satisfies the \emph{open set condition} if there exists a nonempty open subset $U \subseteq \AA$ such that 
	\[
	\bigcup_{\gamma \in \Gamma} \gamma(U) \subseteq U \quad \text{ and } \quad \gamma(U) \cap \gamma'(U) = \varnothing
	\]
	for all $\gamma, \gamma'  \in \Gamma$ with $\gamma \ne \gamma'$. 
\end{dfn}

We record the following well-known facts.
\begin{lem}\label{lem:OSC_dense}
If $U$ witnesses the open set condition in a contractive iterated function system $(\AA,\Gamma)$, then $U$ is dense in $\AA$.
\end{lem}
\begin{proof}
	The set $\ol{U}$ is compact, being a closed subset of the compact space $\AA$.
	We have $\gamma(\ol{U}) \subseteq \ol{U}$ for all $\gamma \in \Gamma$, so $\Gamma(\ol{U}) \subseteq \ol{U}$. Iterating, $\Gamma^n(\ol{U}) \subseteq \ol{U}$ for all $n \in \NN$. Since $\Gamma^n(\ol{U})$ converges in Hausdorff metric to $\AA$, we have $\AA \subseteq \ol{U}$.
\end{proof}
\begin{lem}\label{lem:eventually_in_open_sets}
	If $(\AA,\Gamma)$ is contractive and $U$ is a nonempty open subset of $\AA$, then there exist $k \in \NN$ and $\gamma_1,\ldots,\gamma_k \in \Gamma$ such that $\gamma_1 \circ \cdots \circ \gamma_k(\AA) \subseteq U$. 
\end{lem}
\begin{proof}
	Let $c_{\max} \coloneqq \max\{c_\gamma \mid \gamma \in \Gamma\}$ be the largest contraction constant. Fix $\varepsilon >0$ and $x \in U$ such that the ball of radius $\varepsilon$ about $x$ is contained in $U$.
Choose $k \in \NN$ sufficiently large so that $c_{\max}^k\diam(\AA)< \varepsilon$. By \eqref{eq:Gamma_invariance} there exist $y \in \AA$ and $\gamma_1,\ldots,\gamma_k \in \Gamma$ such that $x = \gamma_1 \circ \cdots \circ \gamma_k(y)$. For every $y' \in \AA$ we have
$
	 d(\gamma_1 \circ \cdots \circ \gamma_k(y'),x)\le c_{\max}^k  d(y',y) < \varepsilon,
$
 so $\gamma_1 \circ \cdots \circ \gamma_k(\AA) \subseteq U$.
\end{proof}

We are interested in the dynamics of iterated function systems, in particular the dynamics of the Hutchinson operator. To trace the dynamics of the Hutchinson operator, we use the following notion of paths.
\begin{dfn}[{cf.~\cite{KW06}}] \label{dfn:n_step_forward_orbit} 
For $n \ge 0$ the space of \emph{paths of length $n$} is 
\begin{align*}
\Gamma^{(n)} &= \{ (x_0,x_1,\ldots,x_{n-1},x_n) \in \AA^{n+1} \mid x_i \in \Gamma x_{i+1}, \forall\, 0 \le i < n\}\\
&= \{ (x_0,x_1,\ldots,x_{n-1},x_n) \in \AA^{n+1} \mid \forall \, 0 \le i < n ,\,\exists \gamma_{i} \in \Gamma \text{ such that } x_i = \gamma_i(x_{i+1})\}.
\end{align*}
For $y \in \AA$ we write $\Gamma^{(n)}y \coloneqq \{ (x_0,\ldots,x_{n-1},x_n) \in \Gamma^{(n)} \mid x_n = y\}$.
\end{dfn}

It is possible for $x \in \AA$ to satisfy $\gamma(x) = \gamma'(x)$ for distinct $\gamma,\gamma' \in \Gamma$. In particular, given $(x_0,\ldots,x_n) \in \Gamma^{(n)}$ the choice of maps $\gamma_i \in \Gamma$ such that $x_i = \gamma_i(x_{i+1})$ may not be unique. 
As such, $|\Gamma^{(n)}x|$ may grow more slowly than $|\Gamma|^n$ as $n \to \infty$ for a given $x \in \AA$. 

\begin{example}\label{ex:overlap}
Let $\AA = [0,1]$ and $\Gamma = \{\gamma_1,\gamma_2\}$ where $\gamma_1(x) = x/2$ and $\gamma_2(x) = 1 - x/2$. That is, $\gamma_1$ and $\gamma_2$ are the inverse branches of the tent map $T(x) = 1 - 2|x-1/2|$.
Then $\gamma_1(1) = \gamma_2(1) = 1/2$, so $|\Gamma^{(1)}1| = 1 \ne 2 = |\Gamma|$.
\end{example}

\subsection{The branch set}\label{subsec:branch}

In this subsection we introduce the notion of a branch point. Some results of this subsection appeared in the second author's thesis \cite[\S2.2]{Mun20}. Our definition of a branch point differs from that introduced in \cite{KW06} and applies more generally. Branch points act as bifurcation points in the dynamics of an iterated function system. To introduce them, we require the following definition. 

\begin{dfn}\label{dfn:cograph}
Let $(\AA,\Gamma)$ be an iterated function system. For each $\gamma \in \Gamma$ we call  $\gr(\gamma) \coloneqq \{ (\gamma(x),x) \mid x \in \AA \} \subseteq \AA^2$ the \emph{cograph of} $\gamma$. The \emph{cograph of} $\Gamma$ is the set
\[
\gr(\Gamma) \coloneqq \bigcup_{\gamma \in \Gamma} \gr(\gamma) = \{ (x,y) \in \AA^2 \mid x \in \Gamma y\} = \Gamma^{(1)},
\]
which we endow with the subspace topology inherited from $\AA^2$. 
\end{dfn}
In this article we make no mention of the \emph{graph} of a function, so the notation $\gr$ always refers to the cograph. 
The Closed Graph Theorem for continuous maps between compact Hausdorff spaces implies that each $\gr(\gamma)$ is compact, so $\gr(\Gamma)$ is compact. 

Throughout this article we let $p_1 \colon \gr(\Gamma) \to \AA$ and $p_2 \colon \gr(\Gamma) \to \AA$ denote the projections onto the first and second factors, respectively. The projection $p_2$ is clearly surjective, and the $\Gamma$-invariance condition \eqref{eq:Gamma_invariance} implies that $p_1$ is also surjective. We record the following. 
\begin{lem}\label{lem:p2_open}
	The projection $p_2 \colon \gr(\Gamma) \to \AA$ is open.
\end{lem}
\begin{proof}
	For each $\gamma \in \Gamma$, let $p_2^\gamma \colon \gr(\gamma) \to \AA$ be the projection onto the second factor. Then $p_2^\gamma$ is a continuous bijection between compact Hausdorff spaces and is therefore a homeomorphism. Consequently, for any open $U \subseteq \gr(\Gamma)$, the set
	$p_2(U) = \bigcup_{\gamma \in \Gamma}p_2^\gamma (U \cap \gr(\gamma))$
	is open in $\AA$.
\end{proof}

\begin{dfn}[{\cite{KW06}}]
	The \emph{branch index} of $(x,y) \in \gr(\Gamma)$ is defined by 
	\[
	b(x,y)\coloneqq |\{ \gamma \in \Gamma \colon x = \gamma(y)\}|.
	\]
\end{dfn}
It is claimed in \cite[p.~1900]{IKW07} that the set of discontinuities of the branch index is given by $\{(x,y) \in \gr(\Gamma) \mid x = \gamma(y) = \gamma'(y) \text{ for some } \gamma \ne \gamma'\}$. This, however, is not true for iterated function systems satisfying Definition~\ref{dfn:IFS}, but remains true for iterated function systems satisfying the open set condition (see Proposition~\ref{prop:open-set_branch} below).
For our general iterated function systems the discontinuities of the branch index may be described in terms of the boundaries of the intersections of the cographs $\gr(\gamma)$. We make this more precise now, and remark that our approach is similar to that taken by Dor-On \cite{Dor18} for weighted partial systems.

	For each nonempty subset $\Lambda \subseteq \Gamma$ consider the closed subset $G_\Lambda \subseteq \gr(\Gamma)$ defined by 
	\begin{equation}\label{eq:G_Lambda}
	G_{\Lambda} \coloneqq \bigcap_{\gamma \in \Lambda} \gr(\gamma).
	\end{equation}
	Observe that if $\Lambda' \subseteq \Lambda$, then $G_\Lambda \subseteq G_{\Lambda'}$. 
The branch index can then be expressed by
\[
b(x,y) = \max \{ |\Lambda| \colon (x,y) \in G_\Lambda \text{ and } \Lambda \subseteq \Gamma  \}.
\]
Let $\partial G_{\Lambda}$ denote the boundary of $G_\Lambda$ taken as a subspace of $\gr(\Gamma)$.
	 
	\begin{dfn}\label{dfn:branched_set} The \emph{branch set} of $(\AA,\Gamma)$ is the compact set
	\[
	B_{\Gamma} \coloneqq \bigcup_{\varnothing \ne \Lambda \subseteq \Gamma} \partial G_{\Lambda} \subseteq \gr(\Gamma).
	\] 
	 We refer to elements of $B_{\Gamma}$ as \emph{branch points}.
	\end{dfn}

\begin{rmk}\label{rmk:only_need_size_2}
	In the union defining $B_\Gamma$ we need only consider $\Lambda$ with $|\Lambda| \ge 2$. Indeed, if $(x,y) \in \partial G_{\{\gamma\}} = \partial \gr(\gamma)$, then for any open neighbourhood $U$ of $(x,y)$, there exists $\gamma' \ne \gamma$ such that $U \cap \gr(\gamma') \ne \varnothing$. By passing to a subsequence if necessary, we can find a single $\gamma' \in \Gamma$ and a sequence $(x_n,y_n) \in \gr(\gamma') \setminus \gr(\gamma)$ such that $(x_n,y_n) \to (x,y)$. Since $\gr(\gamma')$  is closed, $(x,y) \in G_{\{\gamma,\gamma'\}}$. Moreover, every neighbourhood of $(x,y)$ meets $\gr(\Gamma) \setminus \gr(\gamma)$, since $(x,y) \in \partial \gr(\gamma)$, and $\gr(\Gamma) \setminus \gr(\gamma) \subseteq \gr(\Gamma) \setminus G_{\{\gamma,\gamma'\}}$, since $G_{\{\gamma,\gamma'\}} \subseteq \gr(\gamma)$. Hence, $(x,y) \in \partial G_{\{\gamma,\gamma'\}}$.
\end{rmk}

The name ``branch set'' is justified by the fact that it coincides with the points in $\gr(\Gamma)$ at which the projection $p_2$ fails to be locally injective (see Proposition~\ref{prop:branch_characterisation} below). In particular, the branch set corresponds to branching behaviour in the cograph $\gr(\Gamma)$, as the following example illustrates.

\begin{example}\label{ex:clarence}
	Consider the contractive iterated function system $([0,1],\Gamma = \{\gamma_1,\gamma_2,\gamma_3\})$ defined by
\begin{align*}
	\gamma_1(x) &= \frac{x}{2}, & 
	\gamma_2(x) &= 1 - \frac{x}{2}, & & \text{and} & \gamma_3(x) &= 
	\begin{cases}
	\frac{x}{4} + \frac{1}{8} &\text{if } x < 1/2\\
	\frac{x}{2} & \text{if } x \ge 1/2.
	\end{cases}
\end{align*}
For this system $\gr(\Gamma)$ may be depicted visually by its corresponding graph
\begin{center}
\begin{tikzpicture}[vertex/.style={circle, fill=red, inner sep=1.5pt},
scale=4.5]
\begin{scope}[font=\scriptsize]

\draw [thick] (0,0) node [below left] {$0$} -- (1,0);
\draw [thick] (0,0) -- (0,1);

\draw[thick] (1,0.02) -- (1,-0.02) node [below]  {$1$};
\draw[thick] (0.5,0.02) -- (0.5,-0.02) node [below]  {$1/2$};
\draw[thick] (0.02,1) -- (-0.02,1) node [left]  {$1$};
\draw[thick] (0.02,0.5) -- (-0.02,0.5) node [left]  {$1/2$};
\draw[thick] (0.02,0.25) -- (-0.02,0.25) node [left]  {$1/4$};
\draw[thick] (0.02,0.125) -- (-0.02,0.125) node [left]  {$1/8$};

\draw (0,0) --  node [below right] {$\gr(\gamma_1)$} (0.5,0.25) -- (1,0.5) --  node [above right] {$\gr(\gamma_2)$} (0,1) (0,0.125) -- node [above] {$\gr(\gamma_3)$} (0.5,0.25);

\node[vertex] at (1,0.5) {};
\node[vertex] at (0.5,0.25) {};

\end{scope}
\end{tikzpicture}
\end{center}
in which $p_1$ projects onto the vertical axis and $p_2$ projects onto the horizontal axis. The branch set $B_\Gamma = \{(1/2,1), (1/4,1/2)\}$ is highlighted in red. 
This system does not satisfy the open set condition as any prospective witness $U$ must intersect $(1/2,1)$ by Lemma~\ref{lem:OSC_dense}, so $\gamma_1(U) \cap \gamma_3(U) \ne \varnothing$.
\end{example}

Discontinuities of the branch index correspond precisely to the branch set.
\begin{lem}\label{lem:branch_iff_discontinuous}
	The branch index $b \colon \gr(\Gamma) \to \{1,\ldots,|\Gamma|\}$ is upper semicontinuous. Moreover, $b$ is continuous at $(x,y) \in \gr(\Gamma)$ if and only if $(x,y) \notin B_\Gamma$. 
\end{lem}

\begin{proof}
	Fix $(x,y) \in \gr(\Gamma)$. There exists a subset $\Lambda' \subseteq \Gamma$ which is maximal in the sense that $(x,y) \in G_{\Lambda'}$ and if $\Lambda' \subsetneq \Lambda$ then $(x,y) \notin G_{\Lambda}$. Since $G_{\Lambda}$ is closed for all $\Lambda \subseteq \Gamma$, the set
	\[
	K \coloneqq \bigcup \{G_{\Lambda} \mid \Lambda \subseteq \Gamma \text{ and } |\Lambda| > |\Lambda'|\}
	\]
	is closed (potentially empty) in $\gr(\Gamma)$. 
	Since no $G_{\Lambda}$ with $|\Lambda|>|\Lambda'|$ contains $(x,y)$, the set $\gr(\Gamma) \setminus K$ is an open neighbourhood of $(x,y)$ such that $b(x',y') \le |\Lambda'| = b(x,y)$ for all $(x',y')$ in it, so $b$ is upper semicontinuous at $(x,y)$. If moreover $(x,y) \notin B_{\Gamma}$, then $(x,y) \notin \partial G_{\Lambda'}$, so $\interior(G_{\Lambda'}) \setminus K$ is an open neighbourhood of $(x,y)$ on which $b$ is constant. Hence, $b$ is continuous on $\gr(\Gamma) \setminus B_{\Gamma}$.

	To see that the branch points are precisely the discontinuities of $b$, suppose that $b$ is continuous at $(x,y) \in \gr(\Gamma)$. Then there is a neighbourhood $V \subseteq \gr(\Gamma)$ of $(x,y)$ such that $b(x',y') = b(x,y)$ for all $(x',y') \in V$. Let $\Lambda' \subseteq \Gamma$ be such that $(x,y) \in G_{\Lambda'}$ and $b(x,y) = |\Lambda'|$. Consider the open neighbourhood
	\[
	W \coloneqq V \setminus \bigcup \{ G_{\Lambda} \colon |\Lambda| = |\Lambda'| \text{ and } \Lambda \ne \Lambda' \}
	\] 
	of $(x,y)$. Since $b(x',y') = |\Lambda'|$ for all $(x',y') \in W$ it follows that
	$
	W \subseteq \interior(G_{\Lambda'}) \setminus \bigcup \{G_{\Lambda} \mid \Lambda' \subsetneq \Lambda\}.
	$

	We claim that $W \cap B_{\Gamma} = \varnothing$ from which we deduce that $(x,y) \notin B_{\Gamma}$.
	Clearly $W \cap \partial G_{\Lambda'} = \varnothing$. For $\Lambda \subseteq \Gamma$, we have 
	\[
	W \cap \partial G_{\Lambda} \subseteq W \cap G_{\Lambda} \cap G_{\Lambda'} \subseteq W \cap G_{\Lambda \cup \Lambda'}
	\] which is empty whenever $\Lambda \ne \varnothing$ is not a subset of $\Lambda'$. On the other hand, suppose $\Lambda \subsetneq \Lambda'$. Since $G_{\Lambda'} \subseteq G_{\Lambda}$ it follows that $\partial G_{\Lambda} \cap G_{\Lambda'} \subseteq \partial G_{\Lambda'}$. Hence,
	\[
	W \cap \partial G_{\Lambda}= W \cap	\partial G_{\Lambda} \cap G_{\Lambda'}  \subseteq W \cap \partial G_{\Lambda'} = \varnothing. 
	\]
	So if $b$ is continuous at $(x,y)$, then $(x,y) \notin B_{\Gamma}$. Taking the contrapositive, $b$ is discontinuous at branch points. 
\end{proof}

Recall that a function $f \colon X \to Y$ between topological spaces is said to be \emph{locally injective at $x \in X$} if there exists an open neighbourhood $U$ of $x$ such that $\restr{f}{U}$ is injective. 

\begin{prop}\label{prop:branch_characterisation}
	The following are equivalent for $(x,y) \in \gr(\Gamma)$:
	\begin{enumerate}[label=(\roman{enumi})]
		\item\label{itm:branch} $(x,y) \in B_\Gamma$,
		\item\label{itm:discontinuous} $b \colon \gr(\Gamma) \to \{1,\ldots,|\Gamma|\}$ is discontinuous at $(x,y)$, and
		\item\label{itm:nonlocinj} the projection $p_2 \colon \gr(\Gamma) \to \AA$ given by $p_2(x,y)= y$ is not locally injective at $(x,y)$. 
	\end{enumerate}
\end{prop}
\begin{proof}
	The equivalence \ref{itm:branch} $\iff$ \ref{itm:discontinuous} follows from Lemma~\ref{lem:branch_iff_discontinuous}.
	
	For \ref{itm:branch}~$\implies$~\ref{itm:nonlocinj} fix $(x,y) \in B_{\Gamma}$ and let $\Lambda \subseteq \Gamma$ be such that $(x,y) \in \partial G_{\Lambda}$. Fix open neighbourhoods $U$ and $V$ of $x$ and $y$, respectively. Continuity of each $\gamma \in \Gamma$ implies \[W = V \cap \bigcap_{\gamma \in \Lambda} \gamma^{-1}(U)\] is an open neighbourhood of $y$. Since $(U \times W) \cap \gr(\Gamma)$ is an open neighbourhood of $(x,y) \in \partial G_{\Lambda}$, there exists $(x',y') \in (U \times W) \cap \gr(\Gamma)$ such that $(x',y') \notin G_{\Lambda}$. In particular, there exist $\gamma \in \Lambda$ and $\gamma' \in \Gamma$ such that $\gamma(y') \ne x' = \gamma'(y')$. The construction of $W$ ensures that $(\gamma(y'),y')$ and $(\gamma'(y'),y')$ both belong to $(U \times W) \cap \gr(\Gamma) \subseteq (U \times V) \cap \gr(\Gamma)$. Since $p_2(\gamma(y'),y') = y' = p_2(\gamma'(y'),y')$, the restriction of $p_2$ to $(U \times V) \cap \gr(\Gamma)$ is not injective. As $(U \times V) \cap \gr(\Gamma)$ is an arbitrary basic open neighbourhood of $(x,y)$ in $\gr(\Gamma)$, $p_2$ is not locally injective at $(x,y)$. 
	
	For \ref{itm:nonlocinj}~$\implies$~\ref{itm:branch} suppose that $p_2$ is not locally injective at $(x,y)$. Let $\Lambda = \{ \gamma \in \Gamma \mid x = \gamma(y)\}$ so that $(x,y) \in G_{\Lambda}$. Fix an open neighbourhood $U \subseteq \gr(\Gamma)$ of $(x,y)$ and consider the open set \[V \coloneqq U \setminus \bigcup_{\gamma \in \Gamma \setminus \Lambda}\gr(\gamma).\] Since $V$ is open and $p_2$ is not locally injective, there exist $(x',y')$ and $(x'',y')$ in $V$ with $x' \ne x''$. The construction of $V$ ensures that there exist $\gamma',\gamma'' \in \Lambda$ such that $x' = \gamma'(y')$ and $x'' = \gamma''(y')$. Since $x' \ne x''$ it follows that $(x',y') \notin G_{\Lambda}$. Consequently, $(x,y) \in \partial G_{\Lambda}$. 
\end{proof}

If $(\AA,\Gamma)$ is contractive and satisfies the open set condition, then the branch set coincides with the set of discontinuities of $b$ identified in \cite[p.~1900]{IKW07}. 

\begin{prop}\label{prop:open-set_branch}
	If $(\AA,\Gamma)$ is contractive and satisfies the open set condition, then
	\begin{align*}
		B_\Gamma &= \{ (x,y) \in \gr(\Gamma) \mid b(x,y) \ge 2 \}\\
		&= \{ (x,y) \in \gr(\Gamma) \mid x = \gamma(y) = \gamma'(y) \text{ for some }\gamma \ne \gamma'\}.
	\end{align*}
\end{prop}
\begin{proof}
	By Remark~\ref{rmk:only_need_size_2}, $B_{\Gamma} \subseteq \{ (x,y) \in \gr(\Gamma) \mid b(x,y) \ge 2 \}$.	
	For the reverse containment, suppose that $U$ is an open set witnessing the open set condition and fix $(x,y) \in \gr(\Gamma)$ such that $b(x,y) \ge 2$. Then there exist $\gamma \ne \gamma'$ in $\Gamma$ such that $x = \gamma(y) = \gamma'(y)$, so $(x,y) \in G_{\{\gamma,\gamma'\}}$. 
	 By Lemma~\ref{lem:OSC_dense} the set $U$ is dense in $\AA$. Since $p_2 \colon \gr(\Gamma) \to \AA$ is continuous and open, $p_2^{-1}(U)$ is a dense open subset of $\gr(\Gamma)$ with the property that $b(x',y') =1$ for all $(x',y') \in p_2^{-1}(U)$. Hence, for any open neighbourhood $W \subseteq \gr(\Gamma)$ of $(x,y)$ we have $W \cap p_2^{-1} (U) \ne \varnothing$, and so there exists $(x',y') \in W$ with $b(x',y') = 1$. Consequently, $(x,y) \in \partial G_{\{\gamma,\gamma'\}} \subseteq B_{\Gamma}$. 
\end{proof}

\section{Kajiwara--Watatani algebras and KMS states}\label{sec:KW}
In \cite{KW06} Kajiwara and Watatani introduced a family of $C^*$-algebras associated to iterated function systems as a means to analyse their dynamics using operator theoretic tools. In this section we recall the construction of these algebras. We also introduce the $\KMS$-state machinery for them developed in \cite{IKW07}, which is in turn based on the work
of Laca--Neshveyev \cite{LN04}. We assume the reader of this section is familiar with $C^*$-algebras and Hilbert $C^*$-modules. Standard references for these topics include
\cite{Mur90} for the former and \cite{Lan95} for the latter. 

\subsection{Cuntz--Pimsner algebras} We briefly recall some basic definitions and notations required in the construction of Cuntz--Pimsner algebras. 
Let $A$ be a $C^*$-algebra. Given a (right) Hilbert $A$-module $X$ we denote the right action of $a \in A$ on $x \in X$ by $x \cdot a \in X$ and the associated right $A$-linear inner product by $\langle\,\cdot \,| \, \cdot\,\rangle_A$. The $C^*$-algebra of adjointable operators on $X$ is denoted by $\End_A(X)$. The compact operators on $X$ are the closed two-sided ideal $\End_A^0(X) \coloneqq \overline{\spaan} \{\Theta_{x,y} \mid x,y \in X\}$ of $\End_A(X)$, where $\Theta_{x,y}(z) = x \cdot \langle y \mid z\rangle_A$ for all $z \in X$. We say that $X$ is \emph{full} if $\overline{\langle X \mid X\rangle}_A = A$. 

A \emph{$C^*$-correspondence over $A$}  is a pair $(\lambda,X)$ consisting of a right Hilbert $A$-module $X$ together with a $*$-homomorphism $\lambda \colon A \to \End_A(X)$. If $\overline{\lambda(A)X} = X$ we say that $(\lambda,X)$ is \emph{nondegenerate} and if $\lambda$ is injective we say that $(\lambda,X)$ is \emph{injective}. We write $a \cdot x$ as a shorthand for $\lambda(a)x$. 

In \cite{Pim97} Pimsner introduced two universal $C^*$-algebras associated to an injective {$C^*$-correspondence} $(\lambda,X)$ called the \emph{Toeplitz--Pimsner algebra} $\Tt_X$ and \emph{Cuntz--Pimsner algebra} $\Oo_X$. Katsura \cite{Kat04cor} extended Pimsner's work to $C^*$-correspondences that are not necessarily injective. We deal exclusively with injective $C^*$-correspondences. 
To describe these algebras we recall from \cite{Pim97} that a \emph{representation} $(\pi,\psi) \colon (\lambda, X) \to B$ of a $C^*$-correspondence $(\lambda,X)$ over $A$ in a {$C^*$-algebra} $B$ consists of a $*$-homomorphism $\pi \colon A \to B$ together with a linear map $\psi \colon X \to B$ such that for all $x,y \in X$ and $a \in A$,
\begin{align*}
	\psi(x \cdot a) = \psi(x)\pi(a), \quad \psi(a \cdot x) = \pi(a)\psi(x),  \quad \text{and} \quad \psi(x)^* \psi(y) = \pi(\langle x \mid y\rangle_A). 
\end{align*}
A representation $(\pi,\psi) \colon (\lambda, X) \to B$ induces a $*$-homomorphism $\psi^{(1)} \colon \End_A^0(X) \to B$ satisfying $\psi^{(1)}(\Theta_{x,y}) = \psi(x)\psi(y)^*$ for all $x,y \in X$ (see \cite[Lemma 2.2]{KPW98}). If $\lambda \colon A \to \End_A(X)$ is injective, then a representation $(\pi,\psi) \colon (\lambda, X) \to B$ is said to be \emph{covariant} if 
\[
	\psi^{(1)} \circ \lambda(a) = \pi(a) \quad \text{ for all } a \in \lambda^{-1}(\End_A^0(X)).
\]
We call $
J_X \coloneqq \lambda^{-1}(\End_A^0(X))
$
the \emph{covariance ideal} of $(\lambda,X)$.

\begin{dfn} Let $(\lambda,X)$ be an injective nondegenerate $C^*$-correspondence over $A$. 

The \emph{Toeplitz--Pimsner algebra} $\Tt_X$ is the universal $C^*$-algebra generated by a representation  $(j_A,j_X) \colon (\lambda,X) \to \Tt_X$ in the sense that $\Tt_X = C^*(j_A(A) \cup j_X(X))$ and for any representation $(\pi,\psi) \colon (\lambda, X) \to B$ there is a unique $*$-homomorphism $\pi \times \psi \colon \Tt_X \to B$ such that $(\pi \times \psi) \circ j_A = \pi$ and $(\pi \times \psi) \circ j_X = \psi$.

The \emph{Cuntz--Pimsner algebra $\Oo_X$} is the universal $C^*$-algebra generated by a covariant representation $(i_A,i_X) \colon (\lambda,X) \to \Oo_X$ in the sense that $\Oo_X = C^*(i_A(A) \cup i_X(X))$ and for any covariant representation $(\pi,\psi) \colon(\lambda,X) \to B$ there is a unique $*$-homomorphism $\pi \times \psi \colon \Oo_X \to B$ such that $(\pi \times \psi) \circ i_A = \pi$ and $(\pi \times \psi) \circ i_X = \psi$. 
\end{dfn}

Since covariant representations are representations, the universal property of $\Tt_X$ implies that $\Oo_X$ is a quotient of $\Tt_X$. If $(\lambda,X)$ is injective, the universal maps $j_A$ and $i_A$ are injective, so $A$ may be considered to be a subalgebra of both $\Tt_X$ and $\Oo_X$.
If $(j_A,j_X) \colon (\lambda,X) \to \Tt_X$ is a universal representation, then for each $z \in \TT$, $(j_A, \xi \mapsto z j_X(\xi)) \colon (\lambda,X) \to \Tt_X$ is also a universal representation. Hence, there is a strongly continuous action $\alpha \colon \TT \to \Aut(\Tt_X)$ satisfying $\alpha_z(j_A(a)) = j_A(a)$ and $\alpha_z(j_X(\xi)) = zj_X(\xi)$ for all $a \in A$ and $\xi \in X$. Abusing notation, the action $\alpha$ lifts to an action $\alpha \colon \RR \to \Aut(\Tt_X)$ by $\alpha_t = \alpha_{e^{i t}}$. There is also an analogous action $\alpha \colon \RR \to \Aut(\Oo_X)$, and the quotient map $q \colon \Tt_X \to \Oo_X$ is equivariant with respect to these actions.

\begin{dfn}
The actions $\alpha$ on $\Tt_X$ and $\Oo_X$ are called \emph{gauge actions}.
\end{dfn}

\subsection{Kajiwara--Watatani algebras}

In \cite{KW06} Kajiwara and Watatani introduced a class of Cuntz--Pimsner algebras that encode certain aspects of the dynamics of a contractive iterated function system. Kajiwara and Watatani worked with contractive systems, but their construction applies verbatim to the systems of Definition~\ref{dfn:IFS}.

Fix an iterated function system $(\AA,\Gamma)$. 
Let $A \coloneqq C(\AA)$ be the $C^*$-algebra of continuous complex-valued functions on $\AA$ and let $X_{\Gamma} \coloneqq C(\gr(\Gamma))$ be the space of continuous complex-valued functions on the cograph $\gr(\Gamma)$ of Definition~\ref{dfn:cograph}. Define right and left actions of $a \in A$ on $\xi \in X_{\Gamma}$ by
\[
(\xi \cdot a)(x,y) \coloneqq \xi(x,y)a(y) \quad \text{and} \quad (a \cdot \xi)(x,y) \coloneqq a(x)\xi(x,y), \quad (x,y) \in \gr(\Gamma).
\]  
For $\xi, \eta \in X_{\Gamma}$, define an $A$-valued inner product by
\begin{align*}
\langle \xi \mid \eta\rangle_A (y) \coloneqq \sum_{\gamma \in \Gamma} \overline{\xi(\gamma(y),y)} \eta(\gamma(y),y)
 = \sum_{x \in \Gamma y} b(x,y) \overline{\xi(x,y)}\eta(x,y), \quad y \in \AA.
\end{align*}
Defining $\lambda\colon A \to \End_A(X_\Gamma)$ by $\lambda(a)\xi = a \cdot \xi$ makes $(\lambda,X_{\Gamma})$ into a $C^*$-correspondence over $A$. Moreover, the norm on $X_{\Gamma}$, given by $\|\xi\| \coloneqq \|\langle \xi 
\mid \xi\rangle_A\|^{1/2}$, is equivalent to the supremum norm on $C(\gr(\Gamma))$ (see \cite[Proposition 2.1]{KW06}).
The left action $\lambda$ is unital, and surjectivity of $p_1$ implies that $\lambda$ is injective. Since $\langle 1_{X_\Gamma} \mid 1_{X_\Gamma}\rangle_A = |\Gamma|\, 1_A$ the module $X_\Gamma$ is full.

\begin{dfn}
	We call $(\lambda,X_{\Gamma})$ the \emph{Kajiwara--Watatani correspondence} of $(\AA,\Gamma)$ and call the associated Cuntz--Pimsner algebra $\Oo_{\Gamma} \coloneqq \Oo_{X_{\Gamma}}$  the \emph{Kajiwara--Watatani algebra} of $(\AA,\Gamma)$. We also write $\Tt_\Gamma \coloneqq \Tt_{X_{\Gamma}}$.
\end{dfn}

Since $\Oo_{\Gamma}$ is universal for covariant representations of $(\lambda,X_{\Gamma})$, the first step towards understanding the structure of $\Oo_{\Gamma}$ is understanding the covariance ideal $J_{X_{\Gamma}} = \lambda^{-1}(\End_A^0(X_{\Gamma}))$. 
In the case of contractive systems satisfying the open set condition it is shown in \cite[Proposition 2.4]{KW06} that $J_{X_\Gamma}$ consists of the functions vanishing on $\{x \in \AA \mid x = \gamma(y) = \gamma'(y) \text{ for some } y \in \AA \text{ and } \gamma \neq \gamma' \in \Gamma\}$. In general, we describe $J_{X_{\Gamma}}$ in terms of the branch set $B_\Gamma$ of Definition~\ref{dfn:branched_set}.

As observed in \cite{IKW07}, $(\lambda, X_\Gamma)$ can be realised as a $C^*$-correspondence associated to a topological quiver in the sense of Muhly--Tomforde \cite{MT05}. In particular, $\gr(\Gamma)$ is the space of edges, $\AA$ is the space of vertices, and the projections $p_1$ and $p_2$ correspond to the source and range maps, respectively. The branch index induces a weighted counting measure on the fibres $p_2^{-1}(y)$ for $y \in \AA$. 
A general result of Muhly--Tomforde \cite[Theorem 3.11]{MT05} characterises the ideal $J_{X_{\Gamma}}$ in terms of local injectivity of $p_2$. Together with Proposition~\ref{prop:branch_characterisation} this gives a characterisation of the covariance ideal $J_{X_{\Gamma}}$ in terms of the branch set. 
\begin{prop}[cf.~{\cite[Proposition 4.4]{IKW07}}]\label{prop:covariance_ideal}
	The covariance ideal of the Kajiwara--Watatani correspondence is given by
	\[
	J_\Gamma \coloneqq J_{X_{\Gamma}} = C_0(\AA \setminus p_1(B_\Gamma)).
	\]
\end{prop}
For contractive systems satisfying the open set condition Proposition~\ref{prop:open-set_branch} recovers the characterisation of $J_{\Gamma}$ from \cite{KW06}.

\subsection{The Perron--Frobenius operator}\label{subsec:PF}

Central to our analysis of $\KMS$ states on $\Oo_{\Gamma}$ is 
the Perron--Frobenius operator introduced in \cite{IKW07}. Let $(\AA,\Gamma)$ be an iterated function system, let $A = C(\AA)$, and let $B(\AA)$ denote the bounded Borel functions on $\AA$. For $S \subseteq \AA$ we let $1_S$ denote the characteristic function of $S$. We introduce the following dual to the Hutchinson operator on Borel functions.

\begin{lem}\label{lem:dual_Hutchinson_op}
There is a well-defined bounded linear map $\Gamma^* \colon B(\AA) \to B(\AA)$ given by
\[
\Gamma^*(a)(y) \coloneqq \sum_{x \in \Gamma y} a(x) = \sum_{\gamma \in \Gamma} \frac{a(\gamma(y))}{b(\gamma(y),y)}, \quad a \in B(\AA),\, y \in \AA,
\]
with $\|\Gamma^*\| \le |\Gamma|$. If $a\colon \AA \to [0,\infty)$ is bounded and lower semicontinuous, then $\Gamma^*(a)$ is bounded, nonnegative, and lower semicontinuous. 
\end{lem}
\begin{proof}
The branch index $b$ is bounded and upper semicontinuous by Lemma~\ref{lem:branch_iff_discontinuous}, so the function $y \mapsto a(\gamma(y))/b(\gamma(y),y)$ is Borel and bounded by $\|a\|_{\infty}$ for all $\gamma \in \Gamma$. Hence, $\Gamma^*$ is well-defined and $\|\Gamma^*\| \le |\Gamma|$.

Now suppose that $a \colon \AA \to [0,\infty)$ is bounded and lower semicontinuous and $y_n \to y$ in $\AA$. Since $b$ is upper semicontinuous and
$b \ge 1$, we have $\liminf_n b(\gamma(y_n),y_n)^{-1} \ge b(\gamma(y),y)^{-1}$, and since
$a \ge 0$ is lower semicontinuous,
\[
\liminf_{n\to\infty}\frac{a(\gamma(y_n))}{b(\gamma(y_n),y_n)} \ \ge\ \frac{a(\gamma(y))}{b(\gamma(y),y)}
\qquad \text{for each } \gamma \in \Gamma.
\]
Summing over $\gamma \in \Gamma$ and using superadditivity of $\liminf$
gives $\liminf_n \Gamma^*(a)(y_n) \ge \Gamma^*(a)(y)$, so $\Gamma^*(a)$ is lower semicontinuous.
\end{proof}

\begin{dfn}
We call $\Gamma^*$ of Lemma~\ref{lem:dual_Hutchinson_op} the \emph{dual Hutchinson operator} of $(\AA,\Gamma)$.
\end{dfn}

Recall from Definition~\ref{dfn:n_step_forward_orbit} that $\Gamma^{(n)}y$ denotes the space of paths of length $n$ starting at $y \in \AA$. For $a \in B(\AA)$ and $y \in \AA$,
\[
(\Gamma^*)^2(a)(y) = \sum_{x_1 \in \Gamma y} \Gamma^*(a)(x_1) = \sum_{x_1 \in \Gamma y} \ \sum_{x_0 \in \Gamma x_1} a(x_0) = \sum_{(x_0,x_1,x_2) \in \Gamma^{(2)}y} a(x_0),
\]
and inductively
\begin{equation} \label{eq:n_step_dual}
(\Gamma^*)^n(a)(y) = \sum_{x = (x_0,\ldots,x_n) \in \Gamma^{(n)}y} a(x_0)
\end{equation}
for each $n \ge 1$. Since $\sum_{x \in \Gamma^{(n)}y} a(x_0)$ is a sum over paths, a summand $a(x_0)$ may appear more than once if there are multiple paths from $y$ to $x_0$. This is distinct from summing over $\Gamma^n y$, which may contain fewer elements than $\Gamma^{(n)}y$.

We record the following consequence of  Lemma~\ref{lem:dual_Hutchinson_op}.
\begin{lem}\label{lem:path_count_lsc}
For each $n \ge 1$ the function $y \mapsto |\Gamma^{(n)}y|$ is lower semicontinuous on $\AA$. 
\end{lem}
\begin{proof}
Taking $a = 1_{\AA}$ in \eqref{eq:n_step_dual} gives $(\Gamma^*)^n(1_{\AA})(y) = |\Gamma^{(n)}y|$ for all $y \in \AA$ and $n \ge 1$. An induction on $n$ using Lemma~\ref{lem:dual_Hutchinson_op} shows that $y \mapsto |\Gamma^{(n)}y|$ is lower semicontinuous. 
\end{proof}

Let $A_+$ denote the positive cone of $A$ and let $A^*$ and $A_+^*$ denote the bounded linear functionals and positive linear functionals on $A$, respectively.
Since $A$ is commutative, the Riesz--Markov--Kakutani Representation Theorem~\cite[Theorem 6.19]{Rud87} says that $A^*$ may be 
identified with the collection of Radon measures on $\AA$, which, since $\AA$ is compact and metrisable, is precisely the collection of finite Borel measures on $\AA$. The duality pairing $A^* \times A \to \CC$ is given by integration and we frequently identify bounded linear functionals with the associated measure: given a Radon measure $\mu$ on $\AA$ we write
$
\mu(a) = \int_{\AA} a\, d\mu
$
for $a \in A$.

The positive linear functionals $A_+^*$ correspond to positive Radon measures on $\AA$, and also to finite positive traces on $A$. In particular, the set of tracial states $T(A)$ on $A$ may be identified with the set of Borel probability measures on $\AA$. There is a partial order on $A^*$: for $\mu,\nu \in A^*$ we have $\mu \le \nu \iff \nu - \mu \in A^*_+$. This corresponds to the usual partial order on measures. 

\begin{lem}\label{lem:PF_operator}
	For each $\mu \in A^*$, there is a well-defined linear map $F(\mu) \colon B(\AA) \to \CC$ given by 
	\begin{equation}
		\label{eq:PF_operator}
	F(\mu)(a) =  \int_\AA \Gamma^*(a)(y) \, d\mu(y) =  \int_\AA \sum_{x \in \Gamma y}a(x) \, d\mu(y), \qquad a \in B(\AA).
	\end{equation}
	Then $F(\mu) \in A^*$, and $\mu \mapsto F(\mu)$ is a bounded linear map $F \colon A^* \to A^*$ with $F(A^*_+) \subseteq A^*_+$ and $\|F\| = \sup_{y \in \AA} |\Gamma y|$.
\end{lem}
\begin{proof}
Let $K \coloneqq \sup_{y \in \AA}|\Gamma y|$. By Lemma~\ref{lem:dual_Hutchinson_op}, for each $a \in B(\AA)$ the function $\Gamma^*(a)$ is bounded and Borel, with $|\Gamma^*(a)(y)| \le \sum_{x \in \Gamma y}|a(x)| \le K\|a\|_\infty$ for all $y \in \AA$. 
By the Riesz--Markov--Kakutani Theorem, every $\mu \in A^*$ is a finite complex Borel measure on $\AA$, so $F(\mu)$ given by \eqref{eq:PF_operator} is well-defined, linear in both $a$ and $\mu$, and satisfies $|F(\mu)(a)| \le K\|a\|_\infty\|\mu\|$. 
Hence, $F \colon A^* \to A^*$ is a bounded linear map with $\|F\| \le K$. If $\mu \in A^*_+$ and $a \ge 0$, then $\Gamma^*(a) \ge 0$, so $F(\mu)(a) \ge 0$. Finally, choosing $y \in \AA$ with $|\Gamma y| = K$ gives $\|F(\delta_y)\| \ge F(\delta_y)(1_\AA) = |\Gamma y| = K$, so $\|F\| = K$.
\end{proof}
\begin{dfn}[{cf.~\cite{IKW07}}]\label{dfn:PF}
	We call $F$ of Lemma~\ref{lem:PF_operator} the \emph{Perron--Frobenius operator} of $(\AA,\Gamma)$.
\end{dfn}
The operator $F$ is related to the operator $\mu \mapsto \sum_{\gamma \in \Gamma} \omega_{\gamma} \gamma_* (\mu)$ of Hutchinson \cite[\S~4.2(1)]{Hut81}, where $\omega_\gamma$ are positive weights with $\sum_{\gamma \in \Gamma} \omega_\gamma = 1$ and $ \gamma_* (\mu) (a) = \mu(a \circ \gamma)$ for all $a \in C(\AA)$. The distinctions are that $F$ does not depend on a labelling of $\Gamma$ and that the weights ${b(\gamma(y),y)}^{-1}$ vary as a function of $y$ and do not necessarily sum to $1$. Fan and Lau \cite{FL99} study operators of this form with variable, unnormalised weights, and prove a Ruelle--Perron--Frobenius theorem when the weights are Dini continuous \cite[Theorem~1.1]{FL99}. By Lemma~\ref{lem:branch_iff_discontinuous}, the weights $y \mapsto b(\gamma(y),y)^{-1}$ are all continuous if and only if $B_\Gamma = \varnothing$, so the result of Fan--Lau is unavailable in the presence of branch points. The operator $F$ is also not weak-$*$ continuous in general.

\begin{example}\label{ex:not_weak_star_continuous}
	 Let $([0,1],\Gamma = \{\gamma_1,\gamma_2\})$ be the iterated function system with $\gamma_1 (x) = \frac{x}{2}$ and $\gamma_2(x) = 1- \frac{x}{2}$. The sequence of point masses $(\delta_{1-1/n})_{n=1}^{\infty}$  converges in the weak-$*$ topology to $\delta_1$.
	For each $a \in A$,
	\[
	F(\delta_{1-1/n})(a) = a\Big(\frac{1}{2} + \frac{1}{2n}\Big) + a\Big(\frac{1}{2} - \frac{1}{2n}\Big)  \to 2a\Big(\frac{1}{2}\Big),
	\]
	while $F(\delta_1) (a) = a(\frac{1}{2})$. So $F$ is not weak-$*$ continuous.
\end{example}

We record some useful identities for $F$. We identify evaluating $\mu \in A^*_+$ on $1_S$ with the measure of $S$ under $\mu$. 
For each $\mu \in A^*_+$ and Borel set $S \subseteq \AA$, we have
\begin{equation}\label{eq:PF_characteristic}
F(\mu)(S) = F(\mu)(1_S) = \int_\AA \sum_{x \in \Gamma y} 1_S(x) ~d\mu(y) = \int_\AA |\Gamma y \cap S| ~d\mu(y).
\end{equation}
Since $F(\mu)(a) = \mu(\Gamma^*(a))$ for all $\mu \in A_+^*$ and $a \in B(\AA)$, iterating gives $F^n(\mu)(a) = \mu\big((\Gamma^*)^n(a)\big)$, so integrating \eqref{eq:n_step_dual} against $\mu$ yields
 \begin{equation} \label{eq:n_step_PF}
  F^n(\mu)(a) = \int_\AA \sum_{x \in \Gamma^{(n)}x_n} a(x_0) ~d\mu(x_n)
 \end{equation}
for all $\mu \in A_+^*$, $a \in A$ and $n \ge 1$.

\subsection{KMS states and the Laca--Neshveyev machine}\label{subsec:KMS_LN}

The notion of a $\KMS$ state originates in quantum statistical mechanics via the work of Kubo~\cite{Kub57} and Martin--Schwinger~\cite{MS59}, where it characterises thermal equilibrium states at inverse temperature $\beta > 0$.  In the $C^*$-algebraic formulation, given a $C^*$-algebra $B$ and a strongly continuous one-parameter group of automorphisms $\alpha \colon \mathbb{R} \to \Aut(B)$, a state $\varphi$ on $B$ is a \emph{$\KMS_\beta$ state} for $\alpha$ if it satisfies the twisted trace condition
\[
\varphi(a \, \alpha_{i\beta}(b)) = \varphi(ba)
\]
for all $a, b$ in a suitable dense $\alpha$-invariant $*$-subalgebra of $B$. If the action $\alpha$ is clear, we simply say that $\varphi$ is a $\KMS_\beta$ state.
We refer the reader to \cite{BR97b} for a comprehensive treatment.  In this article we are concerned with $\KMS$ states for the gauge action on Kajiwara--Watatani algebras.

In \cite{LN04}, Laca and Neshveyev established a general framework for characterising $\KMS$ states on Toeplitz--Pimsner and Cuntz--Pimsner algebras for a class of $\RR$-actions lifted from $\RR$-actions on the underlying $C^*$-correspondence. 
They showed that such $\KMS$ states are characterised entirely by certain dynamically-compatible traces on the coefficient algebra of the $C^*$-correspondence.
 
Let $(\lambda,X)$ be a countably generated, injective, nondegenerate $C^*$-correspondence over a $C^*$-algebra $A$. Given a  finite trace $\tau$ on $A$, there is an induced semifinite trace $\Tr_{\tau}$ on $\End_A(X)$ given by
\[
\Tr_{\tau}(T) = \sum_{i =1}^{\infty} \tau(\langle u_i \mid T u_i \rangle_A),
\]
whose value is independent of a choice of countable frame $(u_i)_{i=1}^\infty$ for $X$ \cite[Theorem 1.1]{LN04}. If $\Tr_{\tau}(\id_X) < \infty$, then $\Tr_{\tau} \circ \lambda$ is again a finite trace on $A$. Now suppose  that the dynamics on $\Tt_X$ arises from the gauge action $\alpha \colon \RR \to \Aut(\Tt_X)$. A key observation of Laca--Neshveyev is that there is a bijective correspondence between $\KMS_\beta$ states on $\Tt_X$ for $\alpha$ and tracial states $\tau$ on $A$ such that $\Tr_{\tau} (\lambda(a)) \le e^{\beta} \tau(a)$ for all $a \in A_+$. Moreover, the $\KMS_\beta$ state associated to such a trace $\tau$ descends to a $\KMS_\beta$ state on $\Oo_X$ if and only if $\Tr_{\tau}( \lambda(a) )= e^{\beta} \tau(a)$ for all $a \in J_X$.

 Izumi--Kajiwara--Watatani adapted the machinery of Laca--Neshveyev to the gauge action on $\Oo_{\Gamma}$ for iterated function systems in which the set $\{\gamma(y) \mid y \in \AA,\ \gamma(y) = \gamma'(y) \text{ for some } \gamma \ne \gamma'\}$ is finite and the system satisfies the following orbit escape condition.
\begin{dfn}[{cf.~\cite[Theorem~6.6]{IKW07}}]\label{dfn:orbit_escape_condition}
For each $x \in \AA$ let $O(x) \coloneqq \bigcup_{n=0}^{\infty} \Gamma ^n x$ denote the \emph{forward orbit} of $x$.
We say that an iterated function system $(\AA,\Gamma)$ satisfies the \emph{orbit escape condition} if one of the following conditions holds:
\begin{enumerate}
	\item whenever $\gamma(x) = \gamma'(x)$ for some $x \in \AA$, then $\gamma = \gamma'$; or 
	\item if $\gamma(x) = \gamma'(x)$ for some $x \in \AA$ and $\gamma \ne \gamma'$, then there exists $y \in O(x)$ such that for all $z \in O(y)$, we have $\gamma_1(z) \ne \gamma_2(z)$ for all $\gamma_1 \ne \gamma_2$ in $\Gamma$.
\end{enumerate}
\end{dfn}

\begin{example}\label{ex:clarence_no_oec}
	Let $([0,1], \Gamma)$ be the iterated function system of Example~\ref{ex:clarence}. Since $\gamma_2(\AA) = [1/2,1]$, and for every $x \in [1/2,1]$ we have $\gamma_1(x) = \gamma_3(x)$, the system does not satisfy the orbit escape condition. As such, the results of \cite{IKW07} do not apply to this system.
\end{example}
We aim to characterise $\KMS$ states on $\Oo_{\Gamma}$ for a more general class of iterated function systems that do not necessarily satisfy the orbit escape condition. Once again, let $A = C(\AA)$. A key observation of \cite{IKW07} is that for any $\mu \in T(A)$, the trace $F(\mu)$ agrees with the induced trace $\Tr_{\mu} \circ \lambda$ arising from the Kajiwara--Watatani correspondence $(\lambda,X_\Gamma)$. We include a streamlined proof for completeness and refer to \cite{FL02} for an introduction to the theory of frames for Hilbert $C^*$-modules, which we use in the proof. 

Since $\AA$ is second countable, compact and Hausdorff, $\gr(\Gamma)\subseteq \AA \times \AA$ is compact metrisable, so $C(\gr(\Gamma))$ is separable in the supremum norm. Since the module norm on $X_\Gamma$ is equivalent to the supremum norm, any countable dense subset of $X_{\Gamma}$ generates it as a Hilbert $A$-module. By Kasparov's Stabilisation Theorem \cite[Theorem~6.2]{Lan95}, there is a countable (normalised tight) frame $(u_i)_{i=1}^{\infty}$ for $X_\Gamma$ (see~\cite[Example~3.5]{FL02}).

\begin{prop}[{\cite[Theorem 4.2]{IKW07}}]\label{prop:IKW_simplification}
 	Let $(\AA,\Gamma)$ be an iterated function system and let $(u_i)_{i=1}^\infty$ be a countable (normalised tight) frame for $X_\Gamma$. 
 	Then for each $a \in A = C(\AA)$,
 	\[
 	\sum_{i=1}^\infty \langle u_i \mid \lambda(a) u_i \rangle_A(y) = \sum_{x \in \Gamma y} a(x) = \Gamma^*(a)(y), \quad \text{ for all } y \in \AA.
 	\]
 	Moreover, $
 	F(\mu) = \Tr_{\mu} \circ \lambda
 	$
 	for all $\mu \in T(A)$. 
 \end{prop}
 \begin{proof}

 	Fix $(x,y) \in \gr(\Gamma)$. Use Urysohn's Lemma to choose $\psi \in X_\Gamma$  such that $\psi(x,y) = 1$ and $\psi(x',y) = 0$ for all $x' \in \Gamma y$ with $x' \ne x$. Since $(u_i)_{i=1}^{\infty}$ is a frame, $\psi = \sum_{i=1}^{\infty} u_i \cdot \langle u_i\mid \psi \rangle_A$ with convergence in norm. It follows that $\psi(x,y) = 1$ is approximated, for large $n$, by
 	\begin{align*}
 		&\quad \sum_{i=1}^{n} u_i(x,y) \langle u_i \mid \psi \rangle_A(y)
 		= \sum_{i=1}^{n} \sum_{x' \in \Gamma y} b(x',y) {u_i(x,y)}\ol{ u_i(x',y)}\psi(x',y)\\
 		&\quad= \sum_{i=1}^{n}\Big( b(x,y)|u_i(x,y)|^2\psi(x,y) + \sum_{\substack{x' \in \Gamma y\\x' \ne x}} b(x',y){u_i(x,y)} \ol{u_i(x',y)}\psi(x',y) \Big)
 		= \sum_{i=1}^{n} b(x,y)|u_i(x,y)|^2.
 	\end{align*}
 	Consequently, $1 = \lim_{n \nearrow \infty} \sum_{i=1}^{n} b(x,y)|u_i(x,y)|^2$. The convergence is absolute, so for each $a \in A$ and $y \in \AA$,
 	\[
 	\sum_{x \in \Gamma y} a(x) = \sum_{x \in \Gamma y} \sum_{i=1}^{\infty} b(x,y)|u_i(x,y)|^2a(x) = \sum_{i=1}^{\infty} \langle u_i \mid \lambda(a)u_i \rangle_A(y).
 	\]
 	For the second statement, note that
	the partial sums $\sum_{i\le n}\langle u_i\mid\lambda(a)u_i\rangle_A(y)$ are bounded by $\sum_{x\in\Gamma y}|a(x)|\le|\Gamma|\|a\|_\infty$.
	The Dominated Convergence Theorem implies that for all $\mu \in T(A)$ and $a \in A$,
 	\[
 	\Tr_\mu(\lambda(a)) = \sum_{i=1}^\infty \int_\AA \langle u_i \mid \lambda(a) u_i \rangle_A(y) \, d\mu(y) = \int_\AA  \sum_{x \in \Gamma y} a(x) \, d\mu(y) = F(\mu)(a). \qedhere
 	\] 	
 \end{proof}

 By Proposition~\ref{prop:IKW_simplification}, $\Tr_\mu\circ\lambda = F(\mu)$, so for $\mu\in T(A)$ the Laca--Neshveyev subinvariance condition, $\Tr_\mu(\lambda(a))\le e^{\beta}\mu(a)$ for all $a\in A_+$, is precisely the condition $F(\mu)\le e^{\beta}\mu$ in the order on $A^*_+$. Writing $F_\beta\coloneqq e^{-\beta}F$, we may rewrite this as $F_\beta(\mu)\le\mu$.

The following characterisation of $\KMS$ states on $\Oo_{\Gamma}$ applies generally and is adapted from \cite[Theorem 3.2]{IKW07}; it follows from \cite{LN04}.

\begin{thm}[{cf.~\cite[Theorems~2.1 and 2.5]{LN04}}]\label{thm:LN_machine}
Let $(\lambda,X_\Gamma)$ be the Kajiwara--Watatani correspondence associated to an iterated function system $(\AA,\Gamma)$ and let $\beta > 0$. There is an affine isomorphism between $\KMS_\beta$ states on $\Tt_{\Gamma}$ for the gauge action and the set of $\mu \in T(A)$ satisfying
\begin{equation}\label{eq:subinvariance}
	F_\beta(\mu) \le \mu.
\end{equation}
Moreover, the $\KMS_\beta$ state on $\Tt_{\Gamma}$ associated to $\mu$ descends to a $\KMS_\beta$ state for the gauge action on $\Oo_{\Gamma}$ if and only if 
\begin{equation}\label{eq:trace_on_ideal}
F_{\beta}(\mu)(a) = \mu(a) \quad \text{ for all } a \in J_{\Gamma}.
\end{equation}
In either case, the correspondence associates to a $\KMS_\beta$ state $\varphi$ the tracial state $\mu\coloneqq \varphi|_A$. 
\end{thm}

\begin{dfn}
For each $\beta > 0$ we say that a finite positive trace $\nu \in A^*_+$ is $\beta$-\emph{subinvariant} if $F_{\beta}(\nu) \le \nu$, as in \eqref{eq:subinvariance}, and we denote the collection of $\beta$-subinvariant tracial states on $A = C(\AA)$ by $T_{\beta}(A)$. Following \cite[Definition~3.3]{IKW07}, we say that a finite positive trace $\nu \in A^*_+$ is of
\begin{enumerate}
\item \emph{infinite type} if $F_{\beta}(\nu) = \nu$, and of
\item \emph{finite type} if there exists $\nu' \in A^*_+$ such that $\nu = \sum_{n=0}^{\infty} F_{\beta}^n(\nu')$ with convergence in the weak-$*$ topology.
\end{enumerate}
We let $T_{\beta}^{\infty}(A)$ and $T_{\beta}^{f}(A)$ denote the collections of tracial states on $A$ of infinite and of finite type, respectively. A trace of either type is automatically $\beta$-subinvariant, so $T_{\beta}^{\infty}(A) \cup T_{\beta}^{f}(A) \subseteq T_{\beta}(A)$.
A $\KMS_\beta$ state on $\Tt_{\Gamma}$ or $\Oo_{\Gamma}$ is of infinite type (resp. finite type) if the associated $\beta$-subinvariant tracial state is of infinite type (resp. finite type). We write $\KMS_\beta^{\infty}(\Oo_{\Gamma})$ and $\KMS_\beta^f(\Oo_{\Gamma})$ for the collection of infinite-type and finite-type $\KMS_\beta$ states of $\Oo_{\Gamma}$ for the gauge action, and do similarly for $\Tt_{\Gamma}$.
\end{dfn}

In the proof of \cite[Proposition 2.4]{LN04} it is shown that each $\mu \in T_{\beta}(A)$ can be decomposed as $\mu = \mu^{\infty} + \mu^{f}$ for unique $\beta$-subinvariant traces $\mu^{\infty},\mu^{f} \in A^*_+$ of infinite and of finite type, respectively. This divides the problem of classifying $T_{\beta}(A)$ into classifying the $\beta$-subinvariant traces of infinite type and those of finite type. More explicitly (see \cite[p.~1894]{IKW07}), for $\mu \in T_{\beta}(A)$ we have 
\begin{equation}\label{eq:infinite_finite_decomposition}
\mu^{\infty} = \lim_{n \to \infty} F^{n}_\beta(\mu) \quad \text{ and } \quad
\mu^{f} =  \lim_{n \to \infty} \sum_{k=0}^{n} F^k_{\beta}(\mu_0) \quad \text{ where } \mu_0 \coloneqq \mu - F_\beta(\mu),
\end{equation}
with both limits taken in the weak-$*$ topology.

If $\mu \in T_{\beta}^{\infty}(A)$, then \eqref{eq:trace_on_ideal} holds automatically, so the associated $\KMS_\beta$ state descends to $\Oo_{\Gamma}$. For the $\KMS_\beta$ state associated to $\mu \in T_{\beta}^{f}(A)$ to descend to $\Oo_{\Gamma}$, the positive linear functional $\mu_0 = \mu - F_{\beta}(\mu)$ must vanish on $J_{\Gamma}$.
Following \cite{IKW07} we let $T(A/J_{\Gamma})$ denote the tracial states on $A$ that vanish on $J_{\Gamma}$, which correspond to tracial states on $A/J_{\Gamma}$, and let
\[
T(A/J_{\Gamma})_\beta \coloneqq \Big\{ \tau \in T(A/J_{\Gamma}) : \sum_{n=0}^\infty F_\beta^n(\tau) \text{ converges in the weak-$*$ topology}\Big\}.
\]
The following result parametrises finite-type $\KMS$ states on $\Oo_{\Gamma}$.
\begin{prop}[{\cite[Lemma~3.4]{IKW07}}]\label{prop:IKW_seed_bijection}
Let $\beta > 0$. For each $\tau \in T(A/J_{\Gamma})_\beta$ let $\varphi_{\tau,\beta}$ be the $\KMS_\beta$ state on $\Oo_{\Gamma}$ satisfying $\varphi_{\tau,\beta}|_A = m_{\tau,\beta}\sum_{n=0}^\infty F_\beta^n(\tau)$, where $m_{\tau,\beta} \coloneqq \big(\sum_{n=0}^\infty F_\beta^n(\tau)(1_\AA)\big)^{-1}$ is a normalising constant. Then $\tau \mapsto \varphi_{\tau,\beta}$ is a bijection between $T(A/J_{\Gamma})_\beta$ and $\KMS_\beta^f(\Oo_{\Gamma})$.
This bijection restricts to a bijection between the extreme points $\operatorname{ex}(T(A/J_{\Gamma})_\beta)=T(A/J_{\Gamma})_{\beta} \cap \operatorname{ex}(T(A/J_{\Gamma}))$ and $\operatorname{ex}(\KMS_\beta^f(\Oo_{\Gamma}))$.
\end{prop}
The bijection of Proposition~\ref{prop:IKW_seed_bijection} need not be affine because the normalising constant $m_{\tau,\beta}$ depends on $\tau$.

\section{Markov partitions}\label{sec:markov_partitions}
 To facilitate the analysis of $\KMS$ states on $\Oo_{\Gamma}$ we introduce the notion of a Markov partition for an iterated function system. A Markov partition is a Borel partition of $\AA$ that is compatible with the dynamics of $\Gamma$. We use this compatibility to attain better control over the Perron--Frobenius operator.
\subsection{The fundamental partition} We first define the fundamental partition: the coarsest partition of $\AA$ such that, on each of its sets, any two maps $\gamma, \gamma' \in \Gamma$ either agree at every point or agree at no point.

Fix an iterated function system $(\AA,\Gamma)$. 
Let $\Pp$ denote the collection of partitions of $\Gamma$. Recall that if $P_1,P_2 \in \Pp$ then $P_1$ is a refinement of $P_2$ if for every $\Lambda_1 \in P_1$ there exists $\Lambda_2 \in P_2$ such that $\Lambda_1 \subseteq \Lambda_2$. Refinement induces a partial order on $\Pp$ by $P_1 \le P_2$ if $P_2$ refines $P_1$. With this ordering $\Pp$ is a lattice with maximal element $\{\{\gamma\} \mid \gamma \in \Gamma \}$ and minimal element $\{\Gamma\}$. 

We introduce the following strata of $\AA$ associated to each partition of $\Gamma$.

\begin{dfn}\label{dfn:strata}
For each $P \in \Pp$ define
\begin{align*}
S_P 
&\coloneqq \{ x \in \AA \mid \gamma(x) = \gamma'(x) \text{ for all } \gamma,\gamma' \in \Lambda \text{ and } \Lambda \in P \} \quad \text{and}\\
\AA_{P} &\coloneqq \{x \in \AA \mid \gamma(x) = \gamma'(x) \text{ if and only if } \gamma,\gamma' \in \Lambda
\text{ for some } \Lambda \in P\}.
\end{align*}
\end{dfn}

For $x \in \AA$ let $P_x \in \Pp$ denote the partition of $\Gamma$ into the classes
of the equivalence relation defined by $\gamma \sim_x \gamma'$ if and only if
$\gamma(x) = \gamma'(x)$. Then $x \in \AA_P$ if and only if $P_x = P$, and $x \in S_P$
if and only if every $\Lambda \in P$ is contained in a single $\sim_x$-class, that is,
if and only if $P_x \le P$. Hence,
\begin{equation}\label{eq:strata_via_Px}
\AA_P = \{x \in \AA \mid P_x = P\} \quad \text{and} \quad
S_P = \{x \in \AA \mid P_x \le P\}.
\end{equation}
In particular, the sets $\AA_P$ for $P \in \Pp$ are pairwise disjoint with union $\AA$.

\begin{lem}\label{lem:S_P}
Let $P \in \Pp$. Then
\begin{enumerate}[label=(\roman{enumi})]
\item\label{item:S_P} $S_P = \bigcap_{\Lambda \in P} p_2(G_\Lambda)$ and $S_P$ is compact;
\item\label{item:S_P_union} $S_P = \bigsqcup_{P' \le P} \AA_{P'}$, so if $P_1 \le P_2$, then $S_{P_1} \subseteq S_{P_2}$; and
\item\label{item:AA_P} $\AA_P = S_P \setminus \bigcup_{P' < P} S_{P'}$.
\end{enumerate}
\end{lem}

\begin{proof}
For $\Lambda \subseteq \Gamma$ we have $(x,y) \in G_\Lambda$ if and only if $x = \gamma(y)$ for every $\gamma \in \Lambda$. Hence $p_2(G_\Lambda) = \{y \in \AA \mid \gamma(y) = \gamma'(y) \text{ for all } \gamma,\gamma' \in \Lambda\}$, and intersecting over $\Lambda \in P$ gives the description of $S_P$ in \ref{item:S_P}. Since $G_\Lambda$ is a closed subset of the compact space $\gr(\Gamma)$ and $p_2$ is continuous, each $p_2(G_\Lambda)$ is compact, so $S_P$ is compact.
Both \ref{item:S_P_union} and \ref{item:AA_P} follow from \eqref{eq:strata_via_Px}.
\end{proof}

\begin{prop}\label{prop:fundamental_partition}
The collection $\Aa \coloneqq \{\AA_{P} \mid P \in \Pp \text{ such that } \AA_{P} \ne \varnothing\}$
is a Borel partition of $\AA$.
\end{prop}
\begin{proof}
The description \eqref{eq:strata_via_Px} implies $\Aa$ is a partition of $\AA$. Each $S_P$ is compact by Lemma~\ref{lem:S_P}\ref{item:S_P}, hence Borel, so
Lemma~\ref{lem:S_P}\ref{item:AA_P} shows that each $\AA_P$ is Borel.
\end{proof}

\begin{dfn} \label{dfn:fundamental_partition}
We call the partition $\Aa$ of Proposition~\ref{prop:fundamental_partition} the \emph{fundamental partition of $(\AA,\Gamma)$}. 
\end{dfn}

\subsection{Markov partitions} The following definition of a Markov partition is inspired by Markov partitions for single-map dynamical systems (cf.~\cite[\S~6.5]{LM21}). They allow us to approximate the dynamics of an iterated function system with a symbolic system. 
\begin{dfn}\label{dfn:markov_partition}
	A \emph{Markov partition} of $(\AA,\Gamma)$ is a Borel partition $\Mm = \{M_i\}_{i \in I}$ of $\AA$ such that
	\begin{enumerate}[label=(M\arabic{enumi})]
	 \item\label{itm:E_refines_A} $\Mm$ refines $\Aa$, and 
	 \item\label{itm:markov_property} for each $\gamma \in \Gamma$ and $i \in I$ there exists a necessarily unique $j \in I$ such that $\gamma(M_i) \subseteq M_j$.
	 \end{enumerate} 
\end{dfn}

Condition \ref{itm:E_refines_A} ensures that for each $M_i \in \Mm$ there exists $P \in \Pp$ such that $M_i \subseteq \AA_P$. Together, \ref{itm:E_refines_A} and \ref{itm:markov_property} ensure the Markov property. Specifically, let $(\gamma_1, \ldots, \gamma_{n+1})$ be a sequence in $\Gamma$ and let $x \in \AA$. If $\gamma_n \circ \cdots \circ \gamma_1 (x) \in M_{i_n}$ for some $i_n \in I$, then $M_{i_{n+1}}$---the unique element of $\Mm$ containing $\gamma_{n+1} \circ \cdots \circ \gamma_1(x)$---is determined by $i_n$ and $\gamma_{n+1}$ alone. 
Unlike the classical notion, we do not require the $M_i$ to have nonempty interior; such sets play an important role below.

 Trivially, the collection of singletons $\{\{x\} \}_{x \in \AA}$ is a Markov partition of $(\AA,\Gamma)$, but in practice this is too fine. It is more useful to consider Markov partitions that are as coarse as possible. 
\begin{dfn}
A Markov partition $\Mm$ of $(\AA,\Gamma)$ is said to be \emph{minimal} if it is not a strict refinement of another Markov partition. 
\end{dfn}

In Theorem~\ref{thm:finite_iff_Theta_const} below we show that $(\AA,\Gamma)$ always admits a unique minimal Markov partition, and that every Markov partition refines it. 

\begin{example}\label{ex:clarence1}
Consider the contractive iterated function system $([0,1],\Gamma = \{\gamma_1,\gamma_2,\gamma_3\})$ of Example~\ref{ex:clarence}.
The fundamental partition of $([0,1],\Gamma)$ is $\Aa = \{ [0,1/2), [1/2,1), \{1\}\}$. One can check that the refinement $\Mm \coloneqq \{ \{0\}, (0,1/2), [1/2,1), \{1\}\}$ is a Markov partition of $([0,1],\Gamma)$. In fact, $\Mm$ is the unique minimal Markov partition of $([0,1],\Gamma)$ as shown in Example~\ref{ex:clarence_minimal} below.
\end{example}

\begin{example}\label{ex:Markov_partition_trivial}
Suppose that $(\AA,\Gamma)$ is an iterated function system such that $B_{\Gamma}= \varnothing$. Then $\partial G_\Lambda = \varnothing$ for all $\Lambda \subseteq \Gamma$, so each $G_{\Lambda}$ is clopen in $\gr(\Gamma)$. By Lemmas~\ref{lem:p2_open} and~\ref{lem:S_P}\ref{item:S_P}, the sets $S_P$ are clopen for each $P \in \Pp$, so $\AA_P = S_P \setminus \bigcup_{P' < P} S_{P'}$ is clopen for each $P \in \Pp$ by Lemma~\ref{lem:S_P}\ref{item:AA_P}.

Consequently, if $\AA$ is connected, then $\Aa = \{\AA\}$ is the unique minimal Markov partition of $(\AA,\Gamma)$. If $\AA$ is disconnected (as in Example~\ref{ex:cantor_not_hutchinson} below) then the minimal Markov partition may contain more sets. 
\end{example} 

We record some fundamental properties of Markov partitions. 
The following is an immediate consequence of \ref{itm:E_refines_A}, \ref{itm:markov_property}, and the definition of $\AA_P$. 
\begin{lem}\label{lem:markov_unique_image}
Let $\Mm = \{M_i\}_{i \in I}$ be a Markov partition of $(\AA,\Gamma)$ and for each $i \in I$ let $P \in \Pp$ be such that $M_i \subseteq \AA_P$. Then for each $\Lambda \in P$ there is a unique $j \in I$ such that $\gamma(M_i) \subseteq M_j$ for all $\gamma \in \Lambda$. 
\end{lem}

The most important property of a Markov partition is that the number of paths of a given length issuing from a point is constant on each set.

\begin{lem}\label{lem:markov_same_paths}
 Let $\Mm = \{M_i\}_{i \in I}$ be a Markov partition of $(\AA,\Gamma)$. Then for all $i \in I$ and $x,y \in M_i$ we have $|\Gamma^{(n)} x| = |\Gamma^{(n)} y|$ for all $n \in \NN$.
\end{lem}
\begin{proof}
Fix $i \in I$ and $x,y \in M_i$. By \ref{itm:E_refines_A} there exists $P \in \Pp$ such that $M_i \subseteq \AA_P$. Then 
\[
|\Gamma x| = |\{\gamma(x) \colon \gamma \in \Gamma\}| = |P| = |\Gamma y|.
\]
 We have $\gamma(x) = \gamma'(x)$ and $\gamma(y) = \gamma'(y)$ for all $\Lambda \in P$ and $\gamma, \gamma' \in \Lambda$. Hence, Lemma~\ref{lem:markov_unique_image} implies that for each $\Lambda \in P$ there is a unique $j_{\Lambda} \in I$ such that $\gamma(x), \gamma(y) \in M_{j_{\Lambda}}$ for all $\gamma \in \Lambda$. It follows inductively that $|\Gamma^{(n)} x| = |\Gamma^{(n)} y|$ for all $n \in \NN$.
\end{proof}
With Lemma~\ref{lem:markov_same_paths} in mind we abuse cardinality notation and define, for each $n \in \NN$,
\[
|\Gamma^{(n)}M_i| \coloneqq |\Gamma^{(n)}x|, \quad \text{ for some } x \in M_i.
\]

\subsection{Constructing Markov partitions}\label{subsec:constructing_markov_partitions}

Denote by $\Rfr_{\Aa}$ the collection of Borel partitions of $\AA$ refining $\Aa$.
For $\Mm \in \Rfr_{\Aa}$ we outline a process to construct a new partition $\Theta(\Mm) \in \Rfr_{\Aa}$ refining $\Mm$. In Theorem~\ref{thm:finite_iff_Theta_const} below, we show that if repeating this process stabilises, it results in a Markov partition.

Fix $\Mm = \{M_i\}_{i \in I} \in \Rfr_{\Aa}$.
For each $P \in \Pp$, $\Lambda \in P$, and $S \subseteq \AA_{P}$ let 
\[
\Theta_{P,\Lambda}(S) \coloneqq \{ S \cap \gamma^{-1}(M_j) \mid j \in I \text{ and } \gamma(S) \cap M_j \ne \varnothing \},
\]
where $\gamma$ is chosen to be any map belonging to $\Lambda$ (any two $\gamma,\gamma' \in \Lambda$ agree on $\AA_P \supseteq S$, so the definition does not depend on the choice). Since $\Mm$ is a partition of $\AA$, and intersections and unions distribute over preimages, $\Theta_{P,\Lambda}(S)$ is a partition of $S$. If $\gamma(S)$ meets a unique $M_j$, then $\gamma(S)\subseteq M_j$, and so $\Theta_{P,\Lambda}(S) = \{S\}$.

The construction of $\Theta_{P,\Lambda}(S)$ extends to partitions. Indeed, if $\Nn = \{N_k\}_{k \in K}$ is a partition of $S \subseteq \AA_{P}$, then
\[
\Theta_{P,\Lambda}(\Nn) \coloneqq \bigcup_{k \in K} \Theta_{P,\Lambda}(N_k)
\]
is also a partition of $S$ that refines $\Nn$. Observe that if $\Ll \ge \Nn$, then $\Theta_{P,\Lambda}(\Ll) \ge \Theta_{P,\Lambda}(\Nn)$.

\begin{lem}\label{lem:Theta_PL_commute} Fix $P \in \Pp$ and $\Lambda, \Lambda' \in P$.
If $\Nn$ is a partition of $S \subseteq \AA_P$, then $\Theta_{P,\Lambda} \circ \Theta_{P,\Lambda'}(\Nn) = \Theta_{P,\Lambda'} \circ \Theta_{P,\Lambda}(\Nn)$, and this partition is the common refinement of $\Theta_{P,\Lambda}(\Nn)$ and $\Theta_{P,\Lambda'}(\Nn)$.
\end{lem}
\begin{proof}
It suffices to verify the claim for $\Nn = \{S\}$. Choose $\gamma \in \Lambda$ and $\gamma' \in \Lambda'$. The elements of $\Theta_{P,\Lambda} \circ \Theta_{P,\Lambda'}(\{S\})$ are the nonempty sets of the form $S \cap (\gamma')^{-1}(M_j) \cap \gamma^{-1}(M_k)$ for $j,k \in I$. This collection is symmetric in $(\Lambda, \Lambda')$ and is precisely the common refinement of $\Theta_{P,\Lambda}(\{S\})$ and $\Theta_{P,\Lambda'}(\{S\})$.
\end{proof}

\begin{dfn}\label{dfn:Theta}
Suppose that $\Mm = \{M_i\}_{i \in I} \in \Rfr_{\Aa}$, and for each $P \in \Pp$ enumerate $P = \{ \Lambda_1,\ldots,\Lambda_{|P|}\}$. For each $i \in I$ let $P_i \in \Pp$ be the unique partition such that $M_i \subseteq \AA_{P_i}$. Define
\begin{equation}\label{eq:Theta}
\Theta(\Mm) \coloneqq \bigcup_{i \in I} \Theta_{P_i, \Lambda_{|P_i|}} \circ \cdots \circ \Theta_{P_i,\Lambda_1} (M_i).
\end{equation}
\end{dfn}
By Lemma~\ref{lem:Theta_PL_commute}, the composition $\Theta_{P_i,\Lambda_{|P_i|}} \circ \cdots \circ \Theta_{P_i,\Lambda_1}(M_i)$ is the common refinement of the partitions $(\Theta_{P_i,\Lambda}(M_i))_{\Lambda \in P_i}$, so $\Theta(\Mm)$ is independent of the choice of enumerations of the $P_i$.
The idea is that $\Theta(\Mm)$ separates, within each $M_i$, points that some $\Lambda\in P_i$ sends into different elements of $\Mm$, bringing the partition closer to satisfying the Markov property.

\begin{lem}\label{lem:Theta}
The formula \eqref{eq:Theta} defines a map $\Theta \colon \Rfr_{\Aa} \to \Rfr_{\Aa}$ such that
\begin{enumerate}[label=(\roman{enumi})]
\item\label{itm:theta_refines} $\Theta(\Mm) \ge \Mm$ for each $\Mm \in \Rfr_{\Aa}$; and
\item\label{itm:theta_order} if $\Nn \ge \Mm \in \Rfr_{\Aa}$, then $\Theta(\Nn) \ge \Theta(\Mm)$. 
\end{enumerate}
Moreover, a partition $\Mm \in \Rfr_{\Aa}$ is Markov if and only if $\Theta(\Mm) = \Mm$. 

\end{lem}
\begin{proof}

Let $\Mm = \{M_i\}_{i \in I} \in \Rfr_{\Aa}$. By the discussion preceding Lemma~\ref{lem:Theta_PL_commute} the set $\Theta_{P_i, \Lambda_{|P_i|}} \circ \cdots \circ \Theta_{P_i,\Lambda_1} (M_i)$ is a partition of $M_i$. Moreover, the partition is  Borel as $M_i$ is Borel. Since $\Mm$ partitions $\AA$, $\Theta(\Mm)$ is a Borel partition of $\AA$ refining $\Mm$.

Suppose that $\Nn = \{N_k\}_{k \in K} \in \Rfr_{\Aa}$ refines $\Mm$. For each $k \in K$ let $i(k) \in I$ be the unique index such that $N_k \subseteq M_{i(k)}$. Since $\Nn$ refines $\Mm$, for each $k \in K$, $N_k \subseteq \AA_{P_{i(k)}}$. 
Fix $N \in \Theta(\Nn)$, choose $k \in K$ such that $N \subseteq N_k$, and for each $\Lambda \in P_{i(k)}$ fix $\gamma_{\Lambda} \in \Lambda$. By Lemma~\ref{lem:Theta_PL_commute} the partition of $N_k$ contributing to $\Theta(\Nn)$ in \eqref{eq:Theta} is the common refinement of the partitions $\Theta_{P_{i(k)},\Lambda}(N_k)$ for $\Lambda \in P_{i(k)}$, so there exist $l_{\Lambda} \in K$ such that
\[
N = N_k \cap \bigcap_{\Lambda \in P_{i(k)}} \gamma_{\Lambda}^{-1}(N_{l_{\Lambda}}).
\]
Since $N_k \subseteq M_{i(k)}$, and $N_{l_{\Lambda}} \subseteq M_{i(l_{\Lambda})}$ for each $\Lambda \in P_{i(k)}$, we have
\[
\varnothing \ne N \subseteq M_{i(k)} \cap \bigcap_{\Lambda \in P_{i(k)}} \gamma_{\Lambda}^{-1}(M_{i(l_{\Lambda})}),
\]
and the right-hand side is an element of $\Theta(\Mm)$, again by Lemma~\ref{lem:Theta_PL_commute} and \eqref{eq:Theta}. Hence, every element of $\Theta(\Nn)$ is contained in an element of $\Theta(\Mm)$; that is, $\Theta(\Nn) \ge \Theta(\Mm)$.

For the final statement, fix a partition $\Mm = \{M_i\}_{i \in I} \in \Rfr_{\Aa}$.	
If $\Mm$ is a Markov partition, then for each $i \in I$ and $\Lambda \in P_i$ we have $\Theta_{P_i,\Lambda}(M_i) = \{M_i\}$, so $\Theta(\Mm) = \Mm$. 
For the converse, suppose that $\Theta(\Mm) = \Mm$. Since $\Mm$ is a partition and $\Theta_{P_i, \Lambda_{|P_i|}} \circ \cdots \circ \Theta_{P_i,\Lambda_1} (M_i)$ partitions $M_i$, for each $M_i \in \Mm$ we must have $\Theta_{P_i, \Lambda_{|P_i|}} \circ \cdots \circ \Theta_{P_i,\Lambda_1} (M_i) = \{M_i\}$. This can only occur if $\Theta_{P_i,\Lambda}(M_i) = \{M_i\}$ for all $\Lambda \in P_i$. That is, for each $\Lambda \in P_i$ there is a unique $j \in I$ such that $\gamma(M_i) \subseteq M_j$ for all $\gamma \in \Lambda$, so \ref{itm:markov_property} is satisfied. 
\end{proof}

 By Lemma~\ref{lem:Theta}\ref{itm:theta_refines}, the partitions $\Theta^k(\Aa)$ form an increasing sequence of refinements of $\Aa$, and any fixed point of this sequence is a Markov partition by Lemma~\ref{lem:Theta}. In general, the sequence $(\Theta^k(\Aa))_{k \ge 0}$ need not stabilise, but it always determines the unique minimal Markov partition of $(\AA,\Gamma)$.

\begin{dfn}\label{dfn:M_infty}
Let $\Mm_\infty$ be the partition of $\AA$ induced by the equivalence relation
\[
x \sim x' \iff \text{for every } k \ge 0 \text{ there exists } M \in \Theta^k(\Aa) \text{ such that } x,x' \in M.
\]
\end{dfn}

The element of $\Mm_\infty$ containing $x \in \AA$ is the intersection over $k \ge 0$ of the elements of $\Theta^k(\Aa)$ containing $x$, so $\Mm_\infty$ is a Borel partition of $\AA$ refining $\Theta^k(\Aa)$ for every $k \ge 0$. In particular, $\Mm_\infty \in \Rfr_{\Aa}$. 

\begin{thm}\label{thm:finite_iff_Theta_const}
Let $(\AA,\Gamma)$ be an iterated function system. Then $\Mm_\infty$ is the unique minimal Markov partition of $(\AA,\Gamma)$ and every Markov partition refines $\Mm_{\infty}$. Moreover, the following are equivalent:
\begin{enumerate}[label=(\roman{enumi})]
	\item\label{itm:finite_markov} $(\AA,\Gamma)$ admits a finite Markov partition;
	\item\label{itm:m_inf_finite} $\Mm_\infty$ is finite; and
	\item\label{itm:theta_eventually_constant} the sequence $(\Theta^k(\Aa))_{k \ge 0}$ is eventually constant.
\end{enumerate}
If these equivalent conditions hold, then $\Mm_\infty = \lim_{k \to \infty} \Theta^k(\Aa)$.
\end{thm}

\begin{proof}

Since $\Mm_\infty \in \Rfr_{\Aa}$, it satisfies \ref{itm:E_refines_A}. Suppose that $x,x' \in \AA$ lie in the same element of $\Mm_\infty$ and fix $\gamma \in \Gamma$. For each $k \ge 0$ the points $x$ and $x'$ lie in the same element of $\Theta^{k+1}(\Aa) = \Theta(\Theta^k(\Aa))$, so \eqref{eq:Theta} and Lemma~\ref{lem:Theta_PL_commute} place $\gamma(x)$ and $\gamma(x')$ in the same element of $\Theta^k(\Aa)$. Hence, $\gamma(x)$ and $\gamma(x')$ lie in the same element of $\Mm_\infty$, so $\Mm_\infty$ satisfies \ref{itm:markov_property}. 

Let $\Nn$ be a Markov partition of $(\AA,\Gamma)$. Then $\Nn \ge \Aa$ by \ref{itm:E_refines_A} and $\Theta(\Nn) = \Nn$ by Lemma~\ref{lem:Theta}, so Lemma~\ref{lem:Theta}\ref{itm:theta_order} gives $\Nn = \Theta^k(\Nn) \ge \Theta^k(\Aa)$ for all $k \ge 0$. Hence, any two points in the same element of $\Nn$ lie in the same element of $\Theta^k(\Aa)$ for every $k$. That is, $\Nn \ge \Mm_\infty$. If $\Mm_\infty \ge \Nn$ as well, then $\Nn = \Mm_\infty$, so $\Mm_\infty$ is minimal. If $\Nn$ is minimal, then $\Nn \ge \Mm_\infty$ forces $\Nn = \Mm_\infty$, so $\Mm_\infty$ is the unique minimal Markov partition.

\ref{itm:finite_markov} $\implies$ \ref{itm:m_inf_finite}. 
 Suppose $\Mm$ is a finite Markov partition of $(\AA,\Gamma)$. Then $\Mm$ refines $\Mm_\infty$, so $|\Mm_\infty| \le |\Mm| < \infty$.
 
  \ref{itm:m_inf_finite} $\implies$ \ref{itm:theta_eventually_constant}.
 Suppose $\Mm_{\infty}$ is finite. For each $k \ge 0$ we have $\Mm_\infty \ge \Theta^{k+1}(\Aa) \ge \Theta^k(\Aa)$ by Lemma~\ref{lem:Theta}\ref{itm:theta_refines}, so $|\Theta^k(\Aa)|$ is nondecreasing in $k$ and bounded above by $|\Mm_\infty|$. Hence, $|\Theta^{k+1}(\Aa)| = |\Theta^k(\Aa)|$ for all sufficiently large $k$, so $\Theta^{k+1}(\Aa) = \Theta^k(\Aa)$ since $\Theta^{k+1}(\Aa)$ refines $\Theta^{k}(\Aa)$.
 
 \ref{itm:theta_eventually_constant} $\implies$ \ref{itm:finite_markov}.
 Suppose that there is some $k_0 \ge 0$ such that  $\Theta^k(\Aa) = \Theta^{k_0}(\Aa)$ for all $k \ge k_0$. Then $\Mm_\infty = \Theta^{k_0}(\Aa)$. This partition is finite since $\Aa$ is finite and $\Theta$ takes finite partitions to finite partitions. It is Markov by the first paragraph of this proof. 
\end{proof}

Theorem~\ref{thm:finite_iff_Theta_const} provides a convenient algorithm for producing a finite Markov partition when one exists. 
Unpacking the definition of $\Theta$, the algorithm proceeds as follows. 
\begin{algo}\label{algo:Markov_partition}
Initialise $\Mm_0 \coloneqq \Aa$. Given $\Mm_k = \{M_i\}_{i \in I_k}$, construct $\Mm_{k+1} \coloneqq \Theta(\Mm_k)$ by performing the following steps for each $i \in I_k$:
\begin{enumerate}[label=(\roman{enumi})]
	\item identify the unique $P_i \in \Pp$ with $M_i \subseteq \AA_{P_i}$;
	\item for each $\Lambda \in P_i$, choose any $\gamma \in \Lambda$ and partition $M_i$ into the nonempty preimages
	$M_i \cap \gamma^{-1}(M_j)$ for all $j \in I_k$;	
	\item replace $M_i$ in $\Mm_{k+1}$ by the common refinement of these partitions over all $\Lambda \in P_i$.
\end{enumerate}
If $\Mm_{k+1} = \Mm_k$, terminate. Otherwise, increment $k$ and repeat. 
\end{algo}
By Lemma~\ref{lem:Theta} if Algorithm~\ref{algo:Markov_partition} terminates at step $k$, then $\Mm_k$ is Markov, and Theorem~\ref{thm:finite_iff_Theta_const} guarantees it is the unique minimal Markov partition. The procedure terminates in finitely many steps if and only if $(\AA,\Gamma)$ admits a finite Markov partition.

\begin{example}\label{ex:clarence_minimal}
Applying Algorithm~\ref{algo:Markov_partition} to the system of Example~\ref{ex:clarence1}, one computes that $\Theta(\Aa) = \{ \{0\}, (0,1/2), [1/2,1), \{1\}\}$ and  $\Theta^2(\Aa) = \Theta(\Aa)$. Lemma~\ref{lem:Theta} shows that $\Theta(\Aa)$ is Markov, and Theorem~\ref{thm:finite_iff_Theta_const} shows it is the unique minimal Markov partition of $([0,1],\Gamma)$.
\end{example}

 In general, Algorithm~\ref{algo:Markov_partition} may not terminate as an iterated function system may not admit a finite Markov partition. As the following example shows, this can even occur when the iterated function system is contractive on the unit interval.

\begin{example}\label{ex:ikw_no_markov}
 Fix an irrational $y_0 \in (0,1)$ and consider the contractive system $(\AA = [0,1], \Gamma = \{\gamma_1,\gamma_2,\gamma_3\})$ with
\[
\gamma_1(x) = \frac{x}{2}, \qquad \gamma_2(x) = \frac{x+1}{2}, \text{ and} \qquad \gamma_3(x) = \frac{2x + y_0}{6}.
\]
 Since $\gamma_3([0,1]) \subseteq [0,1/2)$, and $\gamma_1(x) = \gamma_3(x)$ if and only if $x = y_0$, the only coincidence between distinct maps is $\gamma_1(y_0) = \gamma_3(y_0) = y_0/2$, so the fundamental partition is $\Aa = \{\{y_0\}, [0,1] \setminus \{y_0\}\}$.

Suppose that $(\AA,\Gamma)$ admits a finite Markov partition $\Mm$. By \ref{itm:E_refines_A}, $\{y_0\} \in \Mm$. For $z \in \AA$ let $d(z)$ be the least integer $k \ge 0$ for which there exist $\lambda_1,\ldots,\lambda_k \in \Gamma$ with $\lambda_1 \circ \cdots \circ \lambda_k(z) = y_0$, the empty composition being the identity, and set $d(z) = \infty$ if no such $k$ exists. We claim that $d$ is constant on sets in $\Mm$. 
For this, we argue by induction on $k \ge 0$, that if $z$ and $z'$ lie in the same element of $\Mm$ and $d(z) \le k$, then $d(z') \le k$. If $k = 0$, then $z = y_0$, so the element of $\Mm$ containing $z$ is $\{y_0\}$ and $z' = y_0$. Now suppose $k \ge 1$. We may assume that $j \coloneqq d(z) \ge 1$, and choose $\lambda_1,\ldots,\lambda_j \in \Gamma$ with $\lambda_1 \circ \cdots \circ \lambda_j(z) = y_0$. By \ref{itm:markov_property}, the points $\lambda_j(z)$ and $\lambda_j(z')$ lie in a common element of $\Mm$, and $d(\lambda_j(z)) \le j-1 \le k-1$, so $d(\lambda_j(z')) \le k-1$ by the inductive hypothesis. Hence, $d(z') \le k$. Consequently, $d$ is constant on each element of $\Mm$, so takes at most $|\Mm|$ values in $\NN \cup \{\infty\}$.

On the other hand, write $y_0 = \sum_{i \ge 1} b_i 2^{-i}$ in binary, put $\lambda^{(0)} \coloneqq \gamma_1$ and $\lambda^{(1)} \coloneqq \gamma_2$, and for each integer $k \ge 0$ let $z_k \coloneqq 2^k y_0 - \lfloor 2^k y_0 \rfloor$. Then $\lambda^{(b_1)} \circ \cdots \circ \lambda^{(b_k)}(z_k) = y_0$, so $d(z_k) \le k$. The points $z_k$ are pairwise distinct since $y_0$ is irrational, while for each $j$ at most $3^j$ points $z$ satisfy $d(z) = j$, since each composition of $j$ maps in $\Gamma$ is injective. If the values $d(z_k)$ were bounded by some $D$, then the infinitely many points $z_k$ would lie in a set with at most $\sum_{j=0}^{D} 3^j$ elements. Consequently, $d$ takes infinitely many finite values, contradicting the finiteness of $\Mm$.
\end{example}

Although Algorithm~\ref{algo:Markov_partition} does not terminate for the system of Example~\ref{ex:ikw_no_markov}, the minimal Markov partition $\Mm_\infty$ of Theorem~\ref{thm:finite_iff_Theta_const} is countable.

\begin{prop}\label{prop:countable_markov}
Let $(\AA,\Gamma)$ be an iterated function system in which the set 
\[
C \coloneqq \{y \in \AA \mid \gamma(y) = \gamma'(y) \text{ for some distinct } \gamma,\gamma' \in \Gamma\}
\]
is countable. Then the minimal Markov partition $\Mm_\infty$ of $(\AA,\Gamma)$ is countable.
\end{prop}
\begin{proof}
Let $W$ be the set of $x \in \AA$ such that $\lambda_1 \circ \cdots \circ \lambda_k(x) \in C$ for some integer $k \ge 0$ and $\lambda_1,\ldots,\lambda_k \in \Gamma$. Then $W$ is countable, since each composition of maps in $\Gamma$ is injective. Each $\gamma \in \Gamma$ maps $\AA \setminus W$ into itself, and $\AA \setminus W$ is contained in the element $\AA_{\{\{\gamma\} \mid \gamma \in \Gamma\}}$ of $\Aa$. An induction on $k$ using \eqref{eq:Theta} shows that $\AA \setminus W$ lies in a single element of $\Theta^k(\Aa)$ for every $k$. Hence, at most one element of $\Mm_\infty$ meets $\AA \setminus W$, and every other element is a nonempty subset of the countable set $W$, so $\Mm_\infty$ is countable. 
\end{proof}

\subsection{The graph associated to a Markov partition}

To a Markov partition we associate a directed graph that captures the combinatorial aspects of the dynamics of the iterated function system. Let $\Mm = \{M_i\}_{i \in I}$  be a Markov partition of $(\AA,\Gamma)$. Define an equivalence relation on $\Gamma \times I$ by $(\gamma,i) \sim (\gamma',i)$ if and only if $\restr{\gamma}{M_i} = \restr{\gamma'}{M_i}$.  We write $[\gamma,i]$ for the equivalence class of $(\gamma,i)$. Observe that if $M_i \subseteq \AA_{P}$ then there is some $\Lambda \in P$ such that $[\gamma,i] = \{ (\gamma',i) \mid \gamma' \in \Lambda \}.$

\begin{dfn}\label{dfn:graph_of_markov_partition}
	Suppose that $\Mm = \{M_i\}_{i \in I}$ is a Markov partition of $(\AA,\Gamma)$. We associate to $\Mm$ a directed graph $E=  (E^0, E^1, r, s)$ with 
\emph{vertices} $E^0 \coloneqq I$;
 \emph{edges} $E^1 \coloneqq \{[\gamma,i] \mid \gamma \in \Gamma, i \in I\}$;
 \emph{source map} $s \colon E^1 \to E^0$ given by $s([\gamma,i]) = i$; and
 \emph{range map} $r \colon E^1 \to E^0$ given by $r([\gamma,i]) = j$ for the unique $j \in I$ such that $\gamma(M_i) \subseteq M_j$.
\end{dfn}

\begin{example}\label{ex:clarence_graph}
	Consider the contractive iterated function system $([0,1],\Gamma = \{\gamma_1,\gamma_2,\gamma_3\})$ of Example~\ref{ex:clarence}. 
The directed graph associated to the Markov partition $\Mm = \{M_i\}_{i=1}^4$ of Example~\ref{ex:clarence1} is described by the diagram
\begin{center}
\begin{tikzpicture}[
    >=Stealth, 
    vertex/.style={draw, circle, inner sep=1.5pt, minimum size=3mm, font=\scriptsize},
    every edge/.style={draw, -latex},
    lbl/.style={font=\scriptsize, inner sep=3pt}
]

\node[vertex] (v1) at (0,0) {$1$};
\node[vertex] (v2) at (0,2) {$2$};
\node[vertex] (v3) at (2,2) {$3$};
\node[vertex] (v4) at (2,0) {$4$};

\path
    (v1) edge[out=270, in=180, loop, distance=12mm] node[lbl, left=2pt] {$[\gamma_1,1]$} (v1)
    (v1) edge node[lbl, below] {$[\gamma_2,1]$} (v4)
    (v1) edge node[lbl, left] {$[\gamma_3,1]$} (v2)

    (v2) edge[out=225, in=140, loop, distance=12mm] node[lbl, left] {$[\gamma_1,2]$} (v2)
    (v2) edge[out=130, in=45, loop, distance=12mm] node[lbl, above] {$[\gamma_3,2]$} (v2)
    (v2) edge[bend left=20] node[lbl, above] {$[\gamma_2,2]$} (v3)

    (v3) edge[bend left=20] node[lbl, below] {$[\gamma_1,3]$} (v2)
    (v3) edge[out=90, in=0, loop, distance=12mm] node[lbl, right=2pt] {$[\gamma_2,3]$} (v3)

    (v4) edge node[lbl, right] {$[\gamma_1,4]$} (v3)
;
\end{tikzpicture} 
\end{center} 
in which $[\gamma_1,3] = [\gamma_3,3]$ and $[\gamma_1,4] = [\gamma_2,4] = [\gamma_3,4]$. 
\end{example}

A vertex $i \in E^0$ is a \emph{sink} if $|s^{-1}(i)| = 0$ and a \emph{source} if $|r^{-1}(i)| = 0$.

\begin{lem}\label{lem:graph_properties}
Let $E$ be the directed graph associated to a Markov partition $\Mm = \{M_i\}_{i \in I}$ of $(\AA,\Gamma)$. Then
\begin{enumerate}[label=(\roman{enumi})]
\item\label{itm:E_no_sinks_sources} $E$ has no sinks or sources; 
\item\label{itm:E_s_inverse} $|s^{-1}(i)| \le |\Gamma|$ for each $i \in I$; and
\item\label{itm:E_finite} $E$ is finite if and only if $\Mm$ is finite.
\end{enumerate}
\end{lem}
\begin{proof}
That $|s^{-1}(i)|> 0$ is clear. For each $x \in M_i$  equation \eqref{eq:Gamma_invariance} gives some $\gamma \in \Gamma$ and $y \in \AA$ such that $x = \gamma(y)$. Let $j \in I$ be such that $y \in M_j$. Then $r([\gamma,j]) = i$, so \ref{itm:E_no_sinks_sources} holds. For \ref{itm:E_s_inverse} observe that for each $\gamma \in \Gamma$ there is at most one edge of the form $[\gamma,i]$, so $|s^{-1}(i)| \le |\Gamma|$. For \ref{itm:E_finite} note that if $\Mm$ is finite, then $E^0$ is finite, so $E^1$ is finite by \ref{itm:E_s_inverse}. Conversely, if $E$ is finite, then $E^0$ is finite, so $\Mm$ is finite.
\end{proof}

Let $E^n \coloneqq \{e_1e_2 \cdots e_n \mid e_i \in E^1,\,s(e_i) = r(e_{i+1})\}$ be the set of paths of length $n$ in $E$, and extend the range and source maps on $E$ to maps $r \colon E^n \to E^0$ and $s \colon E^n \to E^0$ by $r(e_1e_2 \cdots e_n) = r(e_1)$ and $s(e_1e_2 \cdots e_n) = s(e_n)$, respectively, following the conventions of \cite{Rae05}. Let $E^* = \bigcup_{n=0}^\infty E^n$ be the collection of all finite paths in $E$, where $E^0$ is identified with the collection of paths of length zero.
For each $i,j \in I$ and $n \in \NN \cup \{*\}$ we let
	\begin{align*}
	i E^n \coloneqq \{ e \in E^n \mid r(e) = i \}, \quad
	E^n j \coloneqq \{ e \in E^n \mid s(e) = j\}, \text{ and} \quad
	i E^n j \coloneqq i E^n \cap E^n j.
	\end{align*}
The next result follows immediately from the definition of $E$ and highlights how $E$ encodes the combinatorial structure of the iterated function system $(\AA,\Gamma)$.
\begin{lem} \label{lem:markov_partition_paths}
Let $E$ be the directed graph associated to a Markov partition $\Mm = \{M_i\}_{i \in I}$ of $(\AA,\Gamma)$. For each $e \in E^1$ choose a representative $\gamma_e \in \Gamma$ in the sense that $e = [\gamma_e, s(e)]$. Then
for each $n \in \NN$, $i \in I$, and $x \in M_i$ there is a bijection $E^n i \to \Gamma^{(n)} x$ given by 
\[
e_1 \cdots e_n \mapsto (\gamma_{e_1} \circ \cdots \circ \gamma_{e_n}(x),\ldots,\gamma_{e_n}(x),x)
\]
that is independent of the choice of representatives. 
In particular, $|E^n i| = |\Gamma^{(n)} M_i|$ for all $i \in I$ and $n \in \NN$.
\end{lem}

If $S \subseteq E^0$ is a subset of vertices, we write $E|_S$ for the subgraph of $E$ with vertices $E|_S^0 = S$ and edges $E|_S^1 = \{e \in E^1 \mid r(e), s(e) \in S\}$.

\begin{dfn}
A directed graph $E$ is \emph{strongly connected} if for every pair of vertices $v,w \in E^0$ there exists a path in  $vE^nw$ for some $n \ge 0$.
 A \emph{strongly connected component} of a directed graph $E$ is a maximal subset $S \subseteq E^0$, in the inclusion ordering, such that $E|_S$ is strongly connected. A nonempty subset $R \subseteq E^0$ is \emph{recurrent} if $E|_R$ is strongly connected and $s(e) \in R$ implies $r(e) \in R$ for every $e \in E^1$ (that is, $R$ is \emph{hereditary} in the sense of \cite{Rae05}). Such an $R$ is necessarily a strongly connected component, and we call it a \emph{recurrent component}.
 A vertex is \emph{transient} if it is not contained in any recurrent component. We write $T$ for the set of transient vertices of $E$.
\end{dfn}

Belonging to the same strongly connected component is an equivalence relation $\approx$ on $E^0$, so we write $E^0/{\approx}$ for the collection of strongly connected components of $E$.  The set $E^0/{\approx}$  admits a partial order $\le$ where $S \le S'$ for $S,S' \in E^0/{\approx}$ if there is a path $e \in E^*$ such that $s(e) \in S'$ and $r(e) \in S$. Reflexivity and transitivity of $\le$ are clear. Maximality of strongly connected components implies that if $S \le S'$ and $S' \le S$, then $S = S'$. A strongly connected component $R$ is recurrent if and only if $R$ is a minimal element in the partial order $\le$ on $E^0/{\approx}$.

We identify a checkable condition, in the setting of contractive iterated function systems, that guarantees the existence of a unique recurrent component in the graph of a Markov partition.
 \begin{dfn}[{cf.~\cite[Definition~A.1.15]{Kie02}}]
	A map $f \colon X \to Y$ between topological spaces is said to be \emph{weakly open} if $f(U)$ has nonempty interior for all open $U \subseteq X$. 
	An iterated function system $(\AA,\Gamma)$ is \emph{weakly open} if every $\gamma \in \Gamma$ is weakly open.
\end{dfn}

\begin{prop}\label{prop:weakly_open}
Suppose that $(\AA,\Gamma)$ is a contractive iterated function system with a Markov partition $\Mm = \{M_i\}_{i \in I}$ and associated directed graph $E$. If there exists $i \in I$ such that $\interior(M_i) \ne \varnothing$, then there is a unique recurrent component $R$ of $E$. Moreover, $
\{ i \in E^0 \mid \interior(M_i) \ne \varnothing\} \subseteq R,
$
with equality when $(\AA,\Gamma)$ is weakly open. In particular, $\bigcup_{j \in R} M_j$ is dense in $\AA$. 
\end{prop} 

\begin{proof}
Fix $i \in E^0$ such that $\interior(M_i) \ne \varnothing$ and let $R$ be the strongly connected component of $E$ containing $i$.
By Lemma~\ref{lem:eventually_in_open_sets}, for each $j \in E^0$ the set $iE^*j$ is nonempty. In particular, for each $e \in E^1$ with $s(e) \in R$, we have $iE^*r(e) \ne \varnothing$, so $r(e)$ is in the same strongly connected component as $i$. That is, $R$ is recurrent. For any other strongly connected component $S$ there is a path from some vertex in $S$ to $i$, so if $S$ is recurrent, then $S = R$. Hence, $R$ is the unique recurrent component of $E$.
Since the same argument can be applied to any other $k \in E^0$ such that $\interior(M_k) \ne \varnothing$ it follows that $\{ i \in E^0 \mid \interior(M_i) \ne \varnothing\} \subseteq R$.
Since $R$ is a strongly connected component, for each $j \in R$ there is a sequence of maps $\gamma_1,\ldots,\gamma_m \in \Gamma$ such that $\gamma_1 \circ \cdots \circ \gamma_m (\interior(M_i)) \subseteq M_j$. If $(\AA,\Gamma)$ is weakly open, then $\interior(M_j) \ne \varnothing$, so $\{ i \in E^0 \mid \interior(M_i) \ne \varnothing\} = R$.

For the final statement, Hutchinson's Theorem implies  the set	$\bigcup_{n \in \NN} \Gamma^n x$ is dense in $\AA$ for each $x \in \AA$. If $x \in \interior(M_i)$, then $\bigcup_{n \in \NN} \Gamma^n x \subseteq \bigcup_{j \in R} M_j$, so $\bigcup_{j \in R} M_j$ is dense in $\AA$.
\end{proof}

\begin{example}\label{ex:clarence_recurrent}
	Consider the contractive iterated function system $([0,1],\Gamma = \{\gamma_1,\gamma_2,\gamma_3\})$ of Example~\ref{ex:clarence}. The system is weakly open and the graph $E$ of Example~\ref{ex:clarence_graph} has a unique recurrent component $R = \{2,3\}$ corresponding to the sets with nonempty interior in $\Mm$.
\end{example}

The hypothesis of Proposition~\ref{prop:weakly_open} that at least one $\interior(M_i)$ is nonempty is satisfied when $\Mm$ is finite and  minimal. This relies on a topological property of the sets produced by the refinement operator $\Theta$ of Definition~\ref{dfn:Theta}.
Recall that a subset $L$ of a topological space $X$ is \emph{locally closed} if $L = C \cap U$ for some closed $C \subseteq X$ and some open $U \subseteq X$.

\begin{lem}\label{lem:markov_locally_closed}
Every element of $\Theta^k(\Aa)$ is locally closed for every $k \ge 0$.
\end{lem}
\begin{proof}
We induct on $k$. By Lemma~\ref{lem:S_P}\ref{item:S_P}, for each $P \in  \Pp$ the set $S_P$ is compact, hence closed. Since $\Pp$ is finite,  Lemma~\ref{lem:S_P}\ref{item:AA_P} implies that $\AA_P = S_P \setminus \bigcup_{P' < P} S_{P'}$ is locally closed for all $P \in \Pp$. Hence, every element of $\Aa = \Theta^0(\Aa)$ is locally closed.

Suppose that every element of $\Mm \coloneqq \Theta^k(\Aa) = \{M_i\}_{i \in I}$ is locally
closed and fix $i \in I$, with $P_i \in \Pp$ such that $M_i \subseteq \AA_{P_i}$. For each
$\Lambda \in P_i$ fix $\gamma_{\Lambda} \in \Lambda$. By Lemma~\ref{lem:Theta_PL_commute}
and \eqref{eq:Theta}, the elements of $\Theta(\Mm)$ contained in $M_i$ are the nonempty
sets of the form
\begin{equation} \label{eq:theta_sets}
	M_i \cap \bigcap_{\Lambda \in P_i} \gamma_{\Lambda}^{-1}(M_{j_{\Lambda}}),
\qquad (j_{\Lambda})_{\Lambda \in P_i} \in I^{P_i}.
\end{equation}
Each $M_{j_\Lambda}$ is locally closed by the inductive hypothesis and each $\gamma_\Lambda$
is continuous, so each $\gamma_{\Lambda}^{-1}(M_{j_{\Lambda}})$ is locally closed. Since
$P_i$ is finite and locally closed sets are closed under finite intersections, each set from \eqref{eq:theta_sets}
is locally closed. Hence, every element of $\Theta^{k+1}(\Aa)$ is locally closed.
\end{proof}

Recall that a subset $S$ of a topological space is \emph{nowhere dense} if $\interior(\ol{S}) = \varnothing$. We require the following folklore result.

\begin{lem}\label{lem:locally_closed}
Let $L$ be a locally closed subset of a topological space. If $\interior(L) = \varnothing$, then $L$ is nowhere dense.
\end{lem}
\begin{proof}
Write $L = C \cap U$ with $C$ closed and $U$ open. Since $\overline{L} \subseteq C$, we have $\overline{L} \cap U = L$. Hence, $\interior(\overline{L}) \cap U$ is an open subset of $L$, so $\interior(\overline{L}) \cap U = \varnothing$ by assumption. As $L \subseteq U$, it follows that $L \cap \interior(\overline{L}) = \varnothing$; that is, $L \subseteq \partial \overline{L}$. The boundary of a closed set is nowhere dense (see~\cite[Example~11.5.3(e)]{NB11}), and every subset of a nowhere dense set is nowhere dense, so $L$ is nowhere dense.
\end{proof}

\begin{prop}\label{prop:minimal_markov_has_open_set}
Suppose that $(\AA,\Gamma)$ admits a finite Markov partition and let $\Mm = \{M_i\}_{i \in I}$ be its minimal Markov partition. Then each $M_i$ is locally closed, each $M_i$ with $\interior(M_i) = \varnothing$ is nowhere dense, and there exists $i \in I$ such that $\interior(M_i) \ne \varnothing$.
\end{prop}
\begin{proof}
By Theorem~\ref{thm:finite_iff_Theta_const} the sequence $(\Theta^k(\Aa))_{k \ge 0}$ is eventually constant with limit $\Mm$, so $\Mm = \Theta^k(\Aa)$ for some $k \ge 0$. Lemma~\ref{lem:markov_locally_closed} shows that each $M_i$ is locally closed, and Lemma~\ref{lem:locally_closed} shows that each $M_i$ with empty interior is nowhere dense. Since a finite union of nowhere dense sets is nowhere dense, the nonempty space $\AA$ is not a union of finitely many nowhere dense subsets. As $\Mm$ is a finite partition of $\AA$, at least one $M_i$ has nonempty interior.
\end{proof}

\begin{cor}\label{cor:minimal_unique_recurrent}
Suppose that $(\AA,\Gamma)$ is a contractive iterated function system admitting a finite Markov partition, let $\Mm = \{M_i\}_{i \in I}$ be its minimal Markov partition, and let $E$ be the associated directed graph. Then $E$ has a unique recurrent component $R$. If, in addition, $(\AA,\Gamma)$ is weakly open, then $\bigcup_{i \in I \setminus R} M_i$ is nowhere dense.
\end{cor}
\begin{proof}
Proposition~\ref{prop:minimal_markov_has_open_set} provides $i \in I$ with $\interior(M_i) \ne \varnothing$, so Proposition~\ref{prop:weakly_open} gives a unique recurrent component $R$ of $E$. If $(\AA,\Gamma)$ is weakly open, then Proposition~\ref{prop:weakly_open} gives $I \setminus R = \{i \in I \mid \interior(M_i) = \varnothing\}$, so $\bigcup_{i \in I \setminus R}M_i$ is a finite union of nowhere dense sets by Proposition~\ref{prop:minimal_markov_has_open_set}.
\end{proof}

\begin{rmk}\label{rmk:minimality_needed}
Minimality cannot be dropped from Proposition~\ref{prop:minimal_markov_has_open_set} or Corollary~\ref{cor:minimal_unique_recurrent}. 
In a non-minimal finite Markov partition every set may have empty interior, the sets need not be locally closed, and the associated graph may have several recurrent components.

Consider the contractive system $([0,1],\Gamma)$ with $\Gamma = \{\gamma_1,\gamma_2\}$ given by $\gamma_1(x) = x/2$ and $\gamma_2(x) = (x+1)/2$. The maps agree at no point of $[0,1]$, so $\Aa = \{[0,1]\}$ and $\Theta(\Aa) = \Aa$; that is, $\{[0,1]\}$ is the minimal Markov partition. The associated directed graph has a single vertex with two loops, so it has a unique recurrent component.

On the other hand, 
$
\Nn \coloneqq \{\QQ \cap [0,1], \; [0,1] \setminus \QQ\}
$
is also a finite Markov partition, since both $\gamma_1$ and $\gamma_2$ send rationals to rationals and irrationals to irrationals. Both elements of $\Nn$ are dense with empty interior, so by Lemma~\ref{lem:locally_closed} neither is locally closed. The graph of $\Nn$ has two loops at each of its two vertices and no other edges, so it has two recurrent components. 
\end{rmk}

\subsection{The adjacency matrix}

If $\Mm = \{M_i\}_{i \in I}$ is a finite Markov partition of $(\AA,\Gamma)$ and $E$ is the associated directed graph, then we denote by $\AAA_E$ the \emph{adjacency matrix} of $E$: the $I \times I$ matrix with nonnegative integer entries
\[
\AAA_E(i,j) = \big|\{ e \in E^1 \mid r(e) = i \text{ and } s(e) = j\}\big|.
\]	
By Lemma~\ref{lem:markov_same_paths} the quantity $|\Gamma^{(n)}x|$ depends only on the element of $\Mm$ containing $x$, so
\begin{equation}\label{eq:path_growth}
h_n \coloneqq \sup_{x \in \AA}\big|\Gamma^{(n)}x\big| = \max_{i \in I}\big|\Gamma^{(n)}M_i\big|
\end{equation}
is attained, and $1 \le h_n \le |\Gamma|^n$ for every $n \in \NN$.

\begin{prop}\label{prop:rho_intrinsic}
Let $(\AA,\Gamma)$ be an iterated function system admitting a finite Markov partition $\Mm = \{M_i\}_{i \in I}$ and let $E$ be the associated directed graph. Then $\|\AAA_E^n\|_1 = h_n$ for every $n \ge 1$, where $\|\cdot\|_1$ denotes the operator $1$-norm. Consequently,
\[
\rho(\AAA_E) = \lim_{n \to \infty} h_n^{1/n} = \inf_{n \ge 1} h_n^{1/n} .
\]
In particular, $\rho(\AAA_E)$ is independent of the choice of finite Markov partition of $(\AA,\Gamma)$.
\end{prop}
\begin{proof}
By Lemma~\ref{lem:markov_partition_paths}, the $i$-th column sum of $\AAA_E^n$ is $|E^n i| = |\Gamma^{(n)}M_i|$, so \eqref{eq:path_growth} gives
\[
\|\AAA_E^n\|_1 = \max_{i \in I} \sum_{j \in I} \AAA_E^n(j,i) = \max_{i \in I} \big|\Gamma^{(n)}M_i\big| = h_n.
\]
The spectral radius formula gives the limit, and $\rho(\AAA_E)^n \le \|\AAA_E^n\|_1$ gives the infimum.
\end{proof}

\begin{rmk}\label{rmk:path_growth_rate}
Proposition~\ref{prop:rho_intrinsic} exhibits $\ln(\rho(\AAA_E)) = \lim_{n \to \infty} \tfrac{1}{n}\ln h_n$ as the exponential growth rate of the largest number of $n$-step paths through a point of $\AA$, so $\rho(\AAA_E)$ may be thought of as the \emph{effective number of maps} of $\Gamma$.  Systems with $\rho(\AAA_E) < |\Gamma|$ appear in Subsection~\ref{subsec:examples}.
By Lemma~\ref{lem:graph_properties}\ref{itm:E_no_sinks_sources} the graph $E$ has no sinks and no sources, so $\ln(\rho(\AAA_E))$ is also the topological entropy of the shift of finite type determined by $E$ (see \cite[\S4.4]{LM21}).
\end{rmk}

For contractive systems in which some $M_i$ has nonempty interior, the number of paths of a given length with source at any vertex is bounded by the number with source at some recurrent vertex. Consequently, the transient part of $E$ does not affect the spectral radius of $\AAA_E$. 

\begin{prop}\label{prop:transient_subcritical}
Let $(\AA,\Gamma)$ be a contractive iterated function system admitting a finite Markov partition $\Mm = \{M_i\}_{i \in I}$ with associated directed graph $E$, and suppose that $\interior(M_{i_0}) \ne \varnothing$ for some $i_0 \in I$. Then $E$ has a unique recurrent component $R$ and $\rho(\AAA_R) = \rho(\AAA_E)$, where $\AAA_R$ is the adjacency matrix of $E|_R$. In particular, the conclusion holds for the minimal Markov partition of $(\AA,\Gamma)$, which is finite by Theorem~\ref{thm:finite_iff_Theta_const}.
\end{prop}
\begin{proof}
By Proposition~\ref{prop:weakly_open} the graph $E$ has a unique recurrent component $R$ and $\bigcup_{j \in R} M_j$ is dense in $\AA$. Fix $n \ge 1$ and $i \in I$, and choose $x \in M_i$. By Lemma~\ref{lem:path_count_lsc}, the function $y \mapsto |\Gamma^{(n)}y|$ is lower semicontinuous, so there exists an open neighbourhood $V$ of $x$ such that $|\Gamma^{(n)}y| \ge |\Gamma^{(n)}x|$ for all $y \in V$. Since $V$ is open it intersects $\bigcup_{j \in R} M_j$, so there exist $j \in R$ and $y_0 \in V \cap M_j$. Lemma~\ref{lem:markov_partition_paths} gives
\[
|E^n j| = \big|\Gamma^{(n)}y_0\big| \ge \big|\Gamma^{(n)}x\big| = |E^n i| .
\]
As $i \in I$ was arbitrary, \eqref{eq:path_growth} yields $h_n = \max_{i \in I}|E^n i| = \max_{j \in R}|E^n j|$ for every $n \ge 1$. 
The range of any path with source in $R$ lies in $R$, so for $j \in R$ the number $|E^n j|$ is the $j$-th column sum of $\AAA_R^n$, and so $h_n = \|\AAA_R^n\|_1$. Proposition~\ref{prop:rho_intrinsic} gives $\rho(\AAA_E) = \lim_{n \to \infty} h_n^{1/n}= \rho(\AAA_R)$.

Proposition~\ref{prop:minimal_markov_has_open_set} gives the final statement.
\end{proof}

\section{Classification of KMS states}\label{sec:classification}

This section applies the Markov partition machinery of Section~\ref{sec:markov_partitions} to the framework of Section~\ref{sec:KW} to give a complete classification of the $\KMS$ states for gauge actions on the Kajiwara--Watatani algebras. Throughout this section we fix an iterated function system $(\AA,\Gamma)$ and make the following standing assumption.

\begin{ass}\label{ass}
The system $(\AA,\Gamma)$ admits a finite \emph{minimal} Markov partition  $\Mm = \{M_i\}_{i \in I}$, and $\Gamma$ contains at least two maps that differ as functions on $\AA$.
\end{ass}

If $(\AA,\Gamma)$ admits any finite Markov partition then the first part of Assumption~\ref{ass} holds by Theorem~\ref{thm:finite_iff_Theta_const}. The second part only excludes a trivial case: if all maps in $\Gamma$ coincide on $\AA$, then \eqref{eq:Gamma_invariance} forces this map to be a homeomorphism of $\AA$, and if $(\AA,\Gamma)$ is contractive, then $\AA$ is a single point. We let $E$ be the directed graph associated to $\Mm$, and let $\AAA_E$ be its adjacency matrix. We write $\rho \coloneqq \rho(\AAA_E)$ for the spectral radius of $\AAA_E$. More generally, given $S \subseteq I$ we write $\AAA_S$ for the adjacency matrix of the subgraph $E|_S$. By Proposition~\ref{prop:rho_intrinsic}, $\rho$ depends only on $(\AA,\Gamma)$, and not on a choice of finite Markov partition.

The \emph{critical inverse temperature} $\beta_c \coloneqq \ln (\rho)$ divides the analysis of $\KMS$ states into three regimes. For $0 < \beta < \beta_c$ we show in Corollary~\ref{cor:lower_bound_contractive} that for contractive systems, there are no $\KMS_\beta$ states. We show in Theorem~\ref{thm:inf_type} that for contractive systems there is a unique infinite-type $\KMS_{\beta_c}$ state $\varphi_\infty$, whose associated measure $\mu_\infty$ is the intrinsic Hutchinson measure mentioned in the introduction. For $\beta > \beta_c$ we show in Theorem~\ref{thm:finite_type_classification} that every $\KMS_\beta$ state is of finite type, and the full simplex of $\KMS_\beta$ states is parametrised by probability measures on $p_1(B_\Gamma)$.
 In general, $\beta_c \le \ln|\Gamma|$. By Proposition~\ref{prop:rho_intrinsic}, $\beta_c = \ln|\Gamma|$ if and only if $\sup_{x \in \AA} |\Gamma^{(n)} x| = |\Gamma|^n$ for every $n \in \NN$, in which case we recover the critical inverse temperature of \cite{IKW07}. 

\begin{rmk}\label{rmk:beta_c_positive}
If $(\AA,\Gamma)$ is contractive, then $\rho > 1$, so $\beta_c > 0$. Indeed, by Assumption~\ref{ass} the set $U \coloneqq \{y \in \AA \mid |\Gamma y| \ge 2\}$ is nonempty, and it is open by Lemma~\ref{lem:path_count_lsc}. Lemma~\ref{lem:eventually_in_open_sets} gives $\gamma_1,\ldots,\gamma_k \in \Gamma$ with $\gamma_1 \circ \cdots \circ \gamma_k(\AA) \subseteq U$, so for every $y \in \AA$ we have $|\Gamma^{(k+1)}y| \ge 2$, and hence $|\Gamma^{(m(k+1))}y| \ge 2^m$ for all $m \in \NN$. Proposition~\ref{prop:rho_intrinsic} then gives $\rho \ge 2^{1/(k+1)} > 1$. In particular, the results of this section at $\beta = \beta_c$ lie within the range $\beta > 0$ of Theorem~\ref{thm:LN_machine}.
\end{rmk}

Our workhorse is the following lemma.
\begin{lem}\label{lem:PF_operator_markov_partition}
Suppose that $\Mm = \{M_i\}_{i \in I}$ is a finite Markov partition of $(\AA,\Gamma)$ and let $E$ be the associated directed graph with adjacency matrix $\AAA_E$. Let $\mu \in A_+^*$ and let $\vec{\mu} = (\mu(M_i))_{i \in I} \in \RR_+^I$. For each $n \ge 1$ and $i \in I$ we have
\[
F^n(\mu)(M_i) = \sum_{j \in I} \AAA_E^n(i,j) \mu(M_j) = (\AAA_E^n \vec{\mu})_i.
\]
In particular, $\|F^n(\mu)\| = \|\AAA_E^n \vec{\mu}\|_1$ for all $n \in \NN$.
\end{lem}
\begin{proof}
By \eqref{eq:PF_characteristic},
\begin{align*}
F(\mu)(M_i) = \int_\AA |\Gamma y \cap M_i|\, d\mu(y)
= \sum_{j \in I} \int_{M_j} |\Gamma y \cap M_i|\, d\mu(y).
\end{align*}
The definition of $E$ implies that for each $j \in I$ and $y \in M_j$ we have $|\Gamma y \cap M_i| = \AAA_E(i,j)$. Hence,
\[
F(\mu)(M_i) = \sum_{j \in I} \AAA_E(i,j) \mu(M_j) = (\AAA_E \vec{\mu})_i.
\]
Iterating yields the formula for $F^n(\mu)(M_i)$.
For the norm equality we compute \[\|F^n(\mu)\| = F^n(\mu)(\AA) = \sum_{i \in I} F^n(\mu)(M_i) = \sum_{i \in I} (\AAA_E^n \vec{\mu})_i = \|\AAA_E^n \vec{\mu}\|_1. \qedhere
\]
\end{proof}

By Theorem~\ref{thm:LN_machine} classifying $\KMS_\beta$ states on $\Oo_{\Gamma}$ reduces to classifying $\beta$-subinvariant tracial states $\mu \in T_\beta(A)$. By Lemma~\ref{lem:PF_operator_markov_partition}, $F_\beta(\mu)\le\mu$ implies $\AAA_E\vec\mu\le e^\beta\vec\mu$ for the mass vector $\vec\mu=(\mu(M_i))_i$, which brings Perron--Frobenius theory to bear. 

Recall that a nonnegative matrix $Q$ is irreducible if for every pair of indices $i,j$ there exists $n \ge 1$ such that $Q^n(i,j) > 0$, and that a directed graph with no sinks is strongly connected if and only if its adjacency matrix is irreducible. If $R$ is a recurrent component of $E$, then every edge of $E$ with source in $R$ has range in $R$, so $E|_R$ inherits from Lemma~\ref{lem:graph_properties}\ref{itm:E_no_sinks_sources} the property of having no sinks. Since $E|_R$ is strongly connected, it follows that $\AAA_R$ is irreducible.

We first establish a lower bound on the values of $\beta$ for which $T_\beta(A)$ is nonempty.
\begin{prop}\label{prop:lower_bound}
	Let $\beta > 0$. For each recurrent component $R$ of $E$, let $\AAA_R$ be the adjacency matrix of the subgraph $E|_R$. If $\mu \in T_{\beta}(A)$, then
	\[
	\beta \ge \min \big\{\ln(\rho(\AAA_R)) \mid R \text{ is a recurrent component of } E\big\}.
	\]
\end{prop}

\begin{proof}
Suppose that $\mu$ is a Borel probability measure on $\AA$ satisfying $ F_{\beta}(\mu) \le \mu$. Let $\vec{\mu} = (\mu(M_i))_{i \in I}$, so $\vec{\mu}_i = \mu(M_i)$. Then $\vec{\mu} \ge 0$ entrywise and $\|\vec{\mu}\|_1 = 1$. 
By Lemma~\ref{lem:PF_operator_markov_partition}, for each $j \in I$ we have
$
\vec{\mu}_j \ge F_{\beta}(\mu)(M_j) = e^{-\beta} (\AAA_E \vec{\mu})_j.
$
The set $S \coloneqq \{i \in I \mid \vec{\mu}_i > 0\}$ is nonempty since $\|\vec{\mu}\|_1 = 1$, and the subgraph $E|_S$ of $E$ contains at least one recurrent component $R$. If $j \in I \setminus S$, then
\[
0 \le \sum_{i \in I} \AAA_E(j,i) \vec{\mu}_i = (\AAA_E \vec{\mu})_j \le e^{\beta} \vec{\mu}_j = 0.
\]
Since each term of the sum is nonnegative, $\AAA_E(j,i) = 0$ for all $i \in S$. Hence, there are no edges in $E$ with source in $S$ and range in $I \setminus S$, so $R$ is also a recurrent component of $E$.

Let $\vec{\mu}_R \coloneqq (\vec{\mu}_i)_{i \in R}$. Since $R \subseteq S$ we have $\vec{\mu}_R > 0$ entrywise. Let $\AAA_R$ be the  adjacency matrix of $E|_R$. For each $j \in R$,
\[
(\AAA_R \vec{\mu}_R)_j = \sum_{i \in R} \AAA_E(j,i) \vec{\mu}_i \le \sum_{i \in S} \AAA_E(j,i) \vec{\mu}_i = (\AAA_E \vec{\mu})_j \le e^{\beta} \vec{\mu}_j.
\]
Since $\AAA_R$ is irreducible \cite[Theorem~1.6]{Sen06} implies that $\rho(\AAA_R) \le e^{\beta}$. The claim follows. 
\end{proof}
\begin{cor}\label{cor:lower_bound_contractive}
Suppose that $(\AA,\Gamma)$ is contractive and let $\beta > 0$. If $\mu \in T_{\beta}(A)$, then $\beta \ge \beta_c$. In particular, if $0 < \beta < \beta_c$, then $T_{\beta}(A) = \varnothing$ and there are no $\KMS_\beta$ states on $\Oo_{\Gamma}$ or on $\Tt_{\Gamma}$.
\end{cor}
\begin{proof}
By Proposition~\ref{prop:transient_subcritical} the graph $E$ has a unique recurrent component $R$ and $\rho(\AAA_R) = \rho$, so the minimum in Proposition~\ref{prop:lower_bound} is $\ln(\rho) = \beta_c$. The final statement follows from Theorem~\ref{thm:LN_machine}.
\end{proof}

When $\beta > \beta_c$ the powers of $F_\beta$ decay geometrically in norm and we use this to show that every $\KMS_\beta$ state is of finite type in this regime. 
\begin{lem}\label{lem:spectral_decay}
Let $\beta > \beta_c$. There exist $C \ge 1$ and $0 < q < 1$ such that
	\begin{equation}\label{eq:spectral_decay}
	\|F_\beta^n(\mu)\| \le C\, q^n\, \|\mu\| 
	\end{equation}
	for all $\mu \in A_+^*$ and $n \in \NN$. 
	In particular, $\sum_{n=0}^\infty F_\beta^n(\mu)$ converges absolutely in norm for all $\mu \in A_+^*$. 
\end{lem}
\begin{proof}
	Fix $\mu \in A_+^*$ and let $\vec{\mu} = (\mu(M_i))_{i \in I}$. Then $\|\vec{\mu}\|_1 = \mu(\AA) = \|\mu\|$. By Lemma~\ref{lem:PF_operator_markov_partition}, 
	\[
	\|F_\beta^n(\mu)\| = e^{-n\beta}\|\AAA_E^n \vec{\mu}\|_1 \le e^{-n\beta}\|\AAA_E^n\|_1\,\|\mu\|,
	\]
where $\|\AAA_E^n\|_1 = \max_{i \in I} \sum_{j \in I} \AAA_E^n(j,i)$ is the operator $1$-norm of $\AAA_E^n$. The spectral radius formula gives $\lim_{n \to \infty} \|\AAA_E^n\|_1^{1/n} = \rho(\AAA_E) = e^{\beta_c}$, so for any $\beta' \in (\beta_c,\beta)$ there exists $C \ge 1$ such that $\|\AAA_E^n\|_1 \le C e^{n\beta'}$ for all $n \in \NN$. Setting $q \coloneqq e^{\beta' - \beta} < 1$ gives \eqref{eq:spectral_decay}. Now,
\[
\sum_{n=0}^\infty \|F_\beta^n(\mu)\| \le C \|\mu\| \sum_{n=0}^\infty q^n = \frac{C}{1-q}\, \|\mu\| < \infty. \qedhere
\]
\end{proof}

\begin{prop}\label{prop:above_critical_finite_type}
Let $\beta > \beta_c$. Then $T_\beta^\infty(A) = \varnothing$ and $T_\beta(A) = T_\beta^f(A)$. In particular, every $\KMS_\beta$ state on $\Oo_{\Gamma}$ and on $\Tt_{\Gamma}$ is of finite type.
\end{prop}
\begin{proof}
Fix $\mu \in T_\beta(A)$ and write $\mu = \mu^{\infty} + \mu^{f}$ as in \eqref{eq:infinite_finite_decomposition}.
Lemma~\ref{lem:spectral_decay} gives $\mu^{\infty} = \lim_n F_\beta^n(\mu) = 0$. Hence $\mu = \mu^{f}$ is a tracial state of finite type, so $T_\beta(A) = T_\beta^f(A)$. If $\mu \in T_\beta^\infty(A)$, then $\mu = F_\beta^n(\mu) \to 0$ in norm by Lemma~\ref{lem:spectral_decay}, contradicting $\mu(1_\AA) = 1$, so $T_\beta^{\infty}(A) = \varnothing$.  By Theorem~\ref{thm:LN_machine} every $\KMS_\beta$ state on $\Oo_{\Gamma}$ and on $\Tt_{\Gamma}$ arises from a $\beta$-subinvariant tracial state on $A$, so every $\KMS_\beta$ state is of finite type.
\end{proof}

\subsection{Infinite-type measures}\label{subsec:inf_type}

This subsection establishes the existence and uniqueness of an infinite-type $\KMS$ state at the critical inverse temperature $\beta_c = \ln(\rho)$. The central object is a unique $F_{\beta_c}$-invariant Borel probability measure $\mu_\infty$ on $\AA$, that we call the \emph{intrinsic Hutchinson measure}. The existence and uniqueness of $\mu_{\infty}$ depend on $(\AA,\Gamma)$ being \emph{contractive}, so we assume that $\AA$ admits a compatible metric $d$ such that each $\gamma \in \Gamma$ is a contraction with respect to $d$.

\begin{thm}[Existence and uniqueness of an invariant measure]\label{thm:inf_type}
	Let $(\AA,\Gamma)$ be a contractive iterated function system admitting a finite Markov partition. Then $\beta_c \coloneqq \ln(\rho)$ is the unique value of $\beta$ such that there exists a Borel probability measure $\mu$ on $\AA$ satisfying $F_{\beta}(\mu) = \mu$. Moreover,
	\begin{enumerate}[label=(\roman{enumi})]
		\item there exists a unique probability measure $\mu_\infty$ on $\AA$ satisfying $F_{\beta_c}(\mu_\infty) = \mu_\infty$; and
		\item if $\Mm = \{M_i\}_{i \in I}$ is the minimal Markov partition, then the vector $(\mu_\infty(M_i))_{i \in R}$ is the unique right Perron--Frobenius eigenvector of $\AAA_R$ with unit $1$-norm, and $\mu_\infty(M_i) = 0$ for all $i \in I \setminus R$.
	\end{enumerate}
	In particular, the $\KMS_{\beta_c}$ state that Theorem~\ref{thm:LN_machine} associates to $\mu_\infty$ is the unique $\KMS$ state for the gauge action on $\Oo_{\Gamma}$ at inverse temperature $\beta_c$, and it is of infinite type.
\end{thm}

The proof of Theorem~\ref{thm:inf_type} proceeds in two stages. First we determine that $\ln(\rho)$ is the only admissible value of $\beta$ using Perron--Frobenius theory. We then analytically establish the existence and uniqueness of $\mu_\infty$ using an argument inspired by the proof of the existence and uniqueness of Hutchinson measures \cite[\S~4.4]{Hut81}.

We begin with the combinatorial stage. Since $(\AA,\Gamma)$ is contractive, $R$ is the unique recurrent component of $E$ by Proposition~\ref{prop:transient_subcritical}. So $T = I \setminus R$ and there are no edges from $R$ to $T$. As such, we fix a basis of $\RR^I$ such that the adjacency matrix $\AAA_E$ is written in the block lower-triangular form
\begin{equation}\label{eq:block}\begin{pmatrix}
		\AAA_T & 0\\
		\BBB & \AAA_R
	\end{pmatrix},
\end{equation}
where $\AAA_R$ and $\AAA_T$ are the adjacency matrices of $E|_R$ and $E|_T$, respectively, and $\BBB$ is some matrix with nonnegative entries that describes the edges from $T$ to $R$. 

We record the following technical lemmas.
\begin{lem}\label{lem:transient_vanishing}
	If $\vec{x} \in \RR^T$ entrywise satisfies $\vec{x} \ge 0$, 
	$\BBB \vec{x} = 0$, and $\AAA_T \vec{x} \le \lambda \vec{x}$ for some $\lambda \ge 0$,  then $\vec{x} = 0$.
\end{lem}
\begin{proof}
	Let $S \coloneqq \{i \in T \mid \vec{x}_i > 0\}$ and suppose for contradiction 
	that $S \ne \varnothing$. Since $\BBB \vec{x} = 0$ and $\BBB$ has nonnegative 
	entries, $\BBB(j,i)\vec{x}_i = 0$ for all $j \in R$ and $i \in T$. In 
	particular, $\BBB(j,i) = 0$ for all $j \in R$ and $i \in S$, so there are no 
	edges from $S$ to $R$.
	For $i \in T \setminus S$ we have $\vec{x}_i = 0$, so $(\AAA_T\vec x)_i \le \lambda \vec x_i = 0$ gives 
	$\sum_{k \in T} \AAA_T(i,k) \vec{x}_k = 0$. Since each term of the sum is 
	nonnegative, $\AAA_T(i,k) = 0$ for all $k \in S$. Hence, there are no edges 
	from $S$ to $T \setminus S$.
	Thus, $S \subseteq T$ is a nonempty set with no outgoing edges in $E$, so $S$ 
	contains a recurrent component distinct from $R$, 
	contradicting the uniqueness of $R$. Hence, $S = \varnothing$ and $\vec{x} = 0$. 
\end{proof}

\begin{lem}\label{lem:PF_rigidity}
Suppose that $\vec{x} \in \RR^I$ satisfies $\vec{x} \ge 0$, $\vec{x} \ne 0$, and $\AAA_E \vec{x} \le \rho\, \vec{x}$ entrywise. Write $\vec{x} = (\vec{x}_T, \vec{x}_R)$, where $\vec{x}_T \in \RR^T$ and $\vec{x}_R \in \RR^R$. Then $\vec{x}_T = 0$ and $\vec{x}_R$ is a right Perron--Frobenius eigenvector of $\AAA_R$ with eigenvalue $\rho$.
\end{lem}
\begin{proof}
The block decomposition \eqref{eq:block} gives
\begin{equation}\label{eq:rigidity_block}
\AAA_T \vec{x}_T \le \rho\, \vec{x}_T \quad \text{and} \quad \AAA_R \vec{x}_R + \BBB \vec{x}_T \le \rho\, \vec{x}_R.
\end{equation}
For contradiction, suppose that $\vec{x}_R = 0$. The second inequality forces $\BBB \vec{x}_T = 0$ (since $\BBB\vec{x}_T \ge 0$). The first inequality and Lemma~\ref{lem:transient_vanishing} yield $\vec{x}_T = 0$, contradicting $\vec{x} \ne 0$. Hence $\vec{x}_R \ne 0$.

Since $\BBB \vec{x}_T \ge 0$, the second inequality of \eqref{eq:rigidity_block} gives $\AAA_R \vec{x}_R \le \rho\, \vec{x}_R$. Since $R$ is a recurrent component, $\AAA_R$ is irreducible, and since $\rho(\AAA_R) = \rho$ by Proposition~\ref{prop:transient_subcritical}, \cite[Theorem~1.5]{Sen06} yields a strictly positive left Perron--Frobenius eigenvector $\vec{u}_R^{\top}$ for $\AAA_R$ with eigenvalue $\rho$. Then
\[
0 \le \vec{u}_R^\top (\rho\, \vec{x}_R - \AAA_R \vec{x}_R) = \rho\, \vec{u}_R^\top \vec{x}_R - \vec{u}_R^\top \AAA_R \vec{x}_R = 0,
\]
so $\AAA_R \vec{x}_R = \rho\, \vec{x}_R$, and the second inequality of \eqref{eq:rigidity_block} forces $\BBB \vec{x}_T = 0$. The first inequality of \eqref{eq:rigidity_block} together with Lemma~\ref{lem:transient_vanishing} then gives $\vec{x}_T = 0$. Since $\vec{x}_R$ is an eigenvector of $\AAA_R$ with eigenvalue $\rho$, it is a Perron--Frobenius eigenvector of $\AAA_R$ by \cite[Theorem~1.5]{Sen06}. 
\end{proof}

\begin{prop}\label{prop:inf_type_beta}
If $\mu\in T_\beta^{\infty}(A)$, then $\beta = \beta_c$ and the vector $\vec{w} = (\mu(M_i))_{i \in I}$ is the unique nonnegative right eigenvector of $\AAA_E$ with eigenvalue $\rho$ and $\|\vec{w}\|_1 =1$.
Moreover, $T_{\beta_c}(A) = T_{\beta_c}^\infty(A)$.
\end{prop}
\begin{proof}
Suppose that $\mu \in T_{\beta}^{\infty}(A)$. Since $(\AA,\Gamma)$ is contractive, Corollary~\ref{cor:lower_bound_contractive} gives $\beta \ge \ln(\rho)$. Proposition~\ref{prop:above_critical_finite_type} forces $\beta \le \ln(\rho)$, so $\beta = \ln(\rho)$.

By Lemma~\ref{lem:PF_operator_markov_partition}, if $F_\beta(\mu) = \mu$, then $\AAA_E \vec{w} = \rho\, \vec{w}$. Since $\mu$ is a probability measure, $\vec{w} \ge 0$, $\|\vec{w}\|_1 = 1$, and $\vec{w} \ne 0$. By Lemma~\ref{lem:PF_rigidity}, $\vec{w}_T = 0$ and $\vec{w}_R$ may be identified with the Perron--Frobenius eigenvector of $\AAA_R$ satisfying $\|\vec{w}_R\|_1 = \|\vec{w}\|_1 = 1$, which is unique by \cite[Theorem~1.5]{Sen06}. This gives uniqueness of $\vec{w}$ among nonnegative unit-$1$-norm eigenvectors of $\AAA_E$ with eigenvalue $\rho$. 

For the second statement, suppose that $\mu \in T_{\beta_c}(A)$. Then $F_{\beta_c}(\mu) \le \mu$, and Lemma~\ref{lem:PF_operator_markov_partition} implies $\AAA_E \vec{\mu} \le \rho \vec{\mu}$, where $\vec{\mu} = (\mu(M_i))_{i \in I}$. Since $\vec{\mu} \ge 0$ and $\|\vec{\mu}\|_1 = 1$, Lemma~\ref{lem:PF_rigidity} gives $\vec{\mu}_T = 0$ and $\AAA_R \vec{\mu}_R = \rho\, \vec{\mu}_R$, so \eqref{eq:block} gives
\[
\AAA_E \vec{\mu} = \begin{pmatrix} \AAA_T & 0 \\ \BBB & \AAA_R \end{pmatrix} \begin{pmatrix} 0 \\ \vec{\mu}_R \end{pmatrix} = \begin{pmatrix} 0 \\ \rho\, \vec{\mu}_R \end{pmatrix} = \rho\, \vec{\mu}.
\]
Hence, $\mu_0 \coloneqq \mu - F_{\beta_c}(\mu) \ge 0$ satisfies $\mu_0(M_j) = \vec{\mu}_j - \rho^{-1}(\AAA_E \vec{\mu})_j = 0$ for every $j \in I$. Then $\mu_0(\AA) = \sum_{j \in I} \mu_0(M_j) = 0$, so $\mu_0 = 0$. That is, $\mu \in T^{\infty}_{\beta_c}(A)$.
\end{proof}

\begin{rmk}
Contractivity of $(\AA,\Gamma)$ was used critically in the proof of Proposition~\ref{prop:inf_type_beta} to invoke Proposition~\ref{prop:transient_subcritical}. If $E$ admitted multiple recurrent components, then it is possible that there are additional values of $\beta$ for which $T_\beta^\infty(A)$ is nonempty.
\end{rmk}

We have deduced that \emph{if} there is an invariant probability measure for $F_{\beta}$ then it must occur when $\beta = \beta_c$. What remains is to show the existence and uniqueness of such a measure.

For the remainder of this section we fix $\beta = \beta_c$, assume that $(\AA,\Gamma)$ is contractive, and let $\vec{w} = (w_i)_{i \in I}$ be the unique nonnegative right eigenvector of $\AAA_E$ with eigenvalue $\rho$ and $\|\vec{w}\|_1 = 1$. Consider the set of probability measures 
\[
T_{\vec{w}} \coloneqq \{ \mu \in T(A) \mid \mu(M_i) = w_i \text{ for all }i \in I\}.
\]
This set is nonempty: it contains the sum of point masses $\sum_{i \in I} w_i \delta_{x_i}$ for arbitrary $x_i \in M_i$. The operator $F_\beta$ also maps $T_{\vec{w}}$ into itself: by Lemma~\ref{lem:PF_operator_markov_partition}, for each $\mu \in T_{\vec{w}}$ and $j \in I$,
\[
F_\beta(\mu)(M_j) = \rho^{-1}\sum_i \AAA_E(j,i) w_i = w_j.
\]

Let $\Lip_1(\AA)$ denote the collection of Lipschitz functions $f \colon \AA \to \RR$ with Lipschitz constant at most $1$. 
The \emph{Monge--Kantorovich--Rubinstein} distance between finite Borel measures $\mu$ and $\nu$ on $\AA$ is defined by
\[
d_{\Lip}(\mu,\nu) = \sup_{a \in \Lip_1(\AA)} \big|\mu(a) - \nu(a)\big|
\]
with values in $[0,\infty]$. Since $\AA$ is compact and metrisable, the distance $d_{\Lip}$ defines a metric on $T(A)$ that metrises the {weak-$*$} topology (cf.~\cite[Theorem~6.9]{Vil09}). The distance satisfies $d_{\Lip}(k\mu,k\nu) = k\, d_{\Lip}(\mu,\nu)$ for all $k \ge 0$, so by rescaling we see that $d_{\Lip}$ is also finite if $\mu(\AA) = \nu(\AA)$. Conversely, if $\mu(\AA) \ne \nu(\AA)$, then $d_{\Lip}(\mu,\nu) = \infty$ as  
 $\mu(a+k1_\AA) - \nu(a+k1_\AA) = \mu(a) - \nu(a) + k\,(\mu(\AA) - \nu(\AA))$ for every $k \in \RR$. Hence, care must be taken when applying $d_{\Lip}$ to ensure that measures being compared have equal total mass.

For each $\mu \in T_{\vec{w}}$ let $\mu_i$ be the measure on $\AA$ given by $\mu_i(S) \coloneqq \mu(S \cap M_i)$. Since each $M_i$ is a Borel set, each $\mu_i$ is a finite Borel measure on $\AA$ with $\mu_i(M_i) = \mu(M_i) = w_i$. 
If $\mu,\nu \in T_{\vec{w}}$, then $d_{\Lip}(\mu_i,\nu_i)$ is finite for each $i$ since $\mu_i(\AA) = w_i = \nu_i(\AA)$.
The map $d_{\Mm} \colon T_{\vec{w}} \times T_{\vec{w}} \to [0,\infty)$ given by
\[
	d_{\Mm}(\mu,\nu) \coloneqq \big\| (d_{\Lip}(\mu_i,\nu_i))_{i \in I}\big\|_1  = \sum_{i \in I} d_{\Lip}(\mu_i,\nu_i),
\]
defines a metric on $T_{\vec{w}}$. Indeed, symmetry and the triangle inequality are inherited from $d_{\Lip}$, and definiteness holds because $\Mm$ is a partition of $\AA$: if $d_{\Mm}(\mu,\nu) = 0$ then $\mu_i = \nu_i$ for all $i$, so $\mu = \sum_{i \in I} \mu_i = \sum_{i \in I} \nu_i = \nu$. 
Observe that for all $\mu,\nu \in T_{\vec{w}}$ we have,
\begin{equation}\label{eq:lip_bound_M}
d_{\Lip}(\mu,\nu) = \sup_{a \in \Lip_1(\AA)} \Big| \sum_{i \in I} \big(\mu_i(a) - \nu_i(a)\big) \Big| \le d_{\Mm}(\mu,\nu).
\end{equation}
We use this metric to prove the existence of our invariant measure $\mu_\infty$ for $F_\beta$.
We adapt Hutchinson's argument \cite[\S~4.4]{Hut81}. Unlike Hutchinson's operator $\mu \mapsto \sum_{\gamma \in \Gamma} \omega_\gamma \gamma_*(\mu)$ (see Subsection~\ref{subsec:PF}), the operator $F_\beta$ is neither $d_{\Lip}$-contractive nor weak-$*$ continuous (see Example~\ref{ex:not_weak_star_continuous}), so the Banach Fixed-Point Theorem does not apply directly. Instead, we show that $(F_\beta^n(\mu))_{n \in \NN}$ is Cauchy with respect to $d_{\Lip}$, and then that its limit $\mu_\infty$ is the unique fixed point of $F_\beta$.

\begin{prop}\label{prop:mu_inf_exists} 
	Let $\beta = \ln(\rho)$. Then there exists $N \in \NN$ such that $F_{\beta}^n$ is $d_{\Mm}$-contractive on $T_{\vec{w}}$ for all $n \ge N$. Moreover, for any $\mu \in T_{\vec{w}}$ the sequence $(F^n_{\beta}(\mu))_{n \in \NN}$ converges to a probability measure $\mu_\infty$ on $\AA$ in the weak-$*$ topology.
\end{prop}

\begin{proof}
Given a finite Borel measure $\mu$ on $\AA$ and $i \in I$, we write $\mu_i$ for its restriction to $M_i$, so that $\mu = \sum_{i  \in I}\mu_i$. Given a continuous map $g \colon \AA \to \AA$ we let $g_*\mu$ denote the pushforward measure, given by $g_* \mu(a) = \mu (a \circ g)$ for all $a \in C(\AA)$, and recall that $g_* \mu(S) = \mu(g^{-1}(S))$ for all Borel $S \subseteq \AA$.

Using Lemma~\ref{lem:markov_partition_paths} we choose representatives $\gamma_e \in \Gamma$ for each $e \in E^1$, so that for all $n \in \NN$, $i \in I$, and $y \in M_i$ there is a bijection between $E^n i$ and $\Gamma^{(n)}y$ given by taking $e_1 \cdots e_n \in E^n i$ to $(\gamma_{e_1} \circ \cdots \circ \gamma_{e_n}(y), \ldots , \gamma_{e_n}(y), y) \in \Gamma^{(n)}y$. If $\alpha = e_1 \cdots e_n \in E^n$ we write $\gamma_\alpha$ for the composition $\gamma_{e_1} \circ \cdots \circ \gamma_{e_n}$. By \eqref{eq:n_step_PF}, for any finite Borel measure $\mu$ on $\AA$, $n \in \NN$, $i \in I$, and $a \in C(\AA)$, we have
\begin{align*}
F_\beta^n(\mu_i)(a)
&= \frac{1}{\rho^n}\int_{M_i} \sum_{x \in \Gamma^{(n)}y} a(x_0)\, d\mu(y)
= \frac{1}{\rho^n}\int_{M_i} \sum_{\alpha \in E^n i} a(\gamma_\alpha(y))\, d\mu(y)\\
&= \frac{1}{\rho^n}\sum_{\alpha \in E^n i}\mu_i(a \circ \gamma_\alpha) = \frac{1}{\rho^n}\sum_{\alpha \in E^n i}(\gamma_\alpha)_*\mu_i(a).
\end{align*}
That is,  $F_\beta^n(\mu_i) = \tfrac{1}{\rho^n}\sum_{\alpha \in E^n i}(\gamma_\alpha)_*\mu_i$ and by linearity, 
\[
F_\beta^n(\mu) = \frac{1}{\rho^n} \sum_{i \in I} \sum_{\alpha \in E^n i} (\gamma_\alpha)_*\mu_i = \frac{1}{\rho^n}  \sum_{j \in I} \sum_{i \in I} \sum_{\alpha \in jE^n i} (\gamma_\alpha)_*\mu_i.
\]
 Since $\gamma_{\alpha}(M_i) \subseteq M_{r(\alpha)}$ by \ref{itm:markov_property}, for each $\alpha \in j E^n i$ we have $(\gamma_\alpha)_*\mu_i(\AA \setminus M_j) = 0$. Hence, the restriction of $F_\beta^n(\mu)$ to $M_j$ is
\begin{equation*}
F_\beta^n(\mu)_j = \frac{1}{\rho^n}\sum_{i \in I}\sum_{\alpha \in jE^n i}(\gamma_\alpha)_*\mu_i.
\end{equation*}
Fix $\mu,\nu \in T_{\vec{w}}$. Then
\begin{align*}
d_{\Mm}(F_\beta^n(\mu),F_\beta^n(\nu))
&= \sum_{j \in I} d_{\Lip}\big(F_\beta^n(\mu)_j,F_\beta^n(\nu)_j\big)\\
&= \frac{1}{\rho^n} \sum_{j \in I} d_{\Lip}\Big( \sum_{i \in I}\sum_{\alpha \in jE^n i}(\gamma_\alpha)_*\mu_i, \sum_{i \in I}\sum_{\alpha \in jE^n i}(\gamma_\alpha)_*\nu_i\Big)\\
&\le \frac{1}{\rho^n}\sum_{j \in I}\sum_{i \in I}\sum_{\alpha \in jE^n i} d_{\Lip}\big((\gamma_\alpha)_*\mu_i,(\gamma_\alpha)_*\nu_i\big)\\
&= \frac{1}{\rho^n}\sum_{i \in I}\sum_{\alpha \in E^n i} d_{\Lip}\big((\gamma_\alpha)_*\mu_i,(\gamma_\alpha)_*\nu_i\big).
\end{align*}
Let $c_\gamma < 1$ denote the Lipschitz constant of $\gamma \in \Gamma$, and let $C \coloneqq \max\{c_\gamma \mid \gamma \in \Gamma\}$. We may assume $C>0$, else $\AA$ is a single point. For each $\alpha \in E^n$, $\gamma_\alpha$ is a composition of $n$ maps in $\Gamma$, so $\gamma_\alpha$ has Lipschitz constant bounded by $C^n$. Hence, $C^{-n}(a \circ \gamma_\alpha) \in \Lip_1(\AA)$ for every $a \in \Lip_1(\AA)$. Therefore
\[
d_{\Lip}\big((\gamma_\alpha)_*\mu_i,(\gamma_\alpha)_*\nu_i\big)
= \sup_{a \in \Lip_1(\AA)}\big|\mu_i(a \circ \gamma_\alpha)-\nu_i(a \circ \gamma_\alpha)\big|
\le C^n\, d_{\Lip}(\mu_i,\nu_i),
\]
and so
\[
d_{\Mm}(F_\beta^n(\mu),F_\beta^n(\nu))
\le \Big(\frac{C}{\rho}\Big)^n \sum_{i \in I} |E^n i|\, d_{\Lip}(\mu_i,\nu_i).
\]
Let $\vec{u} \coloneqq (d_{\Lip}(\mu_i,\nu_i))_{i \in I}$. Then
\[
\sum_{i \in I}|E^n i|\, \vec{u}_i = \sum_{i \in I}\sum_{j \in I}| j E^n i|\, \vec{u}_i = \sum_{i \in I}\sum_{j \in I}\AAA_E^n(j,i)\vec{u}_i = \sum_{j \in I}(\AAA_E^n\vec{u})_j = \|\AAA_E^n\vec{u}\|_1.
\]
Consequently,
\begin{equation}\label{eq:contractive_bound}
d_{\Mm}(F_\beta^n(\mu),F_\beta^n(\nu))
\le C^n\rho^{-n}\|\AAA_E^n\vec{u}\|_1
\le C^n\rho^{-n}\|\AAA_E^n\|_1\,\|\vec{u}\|_1
= C^n\rho^{-n}\|\AAA_E^n\|_1\, d_{\Mm}(\mu,\nu),
\end{equation}
where $\|\AAA_E^n\|_1 = \max_{i} \sum_{j} \AAA_E^n(j,i)$ is the operator $1$-norm of $\AAA_E^n$. 
	
	Fix $0< \varepsilon < \rho(C^{-1} - 1)$.
	Since $\lim_{k \to \infty}\|\AAA_E^k\|_1^{1/k} = \rho$ there exists $N \in \NN$ such that for all $n \ge N$, we have $\|\AAA^n_{E}\|_1^{1/n} < \rho + \varepsilon$. So for $n \ge N$,
	\[
	C^n \rho^{-n} \|\AAA^n_{E}\|_1 \le C^n \rho^{-n} (\rho + \varepsilon)^n = C^n(1+\varepsilon\rho^{-1})^n < 1.
	\]
	Let $K \coloneqq C(1+\varepsilon\rho^{-1}) < 1$. 
	By \eqref{eq:contractive_bound}, for  $n\ge N$, $F_\beta^n \colon T_{\vec{w}} \to T_{\vec{w}}$ is $d_\Mm$-contractive with Lipschitz constant bounded by $K^n$. 

	Fix $\mu \in T_{\vec{w}}$.
For convergence of $F_\beta^n(\mu)$ observe that for $m \ge n \ge N$, 
\begin{align*}
d_{\Mm}(F^{m}_\beta(\mu),F^{n}_\beta(\mu)) &\le d_{\Mm}(F^{m}_\beta(\mu),F^{m-1}_\beta(\mu)) + \cdots + d_\Mm (F^{n+1}_\beta(\mu),F^{n}_\beta(\mu))\\
&\le (K^{m-1} + \cdots +  K^n) \, d_{\Mm} (F_\beta(\mu),\mu)\\
&\le \frac{K^n}{1-K} \, d_{\Mm} (F_\beta(\mu),\mu). 
\end{align*}
Since $K < 1$ the sequence $(F_\beta^n(\mu))_{n \in \NN}$ is $d_{\Mm}$-Cauchy, so by \eqref{eq:lip_bound_M} it is also $d_{\Lip}$-Cauchy.
Since $d_{\Lip}$ metrises the weak-$*$ topology on $T(A)$, which is weak-$*$ compact (hence complete) by the Banach--Alaoglu Theorem, there is a probability measure $\mu_\infty$ on $\AA$ such that $F_\beta^n(\mu) \to \mu_\infty$ in the weak-$*$ topology. 
\end{proof}

Although the limit $\mu_\infty$ of Proposition~\ref{prop:mu_inf_exists} is a well-defined Borel probability measure, it is not immediately clear that $\mu_\infty \in T_{\vec{w}}$, nor is it clear that $\mu_{\infty}$ is a fixed point of $F_\beta$. 

\begin{lem}\label{lem:mu_inf_subinvariant_and_mass}
	Let $\beta = \ln(\rho)$, let $\mu \in T_{\vec{w}}$, and let $\mu_\infty$ be the weak-$*$ limit of $(F_\beta^n(\mu))_{n \in \NN}$, which exists by Proposition~\ref{prop:mu_inf_exists}. Then $F_{\beta}(\mu_\infty) = \mu_\infty$ and $\mu_\infty \in T_{\vec{w}}$. 
\end{lem}
\begin{proof}
	Since $F_{\beta}^n(\mu) \overset{w^*}{\to} \mu_{\infty}$ the Portmanteau Theorem \cite[Theorem~1.3.4]{vdVW96} says that $\liminf F_{\beta}^n(\mu)(h) \ge \mu_{\infty}(h)$ for all bounded lower semicontinuous $h \colon \AA \to \RR$. By Lemma~\ref{lem:dual_Hutchinson_op} the map $\Gamma^*(a)$ is bounded lower semicontinuous for all $a \in A_+$. So
	\begin{align*}
	\mu_\infty(a) 
	= \lim_{n \to \infty} F_\beta^{n+1}(\mu)(a) 
	= \lim_{n \to \infty} F_\beta^{n}(\mu)\big(\rho^{-1}\Gamma^*
	(a)\big) 
	\ge  \mu_{\infty}\big(\rho^{-1}\Gamma^*(a)\big)
	= F_\beta(\mu_\infty)(a)
	\end{align*}
	for all $a \in A_+$. That is, $F_\beta(\mu_\infty) \le \mu_\infty$, so $\mu_\infty \in T_\beta(A)$. Proposition~\ref{prop:inf_type_beta} gives $\mu_\infty \in T^\infty_\beta(A)$, that is, $F_\beta(\mu_\infty) = \mu_\infty$, and $(\mu_\infty(M_i))_{i \in I} = \vec{w}$, so $\mu_\infty \in T_{\vec{w}}$.
\end{proof}

We are now in a position to prove  Theorem~\ref{thm:inf_type}.
\begin{proof}[Proof of Theorem~\ref{thm:inf_type}]
	By Proposition~\ref{prop:inf_type_beta}, $\beta = \ln(\rho)$ is the only inverse temperature for which there can exist an $F_\beta$-invariant probability measure. 
	Fix a measure $\mu \in T_{\vec{w}}$ and let $\mu_{\infty} = \lim_{n \to \infty} F^n_{\beta}(\mu)$ as in Proposition~\ref{prop:mu_inf_exists}.	
	Lemma~\ref{lem:mu_inf_subinvariant_and_mass} gives $F_{\beta}(\mu_\infty) = \mu_\infty$ and $\mu_\infty \in T_{\vec{w}}$.	
	
	For uniqueness suppose that $\tau \in T(A)$ is another fixed point of $F_{\beta}$. Then $\tau \in T_{\vec{w}}$ by Proposition~\ref{prop:inf_type_beta}.	
	By Proposition~\ref{prop:mu_inf_exists} there exist $K < 1$ and $N \in \NN$ such that
	\[
	d_{\Mm}(\mu_{\infty},\tau) = d_{\Mm}(F_\beta^N(\mu_{\infty}),F_\beta^N(\tau)) \le K d_{\Mm}(\mu_{\infty},\tau).
	\]
	Then $d_{\Mm}(\mu_{\infty},\tau) = 0$, so $\mu_{\infty} = \tau$. 

	 Proposition~\ref{prop:inf_type_beta} and the affine isomorphism of Theorem~\ref{thm:LN_machine} give the final statement.
\end{proof}

We recall from \cite[\S~4]{Hut81} that for any contractive iterated function system $(\AA,\Gamma)$ there is a unique Borel probability measure $\mu_H$ on $\AA$ satisfying
\begin{equation}\label{eq:Hutchinson_identity}
\mu_H = \frac{1}{|\Gamma|}\sum_{\gamma \in \Gamma} \gamma_*\mu_H,
\end{equation}
namely Hutchinson's self-similar measure for the uniform weights $1/|\Gamma|$. Since it depends on weights assigned to the labelled maps in $\Gamma$, and changes if a map is duplicated, we call $\mu_H$ the \emph{extrinsic Hutchinson measure} of $(\AA,\Gamma)$, to distinguish it from the intrinsic Hutchinson measure $\mu_\infty$.

In \cite[Theorem~6.6]{IKW07} it is shown that if $(\AA,\Gamma)$ is a contractive iterated function system in which the set of points at which distinct maps in $\Gamma$ agree is finite, and which satisfies the orbit escape condition (Definition~\ref{dfn:orbit_escape_condition}), then there is a unique infinite-type $\KMS$ state at $\beta = \ln|\Gamma|$, and its associated measure is $\mu_H$. For contractive iterated function systems satisfying the open set condition, the first author's thesis \cite{Mam19} shows, with no assumption on the size of $p_1(B_\Gamma)$, that there is a unique $\KMS$ state at $\beta = \ln|\Gamma|$, that it is of infinite type, and that its associated measure is $\mu_H$. The following corollary recovers this in our framework.

\begin{cor}[{cf.~\cite[Proposition~6.36 and Corollary~6.49]{Mam19}}]\label{cor:OSC_Hutchinson}
Let $(\AA,\Gamma)$ be a contractive iterated function system satisfying the open set condition and admitting a finite Markov partition. Then $\beta_c = \ln|\Gamma|$ and $\mu_{\infty} = \mu_H$. 
\end{cor}
\begin{proof}

Let $U$ be a nonempty open set witnessing the open set condition. Since $\gamma(U) \subseteq U$ for every $\gamma \in \Gamma$, every composition of maps from $\Gamma$ carries $U$ into $U$. Fix $n \ge 1$ and let $(\lambda_1,\ldots,\lambda_n)$ and $(\lambda_1',\ldots,\lambda_n')$ be distinct elements of $\Gamma^n$. Take the least $k$ such that $\lambda_k \ne \lambda_k'$ and put $\psi \coloneqq \lambda_1 \circ \cdots \circ \lambda_{k-1}$, interpreted as the identity when $k = 1$. Then
\[
\lambda_1 \circ \cdots \circ \lambda_n(U) \subseteq \psi\big(\lambda_k(U)\big)
\quad \text{ and } \quad
\lambda_1' \circ \cdots \circ \lambda_n'(U) \subseteq \psi\big(\lambda_k'(U)\big),
\]
while $\lambda_k(U) \cap \lambda_k'(U) = \varnothing$ by the open set condition. As $\psi$ is injective, the sets $\lambda_1 \circ \cdots \circ \lambda_n(U)$ and $\lambda_1' \circ \cdots \circ \lambda_n'(U)$ are disjoint. Hence, for any $y \in U$ the $|\Gamma|^n$ points $\lambda_1 \circ \cdots \circ \lambda_n(y)$, for $(\lambda_1,\ldots,\lambda_n) \in \Gamma^n$, are pairwise distinct, so $|\Gamma^{(n)}y| = |\Gamma|^n$. Consequently, $\sup_{x \in \AA}|\Gamma^{(n)}x| = |\Gamma|^n$ for every $n \in \NN$, so by Proposition~\ref{prop:rho_intrinsic} we have $\beta_c = \ln|\Gamma|$. 

For the equality of measures, fix a labelling $\Gamma = \{\gamma_1,\ldots,\gamma_N\}$ and for each $k \ge 1$ and $w = w_1\cdots w_k \in \{1,\ldots,N\}^k$ let $\gamma_w \coloneqq \gamma_{w_1} \circ \cdots \circ \gamma_{w_k}$. Let $K \coloneqq \AA \setminus U$. We claim that $\mu_H(K) = 0$.
Since $U$ is a nonempty open set, Lemma~\ref{lem:eventually_in_open_sets} provides $n \ge 1$ and $u \in \{1,\ldots,N\}^n$ such that $\gamma_u(\AA) \subseteq U$, and hence $\gamma_u^{-1}(K) = \varnothing$. Iterating \eqref{eq:Hutchinson_identity} gives
\[
\mu_H = \frac{1}{|\Gamma|^n}\sum_{w \in \{1,\ldots,N\}^n}(\gamma_w)_*\mu_H .
\]
For every $w \in \{1,\ldots,N\}^n$ the open set condition yields $\gamma_w(U) \subseteq U$, so $\gamma_w^{-1}(K) \subseteq K$. Evaluating on $K$ and discarding the vanishing term indexed by $u$ gives
\[
\mu_H(K) = \frac{1}{|\Gamma|^n}\sum_{w \in \{1,\ldots,N\}^n}\mu_H\bigl(\gamma_w^{-1}(K)\bigr) \le \frac{|\Gamma|^n - 1}{|\Gamma|^n}\,\mu_H(K).
\]
Since $\mu_H(K)$ is finite and $(|\Gamma|^n - 1)/|\Gamma|^n < 1$, this forces $\mu_H(K) = 0$.

As $\mu_H(U) = 1$, and $b(\gamma(y),y) = 1$ for all $y \in U$ and $\gamma \in \Gamma$, it follows that
\[
F_{\beta_c}(\mu_H)(a) = \frac{1}{|\Gamma|}\int_\AA \sum_{\gamma \in \Gamma} \frac{a(\gamma(y))}{b(\gamma(y),y)}\,d\mu_H(y) = \frac{1}{|\Gamma|}\int_\AA \sum_{\gamma \in \Gamma} a(\gamma(y))\,d\mu_H(y) = \mu_H(a).
\]
The uniqueness of the $F_{\beta_c}$-invariant measure in Theorem~\ref{thm:inf_type} forces $\mu_\infty = \mu_H$.
\end{proof}

When $B_\Gamma = \varnothing$, the finite Markov partition of Assumption~\ref{ass} is unnecessary. The weights $y \mapsto b(\gamma(y),y)^{-1}$ of $\Gamma^*$ in this case are locally constant, so \cite[Theorem~1.1]{FL99} applies directly.

\begin{prop}\label{prop:unique_kms_no_branch}
Let $(\AA,\Gamma)$ be a contractive iterated function system with $B_\Gamma = \varnothing$, and let $\rho$ be the spectral radius of $\Gamma^*$ on $C(\AA)$. Then $\rho > 1$, there is a unique probability measure $\nu$ on $\AA$ such that $F(\nu) = \rho\nu$, and the $\KMS_{\ln(\rho)}$ state associated to $\nu$ by Theorem~\ref{thm:LN_machine} is the unique $\KMS$ state for the gauge action on $\Oo_\Gamma$. If $(\AA,\Gamma)$ also admits a finite Markov partition, then $\ln(\rho) = \beta_c$ and $\nu = \mu_\infty$.
\end{prop}
\begin{proof}
Since $B_\Gamma = \varnothing$, by Lemma~\ref{lem:branch_iff_discontinuous} the branch index is continuous, so locally constant on $\gr(\Gamma)$. Each $\gamma \in \Gamma$ is a homeomorphism onto the compact set $\gamma(\AA)$, so $p_\gamma(x) \coloneqq b(x,\gamma^{-1}(x))^{-1}$ is locally constant on $\gamma(\AA)$. The function $p_\gamma$ takes finitely many values on pairwise disjoint compact sets so extends to a Lipschitz function on $\AA$ with values in $[1/|\Gamma|,1]$. As $\Gamma^*(a)(y) = \sum_{\gamma \in \Gamma} p_\gamma(\gamma(y))\,a(\gamma(y))$ for all $a \in C(\AA)$ and $y \in \AA$, the operator $\Gamma^*$ is the Ruelle operator of \cite{FL99} for these weights, each $\log p_\gamma$ satisfies the Dini condition, and $F(\mu)(a) = \mu(\Gamma^*(a))$ for $a \in C(\AA)$. Hence, \cite[Theorem~1.1]{FL99} gives a strictly positive $g \in C(\AA)$ and a unique probability measure $\nu$ with $\Gamma^*(g) = \rho g$ and $F(\nu) = \rho\nu$.
As $\|(\Gamma^*)^n\| = \sup_{y \in \AA}|\Gamma^{(n)}y|$ by \eqref{eq:n_step_dual}, the path bound $|\Gamma^{(m(k+1))}y| \ge 2^m$ of Remark~\ref{rmk:beta_c_positive} and the spectral radius formula give $\rho \ge 2^{1/(k+1)} > 1$.

Since $B_\Gamma = \varnothing$, Proposition~\ref{prop:covariance_ideal} gives $J_\Gamma = C(\AA)$, so by Theorem~\ref{thm:LN_machine} and \eqref{eq:trace_on_ideal} the $\KMS_\beta$ states on $\Oo_\Gamma$ correspond to the $\mu \in T(A)$ with $F(\mu) = e^{\beta}\mu$. For such $\mu$ we have $e^{\beta}\mu(g) = \mu(\Gamma^*(g)) = \rho\,\mu(g)$ with $\mu(g) > 0$, so $\beta = \ln(\rho)$ and $\mu = \nu$. Conversely, $F_{\ln(\rho)}(\nu) = \nu$, so $\nu$ determines a $\KMS_{\ln(\rho)}$ state on $\Oo_\Gamma$. If $(\AA,\Gamma)$ admits a finite Markov partition, then $\rho = \rho(\AAA_E)$ by Proposition~\ref{prop:rho_intrinsic}, so $\ln(\rho) = \beta_c$, and $\nu = \mu_\infty$ by the uniqueness in Theorem~\ref{thm:inf_type}.
\end{proof}

If no two maps in $\Gamma$ coincide at any point of $\AA$, then $B_\Gamma = \varnothing$ and the orbit escape condition holds trivially, so in this case Proposition~\ref{prop:unique_kms_no_branch} is covered by \cite[Theorem~6.6]{IKW07} (see also \cite[Theorem~4.18]{deC23}), with $\rho = |\Gamma|$ and $\nu = \mu_H$. When $\AA$ is connected this is the only way that $B_\Gamma = \varnothing$ can occur; Example~\ref{ex:cantor_not_hutchinson} below shows that Proposition~\ref{prop:unique_kms_no_branch} goes beyond \cite[Theorem~6.6]{IKW07} for disconnected attractors.

\subsection{Finite-type measures}

We now classify the finite-type $\beta$-subinvariant measures. By Proposition~\ref{prop:IKW_seed_bijection}, the finite-type $\KMS_\beta$ states correspond to the seeds $\tau \in T(A/J_\Gamma)$ for which $\sum_{n=0}^\infty F_\beta^n(\tau)$ converges. For $\beta > \ln|\Gamma|$ this is automatic for every iterated function system, since $\|F_\beta^n\| \le (e^{-\beta}|\Gamma|)^n$ by Lemma~\ref{lem:PF_operator}; this is the regime treated in \cite{IKW07}. The point of the following theorem is that summability persists for all $\beta > \beta_c$, which is a strictly larger range whenever $\rho < |\Gamma|$. Proposition~\ref{prop:covariance_ideal} identifies $A/J_\Gamma$ with $C(p_1(B_\Gamma))$, so $T(A/J_\Gamma)$ may be identified with $T(C(p_1(B_\Gamma)))$, the set of probability measures on $p_1(B_\Gamma)$.

\begin{thm}\label{thm:finite_type_classification}
Let $(\AA,\Gamma)$ be an iterated function system admitting a finite Markov partition, and let $\beta > \beta_c$. Then the following hold. 
\begin{enumerate}[label=\textup{(\arabic*)}]
	\item\label{item:all_finite_type} Every $\KMS_\beta$ state for the gauge action on $\Oo_{\Gamma}$ is of finite type.
	\item\label{item:T_beta} $T(A/J_\Gamma)_{\beta} = T(C(p_1(B_\Gamma)))$.
	\item\label{item:KMS_bijection} The map $\tau \mapsto \varphi_{\tau,\beta}$ of Proposition~\ref{prop:IKW_seed_bijection} is a bijection
	\[
	T(C(p_1(B_\Gamma))) \;\longleftrightarrow\; \KMS_\beta(\Oo_{\Gamma}),
	\]
	which restricts to a bijection on extreme points.
	\item\label{item:resolvent_formula} Let $\tau \in T(C(p_1(B_\Gamma)))$ and $\vec{\tau} = (\tau(M_i))_{i \in I}$. Then
	\begin{equation}\label{eq:resolvent_v2}
		\big(\varphi_{\tau,\beta}|_A(M_j)\big)_{j \in I} = m_{\tau,\beta}\,(\Id - e^{-\beta} \AAA_E)^{-1}\, \vec{\tau},
	\end{equation}
	where $m_{\tau,\beta} = \big\| (\Id - e^{-\beta} \AAA_E)^{-1} \vec{\tau} \big\|_1^{-1}$. 
\end{enumerate}
\end{thm}
\begin{proof}
 Proposition~\ref{prop:above_critical_finite_type} gives \ref{item:all_finite_type}.
Lemma~\ref{lem:spectral_decay} implies that $\sum_{n=0}^{\infty} F_\beta^n(\tau)$ converges absolutely in norm on $A^*$ for every $\tau \in T(C(p_1(B_\Gamma)))$ ensuring \ref{item:T_beta}. Proposition~\ref{prop:IKW_seed_bijection} now gives \ref{item:KMS_bijection}. 

For \ref{item:resolvent_formula}, Proposition~\ref{prop:IKW_seed_bijection} identifies $\varphi_{\tau,\beta}|_A$ with the normalised sum $m_{\tau,\beta}\sum_{n=0}^\infty F_\beta^n(\tau)$. By Lemma~\ref{lem:PF_operator_markov_partition}, $F_\beta^n(\tau)(M_j) = e^{-n\beta}(\AAA_E^n \vec{\tau})_j$ for each $j \in I$. Since $\rho(e^{-\beta}\AAA_E) = e^{-\beta}\rho < 1$, the Neumann series $(\Id - e^{-\beta}\AAA_E)^{-1} = \sum_{n=0}^\infty e^{-n\beta}\AAA_E^n$ converges, giving
\[
\sum_{n=0}^\infty F_\beta^n(\tau)(M_j) = \big((\Id - e^{-\beta}\AAA_E)^{-1}\vec{\tau}\big)_j.
\]
Summing over $j$ gives $m_{\tau,\beta}^{-1} = \|(\Id - e^{-\beta}\AAA_E)^{-1}\vec{\tau}\|_1$.
\end{proof}

\begin{rmk}[Affine reformulation]\label{rmk:un_normalised_seeds}
The bijection of Theorem~\ref{thm:finite_type_classification} need not be affine because of the data-dependent normalisation $m_{\tau,\beta}$. Absorbing the normalisation into the seed by setting $\mu' \coloneqq m_{\tau,\beta}\,\tau$ replaces $T(C(p_1(B_\Gamma)))$ by the convex set
\[
	\Sigma_{\beta} \coloneqq \Big\{\mu' \in A^*_+ \mid \supp(\mu') \subseteq p_1(B_{\Gamma}),\; \sum_{n=0}^{\infty} F_{\beta}^n(\mu')(1_\AA) = 1\Big\}.
\]
The affine map $\mu' \mapsto \sum_{n=0}^\infty F_\beta^n(\mu')$ has affine inverse $\mu \mapsto \mu - F_\beta(\mu)$, giving an affine isomorphism between $\Sigma_\beta$ and $\KMS_\beta(\Oo_{\Gamma})$.
\end{rmk}

Under the orbit escape condition and finiteness of the set of points at which distinct maps in $\Gamma$ agree, \cite[Theorem~6.6]{IKW07} parametrises the extreme finite-type $\KMS_\beta$ states for $\beta > \ln|\Gamma|$ by atomic seeds on the common values $\gamma(y) = \gamma'(y)$ of distinct maps $\gamma,\gamma' \in \Gamma$. Theorem~\ref{thm:finite_type_classification} instead assumes a finite Markov partition, replaces $\ln|\Gamma|$ by $\beta_c=\ln(\rho(\AAA_E))$, and  makes no assumptions on the size of $B_\Gamma$.

Examples~\ref{ex:clarence_no_oec} and~\ref{ex:clarence_minimal} show that systems admitting a finite Markov partition do not necessarily satisfy the hypotheses of \cite[Theorem~6.6]{IKW07}. Conversely, systems satisfying the hypotheses of \cite[Theorem~6.6]{IKW07} do not necessarily admit finite Markov partitions, as the following example demonstrates. 

\begin{example}\label{ex:ikw_no_finite_markov}
The system of Example~\ref{ex:ikw_no_markov} satisfies the hypotheses of \cite[Theorem~6.6]{IKW07}, but admits no finite Markov partition. Its only coincidence is $\gamma_1(y_0) = \gamma_3(y_0) = y_0/2$, so $y_0$ is the only point at which distinct maps in $\Gamma$ agree.
To see that it satisfies the orbit escape condition (Definition~\ref{dfn:orbit_escape_condition}), note that every point of $O(y_0/2)$ has the form $s y_0 + d$ with $s$ and $d$ rational and $0 \le s \le 1/2$, since this holds for $y_0/2$ and is preserved by each of $\gamma_1$, $\gamma_2$, and $\gamma_3$. As $y_0$ is irrational, $y_0 \notin O(y_0/2)$, so no point of $O(y_0/2)$ is a point at which distinct maps in $\Gamma$ agree. Since $y_0/2 = \gamma_1(y_0) \in O(y_0)$, the second condition of Definition~\ref{dfn:orbit_escape_condition} holds with $x = y_0$ and $y = y_0/2$. So the system satisfies the orbit escape condition. 
\end{example}

\begin{rmk}\label{rmk:countable_markov_possible}
If $(\AA,\Gamma)$ satisfies the hypotheses of \cite[Theorem~6.6]{IKW07}, then the set $C$ of Proposition~\ref{prop:countable_markov} is finite. Hence, $(\AA,\Gamma)$ admits a countable minimal Markov partition. An extension of our results to countable Markov partitions would therefore apply to the systems considered in \cite{IKW07}; we do not pursue this here. We note only that for a countable Markov partition $\Mm = \{M_i\}_{i \in I}$ the matrix $\AAA_E(i,j) = |\{e \in E^1 \mid r(e) = i \text{ and } s(e) = j\}|$ is still defined, and its column sums $|s^{-1}(j)|$ are bounded by $|\Gamma|$ by Lemma~\ref{lem:graph_properties}\ref{itm:E_s_inverse}, so $\AAA_E$ is a bounded operator on $\ell^1(I)$ acting by $(\AAA_E x)_i = \sum_{j \in I} \AAA_E(i,j)\,x_j$. The proof of Proposition~\ref{prop:rho_intrinsic} shows that its spectral radius is $\lim_{n \to \infty} h_n^{1/n}$, with $h_n$ as in \eqref{eq:path_growth}. The Perron--Frobenius arguments of Section~\ref{sec:classification}, however, rely on $I$ being finite.
\end{rmk}

To compare to \cite{IKW07} directly, we separately record the special case of finite $p_1(B_\Gamma)$.
\begin{cor}\label{cor:finite_branch_simplex}
Suppose that $(\AA,\Gamma)$ admits a finite Markov partition, that $\beta > \beta_c$, and that $p_1(B_\Gamma) = \{x_1,\ldots,x_m\}$ is finite. For each $1 \le k \le m$ let $i(k) \in I$ be such that $x_k \in M_{i(k)}$. Let $\vec{e}_i$ be the $i$-th standard basis vector in $\RR^I$ and let
$
C_{k,\beta} \coloneqq \big\| (\Id - e^{-\beta} \AAA_E)^{-1} \vec{e}_{i(k)} \big\|_1.
$
Then, in the affine reformulation of Remark~\ref{rmk:un_normalised_seeds},
\[
\Sigma_\beta = \Big\{ \sum_{k=1}^m \alpha_k \delta_{x_k} \Bigm| \alpha_k \ge 0 \text{ and } \sum_{k=1}^m \alpha_k C_{k,\beta} = 1 \Big\},
\]
which is a simplex with extreme points $C_{k,\beta}^{-1} \delta_{x_k}$ for $k = 1,\ldots,m$. Consequently, the extreme $\KMS_\beta$ states of $\Oo_{\Gamma}$ are the states $\varphi_{\delta_{x_k},\beta}$, and
\[
\varphi_{\delta_{x_k},\beta}|_A = C_{k,\beta}^{-1} \sum_{n=0}^\infty F_\beta^n(\delta_{x_k}), \qquad k = 1,\ldots,m.
\]
\end{cor}
\begin{proof}
Since $p_1(B_\Gamma)$ is finite, every $\mu' \in A^*_+$ with $\supp(\mu') \subseteq p_1(B_\Gamma)$ satisfies $\mu' = \sum_{k=1}^m \alpha_k \delta_{x_k}$, where $\alpha_k \coloneqq \mu'(\{x_k\}) \ge 0$. The seed $\delta_{x_k}$ has coordinate vector $\vec{e}_{i(k)}$, since $\delta_{x_k}(M_i) = 1$ if $i = i(k)$ and $\delta_{x_k}(M_i) = 0$ otherwise. Hence, by the definition of $m_{\tau,\beta}$ in Proposition~\ref{prop:IKW_seed_bijection} and Theorem~\ref{thm:finite_type_classification}\ref{item:resolvent_formula},
\[
\sum_{n=0}^\infty F_\beta^n(\delta_{x_k})(1_\AA) = m_{\delta_{x_k},\beta}^{-1} = \big\|(\Id - e^{-\beta} \AAA_E)^{-1} \vec{e}_{i(k)}\big\|_1 = C_{k,\beta}.
\]
 As $\mu' \mapsto \sum_{n=0}^\infty F_\beta^n(\mu')(1_\AA)$ is linear, the normalisation constraint defining $\Sigma_\beta$ becomes $\sum_{k=1}^m \alpha_k C_{k,\beta} = 1$, which gives the stated description of $\Sigma_\beta$. The map $\sum_{k=1}^m \alpha_k \delta_{x_k} \mapsto (\alpha_k C_{k,\beta})_{k=1}^m$ is an affine isomorphism from $\Sigma_\beta$ onto the standard $(m-1)$-simplex in $\RR^m$, so the extreme points of $\Sigma_\beta$ are the preimages of the standard basis vectors, namely $C_{k,\beta}^{-1} \delta_{x_k}$. The affine isomorphism of Remark~\ref{rmk:un_normalised_seeds} carries $C_{k,\beta}^{-1} \delta_{x_k}$ to $C_{k,\beta}^{-1} \sum_{n=0}^\infty F_\beta^n(\delta_{x_k}) = \varphi_{\delta_{x_k},\beta}|_A$, giving the final claim.
\end{proof}

\subsection{Examples}\label{subsec:examples}
In this subsection we give examples of iterated function systems and their associated $\KMS$ states that are not covered by \cite{IKW07} due to the presence of overlaps and infinite branch sets. 
\begin{example}\label{ex:clarence_kms}
Let $(\AA,\Gamma)$ be the iterated function system of Example~\ref{ex:clarence}. In Example~\ref{ex:clarence_minimal} we saw that $\Mm = \{ \{0\}, (0,1/2), [1/2,1), \{1\}\}$ is the unique minimal Markov partition, and in Example~\ref{ex:clarence_recurrent} we saw that $R = \{2,3\}$ is the unique recurrent component of $E$. One checks that $\rho(\AAA_E) = \rho(\AAA_R) = (3+\sqrt{5})/2$ and that the unique unit-$1$-norm eigenvector of $\AAA_E$ for $\rho$ is $\vec{w} = 1/2 \cdot (0, \sqrt{5}-1, 3 - \sqrt{5}, 0)$. 

By Theorem~\ref{thm:inf_type}, $\Oo_\Gamma$ has a unique $\KMS_{\beta_c}$ state $\varphi_\infty$.  It is of infinite type, and $\varphi_\infty|_A = \mu_\infty$, where $\mu_\infty$ is the unique $F_{\beta_c}$-invariant probability measure, with $\mu_\infty(M_1) = \mu_\infty(M_4) = 0$, $\mu_\infty(M_2) = (\sqrt{5}-1)/2$ and $\mu_\infty(M_3) = (3-\sqrt{5})/2$. For $\beta > \beta_c$, Theorem~\ref{thm:finite_type_classification} gives a bijection between $\KMS_\beta(\Oo_{\Gamma})$ and probability measures supported on $p_1(B_\Gamma) = \{1/4, 1/2\}$.
\end{example}

\begin{example}\label{ex:four_maps}
	Consider the contractive iterated function system $([0,1],\Gamma = \{\gamma_1,\gamma_2,\gamma_3,\gamma_4\})$ defined by
		\begin{align*}
			\gamma_1(x) &=  \frac{x}{2}, &
			\gamma_2(x) &= 1- \frac{x}{2},\\ 
			\gamma_3(x) &= 
			\begin{cases}
				\frac{x}{4} + \frac{1}{8} &\text{if } x \le 1/2\\
				\frac{x}{2} & \text{if } x > 1/2
			\end{cases} &
		\gamma_4(x) &= 
		\begin{cases}
			1 - \frac{x}{2} & \text{if } x \le 1/2\\
			\frac{7}{8} - \frac{x}{4} &\text{if } x > 1/2.
		\end{cases}
		\end{align*}
The graph of $\Gamma$ is 
\begin{center}
 \begin{tikzpicture}[vertex/.style={circle, fill=red, inner sep=1.5pt},
scale=4.5]
 	\begin{scope}[font=\scriptsize]
 		
 		\draw [thick] (0,0) node [below left] {$0$} -- (1,0);
 		\draw [thick] (0,0) -- (0,1);
 		
 		\draw[thick] (1,0.02) -- (1,-0.02) node [below]  {$1$};
 		\draw[thick] (0.5,0.02) -- (0.5,-0.02) node [below]  {$1/2$};
 		\draw[thick] (0.02,1) -- (-0.02,1) node [left]  {$1$};
 		\draw[thick] (0.02,0.5) -- (-0.02,0.5) node [left]  {$1/2$};
		\draw[thick] (0.02,0.75) -- (-0.02,0.75) node [left]  {$3/4$};
 		\draw[thick] (0.02,0.25) -- (-0.02,0.25) node [left]  {$1/4$};
 		
 		\draw (0,0) --  (0.5,0.25) -- (1,0.5) --  (0,1) (0,0.125) -- (0.5,0.25) (0.5,0.75) -- (1,0.625);
 		
		\node[vertex] at (1,0.5) {};
		\node[vertex] at (0.5,0.25) {};
		\node[vertex] at (0.5,0.75) {};

 	\end{scope}
 \end{tikzpicture}
\end{center}
where the branch set $B_\Gamma = \{(1/4,1/2), (1/2,1), (3/4,1/2)\}$ is highlighted in red. A Markov partition is $\Mm = \{\{0\}, (0,1/2), \{1/2\}, (1/2,1), \{1\}\}$ and using Theorem~\ref{thm:finite_iff_Theta_const} one checks that it is minimal. It has a unique recurrent component corresponding to the intervals $M_2 = (0,1/2)$ and $M_4 = (1/2,1)$, and $\AAA_R = \begin{psmallmatrix} 2 & 1 \\ 1 & 2 \end{psmallmatrix}$. One checks that $\rho(\AAA_E) = \rho(\AAA_R) = 3$. So $\Gamma$ behaves like a system of $3$ maps rather than $4$.

By Theorem~\ref{thm:inf_type}, there is a unique $\KMS_{\ln(3)}$ state on $\Oo_{\Gamma}$, and it is of infinite type. For $\beta > \ln(3)$, Theorem~\ref{thm:finite_type_classification} gives a bijection between $\KMS_\beta(\Oo_{\Gamma})$ and probability measures supported on $p_1(B_\Gamma) = \{1/4, 1/2, 3/4\}$.
\end{example}

 The following example shows that even if $B_{\Gamma} = \varnothing$ it is not necessarily the case that $\beta_c = \ln|\Gamma|$ and $\mu_\infty = \mu_H$.

\begin{example} \label{ex:cantor_not_hutchinson}
Let $\Sigma_2 = \{0,1\}^\NN$ be the Cantor space with metric $d(x,y) = 2^{-\min\{n \ge 0  \mid x_n \ne y_n\}}$. For $i = 0,1$ let $\gamma_i  \colon \Sigma_2 \to \Sigma_2$ be given by $\gamma_i(x_0 x_1 x_2 \cdots) = i x_0 x_1 x_2 \cdots$. For each $i = 0,1$ let $\Zz(i) = \{x \in \Sigma_2 \mid x_0 = i\}$ and  let $\ol{i} = 1-i$. 
Define $\gamma_2 \colon \Sigma_2 \to \Sigma_2$ by
\[
\gamma_2(x) = \begin{cases} 0x & \text{if }x\in\Zz(0)\\  01\ol{x_1}\ol{x_2}\ol{x_3}\cdots & \text{if } x = 1x_1x_2x_3\cdots \in \Zz(1).\end{cases}
\]
Then $(\Sigma_2, \Gamma=\{\gamma_0,\gamma_1,\gamma_2\})$ is a contractive iterated function system. For $i \ne j$ we have $\gamma_i(x) = \gamma_j(x)$ if and only if $x \in \Zz(0)$ and $i,j \in \{0,2\}$. The set $G_{\{\gamma_0,\gamma_2\}} = \{(0y,y)  \mid y \in \Zz(0)\}$ is clopen in $\gr(\Gamma)$, so $B_\Gamma = \varnothing$. 
The fundamental partition is a Markov partition consisting of the sets $M_1 = \AA_{\{\{\gamma_0,\gamma_2\},\{\gamma_1\}\}} = \Zz(0)$ and $M_2 = \AA_{\{\{\gamma_0\},\{\gamma_1\},\{\gamma_2\}\}} = \Zz(1)$. Both vertices of the associated directed graph $E$ belong to the unique recurrent component, and $\AAA_E = \begin{psmallmatrix} 1 & 2 \\ 1 & 1 \end{psmallmatrix}$ has spectral radius $\rho = 1 + \sqrt{2}$ with corresponding positive unit-$1$-norm eigenvector $(2-\sqrt{2}, \sqrt{2}-1)$. 

Since $B_{\Gamma} = \varnothing$, the algebra $\Oo_\Gamma$ has a unique $\KMS$ state by Proposition~\ref{prop:unique_kms_no_branch}, and it occurs at inverse temperature $\beta_c = \ln(1 + \sqrt{2})$. The corresponding measure $\mu_\infty$ is not the Hutchinson measure $\mu_H$ since, for instance, $\mu_\infty(M_1) = 2-\sqrt{2} \ne 2/3 = \mu_H(M_1)$.
\end{example}

In every example so far, the branch set $B_\Gamma$ has been zero-dimensional. The following example---which arises from a 2-dimensional version of the tent map---shows that this is not always the case, even for systems satisfying the open set condition.
\begin{example}\label{ex:squarebranched}
Consider the iterated function system $(\AA \coloneqq [0,1]^2,\Gamma = \{\gamma_1,\gamma_2, \gamma_3,\gamma_4\})$, where
\begin{align*}
\gamma_1(s,t)       & = \left(\frac{s}{2},\, \frac{t}{2} \right),               & 
\quad \gamma_2(s,t) & = \left(1-\frac{s}{2},\, \frac{t}{2} \right),               \\
\quad \gamma_3(s,t) & = \left(\frac{s}{2},\,1- \frac{t}{2} \right), \text{ and} & 
\gamma_4(s,t)       & = \left(1-\frac{s}{2},\,1-\frac{t}{2}\right).               
\end{align*}
This system is contractive and satisfies the open set condition with witness $U \coloneqq (0,1)^2$. For $u \in [0,1) \times [0,1)$ we have $|p_2^{-1}(u)| = 4$, for $u\in (\{1\} \times [0,1)) \cup ([0,1) \times \{1\})$ we have $|p_2^{-1}(u)| = 2$, and $|p_2^{-1}(1,1)| = 1$. By Proposition~\ref{prop:open-set_branch}, $B_\Gamma = \{(x,y) \in \gr(\Gamma) \mid y \in (\{1\} \times [0,1]) \cup ([0,1] \times \{1\})\}$. 

Using dashed lines to represent the images of $B_\Gamma$ under the coordinate projections $p_1$ (blue) and $p_2$ (red), we can visualise $(\AA,\Gamma)$ by
\vspace{0.5em}
\begin{center}
		\tikzset{
			dot/.style = {circle, fill, minimum size=#1,
				inner sep=0pt, outer sep=0pt},
			dot/.default = 3pt 
		}
		\begin{tikzpicture}[scale=3]
		\begin{scope}[]
		
		\tikzstyle{thickarrow} = [line width=2mm, draw=black, -triangle 60, fill=white,postaction={draw, line width=4mm, shorten >=6mm, -}]
		
		\pgfmathsetmacro{\shift}{2}
		
		\draw[] (0,1) -- (0,0) -- (1,0);
		\draw[thick,red, densely dashdotted] (1,0) node[dot,fill = red] {} -- (1,1) -- (0,1) node[dot,fill = red, dashed] {};

		\draw[->] (0.2,0.8) -- (0.8,0.8);
		
		\node (A) at (0.5,0.5){$\AA$};
		
		\draw[]  (\shift,0) -- (1+\shift,0) -- (1+\shift,1) -- (\shift,1) -- (\shift,0);
		\draw[-stealth,->] (0.2,0.8) -- (0.8,0.8);

		\draw[thick,blue,densely dashdotted] (\shift,0.5) node[dot,fill = blue] {} -- (\shift + 1, 0.5) node[dot,fill = blue] {};
		\draw[thick,blue,densely dashdotted] (\shift+0.5,0) node[dot,fill = blue] {} -- (\shift + 0.5, 1) node[dot,fill = blue] {};
		
		\draw[->] (\shift+0.1,0.4) -- (\shift+0.4,0.4);
		\draw[->] (\shift+0.1,0.6) -- (\shift+0.4,0.6);
		\draw[->] (\shift+0.9,0.4) -- (\shift+0.6,0.4);
		\draw[->] (\shift+0.9,0.6) -- (\shift+0.6,0.6);
		
		\node (A) at (\shift+0.25, 0.25){$\gamma_1(\AA)$};
		\node (A) at (\shift+0.75, 0.25){$\gamma_2(\AA)$};
		\node (A) at (\shift+0.25, 0.75){$\gamma_3(\AA)$};
		\node (A) at (\shift+0.75, 0.75){$\gamma_4(\AA)$};
		
		\draw[double, double distance=2pt, -Implies] (1.1,0.5) to[bend left] node[above=2pt]{$\Gamma$} (\shift-0.1,0.5);

		\end{scope}
		\end{tikzpicture}
\end{center}
where we orient the square $\AA$ with an arrow so that the orientation-changing behaviour of the maps is apparent. Its minimal Markov partition is $\Mm = \{M_i\}_{i=1}^9$, where $M_1 = U$, $M_2 = \{0\} \times (0,1)$, $M_3 = (0,1) \times \{0\}$, $M_4 = \{1\} \times (0,1)$, $M_5 = (0,1) \times \{1\}$, $M_6 = \{(0,0)\}$, $M_7 = \{(0,1)\}$, $M_8 = \{(1,0)\}$ and $M_9 = \{(1,1)\}$. 

We have $\beta_c = \ln(4)$ by Corollary~\ref{cor:OSC_Hutchinson}, and by Theorem~\ref{thm:inf_type} the unique $\KMS_{\beta_c}$ state restricts to
$\mu_\infty = \mu_H$ on $A$. 
By Theorem~\ref{thm:finite_type_classification} the finite-type $\KMS_\beta$ states for $\beta > \beta_c$ are parametrised by the simplex of probability measures on the cross $p_1(B_\Gamma)$. 
\end{example}

\end{document}